\newcommand{\tgreen}[1]{}
\newcommand{\tred}[1]{}
\newcommand{\tblue}[1]{}
\let\productionversion\relax
\let\withoutdoublestaruseless\relax
\let\withoutnextpaper\relax
\IfFileExists{algo-paper-reduced-production-defs.tex}
{\input{algo-paper-reduced-production-defs}}{}
\newif\ifserialtitle
\serialtitletrue

\newif\ifojmo\ojmofalse
\ifojmo
\documentclass[OJMO,NoWriteXVIII,Unicode]{cedram}
\usepackage{comment}

\DeclareMathOperator    \cl                     {cl}
\newcommand\ctscl{\clsemi\relax}
\newcommand\ContinuitySet{\texorpdfstring{A}{}}
\newcommand{\clsemi}[1]{\mathop{\mathrm{clsemi}_{\ContinuitySet}}\ifx#1\relax\else(#1)\fi}   
\newcommand\xmapstonormalsize[1]{\xmapsto{\mbox{\normalsize$\;#1\;$}}}

\DeclareMathOperator    \dom                     {dom}

\DeclareMathOperator    \im                     {im}
\DeclareMathOperator    \intr                   {int}

\DeclareMathOperator    \maxdom                   {Max}
\DeclareMathOperator    \relint         {rel\,int}

\DeclareMathOperator    \verts          {vert}

\DeclareMathOperator    \Aff {Aff}  
\DeclareMathOperator    \graph          {Gr}
\newcommand\Rf{R_{\texorpdfstring{f}{\unichar{"1D453}}}}
\newcommand{\semi}[1]{\mathop{\mathrm{i{}semi}}\ifx#1/\tred{error:use$\Gamma$}\else\ifx#1\relax\else(#1)\fi\fi}
\newcommand\xmapstosemi{\xmapstonormalsize{\semi\relax}}
\newcommand\xmapstojoinsemi{\xmapstonormalsize{\joinsemi\relax}}
\newcommand{\join}[1]{\ifx#1/\texorpdfstring{\vee}{\unichar{"1D5B}}\else\mathop{\mathrm{join}}(#1)\fi}
\newcommand{\IsJoin}[1]{{#1}^{\join/}}
\def\fin/{\texorpdfstring{\mathrm{fin}}{\unichar{"1DA0}\unichar{"2071}\unichar{"207F}}}    
\def\red/{\texorpdfstring{\mathrm{red}}{\unichar{"2B3}\unichar{"1D49}\unichar{"1D48}}}    
\newcommand{\restrict}[1]{\ifx#1/\texorpdfstring{\subseteq}{\unichar{"1D9C}\unichar{"331}%
  }\else\mathop{\mathrm{restrict}}(#1)\fi}
\newcommand{\extend}[1]{\ifx#1/\tred{ERROR: Bad macro use}\else\mathop{\mathrm{\text{\upshape2-}extend}_{\ContinuitySet}}\ifx#1\relax\else(#1)\fi\fi}
\newcommand{\IsExtend}[1]{\overline{#1}{}\texorpdfstring{}{\unichar{"305}}}
\newcommand{\joinextend}[1]{\mathop{\mathrm{extend}_{\ContinuitySet}}\ifx#1\relax\else(#1)\fi}
\newcommand{\IsJoinExtend}[1]{\IsExtend{#1}{}^{\join/}}
\newcommand\xmapstoextend{\xmapstonormalsize{\joinextend\relax}}
\newcommand{\IsInv}[1]{{#1}^{\mathrm{inv}}}
\newcommand{\joinsemi}[1]{\mathop{\mathrm{jsemi}}\ifx#1/\tred{error:use$\Gamma$}\else\ifx#1\relax\else(#1)\fi\fi}
\DeclareRobustCommand{\mergeplus}[1]{\ifx#1/\ifmmode\boxtimes\else kaleidoscopic\fi\else\mathop{\mathrm{kaleido}}(#1)\fi}
\newcommand{\IsLimit}[1]{\bar{#1}}
\newcommand{\limit}[1]{\mathop{\mathrm{lim}}\ifx#1/\else(#1)\fi}
\newcommand{\symlim}[1]{\mathop{\mathrm{symlim}}\ifx#1/\else(#1)\fi}
\newcommand{\arblim}[1]{\mathop{\mathrm{arblim}}\ifx#1/\else(#1)\fi}
\DeclareRobustCommand{\CC}{{\texorpdfstring{\BeginAccSupp{method=hex,unicode,ActualText=D835 D49E}\mathcal C\EndAccSupp{}}{\unichar{"1D49E}}}}
\newcommand{\Comp}[1][\tred{ERROR}]{C_{#1}}   
\newcommand{\Uncomp}[1][\tred{ERROR}]{U_{#1}} 
\newcommand{\moves}{\mathop{\mathrm{moves}}\nolimits}
\newcommand{\tMoves}{\moves\unisubscript+}
\newcommand{\rMoves}{\moves\unisubscript-}
\newcommand{\MovesOfGraph}{\moves}
\newcommand{\tMovesOfGraph}{\tMoves}
\newcommand{\rMovesOfGraph}{\rMoves}
\newcommand\joinmoves{\mathop{\mathrm{jmoves}}\nolimits}
\newcommand{\FullMoveGroup}{\Gamma(\R)}
\newcommand{\FullMoveSemigroup}{\Gamma^{\restrict/}(\R)} 

\providecommand\symcal[1]{\mathcal{#1}}
\providecommand{\R}{\bb R}

\providecommand{\Z}{\bb Z}
\providecommand{\N}{\bb N}
\providecommand{\C}{\bb C}
\newcommand{\T}{{\symcal T}}
\newcommand\tildePi{\tilde\Pi}
\newcommand\tildepi{\tilde\pi}
\newcommand\RZ{}            

\newcommand\UOOU{\Omega^0|_U{}}
\newcommand\UOOUk{\UOOU^{k}}

\newcommand\st{\mid}
\newcommand\bigst{\mathrel{\big|}}
\newcommand\Bigst{\mathrel{\Big|}}
\newcommand\biggst{\mathrel{\bigg|}}

\renewcommand{\C}{\symcal C}

\newcommand{\Omegainit}{\Omega^{\texorpdfstring{0}{\unichar{"2070}}}}
\newcommand\Omegaresp[1]{\Gamma^{\texorpdfstring{\mathrm{resp}}{\unichar{"2B3}\unichar{"1D49}\unichar{"2E2}\unichar{"1D56}}}\ifx#1/\else(#1)\fi} 

\newcommand{\setcond}[2]{\left\{\, #1 \st #2 \,\right\}}

\def\st{\mid}

\newcommand\InlineFrac[2]{#1/#2}  
\newcommand\ColVec[3][\relax]
{
  \ifx#1\relax
  \bgroup\let\frac=\InlineFrac\begin{psmallmatrixbig}#2\vphantom{/}\\#3\vphantom{/}\end{psmallmatrixbig}\egroup
  \else
  \bgroup\let\frac=\InlineFrac\begin{psmallmatrixbig}\ifx#200\else#2/#1\fi\\\ifx#300\else#3/#1\fi\end{psmallmatrixbig}\egroup
  \fi
}

\renewcommand{\P}{\mathcal{P}}
\newcommand{\B}{B}  

\newcommand\MOVEDIAG[1]{\vcenter{\hbox{\includegraphics[height=.25\linewidth]{graphics-for-algo-paper/#1}}}}
\newcommand\VOP[1]{\MOVEDIAG{move_#1}}
\newcommand\COMPOSE[2]{\VOP{#1} \circ \VOP{#2} &= \VOP{#1_o_#2}}

\newcommand\Bounded[1]{{\mathcal B({#1})}}
\newcommand\BoundedContinuous[1]{{\mathcal C_{\mathrm b}({#1})}}

\else
\documentclass[11pt,svgnames]{amsart}
\IfFileExists{mk-pdf-copypastable.sty}
{\usepackage{mk-pdf-copypastable}}
{\usepackage{../mk-pdf-copypastable}}
\usepackage[shortlabels]{enumitem}
\usepackage[type={CC},modifier={by-sa},version={4.0},imagewidth=2.5cm]{doclicense}

\numberwithin{equation}{section}

\usepackage{mathtools}
\usepackage{pdfpages}
\usepackage{booktabs}
\usepackage{amssymb}
\usepackage{xcolor}
\usepackage{graphicx}
\usepackage{upquote}
\usepackage{epstopdf}
\usepackage{url}
\usepackage[square,numbers,sort&compress]{natbib}

\usepackage{multirow}
\usepackage{bigdelim}

\usepackage{tikz}
\usepackage{pst-node}
\usepackage{tikz-cd}
\definecolor{mediumspringgreen}{rgb}{0.0, 0.98039215, 0.60392156}

\usepackage{pgfplots} 

\usetikzlibrary{positioning}
\usetikzlibrary{spy}
\usetikzlibrary{calc}

\def\visible<#1>{}  

\usepackage[utf8]{inputenc}

\usepackage[english]{babel}
\usepackage{amsfonts}
\usepackage{amsmath}
\usepackage{latexsym}
\usepackage{stmaryrd}
\usepackage{color}
\usepackage{subfigure}
\usepackage[shortlabels]{enumitem}      

\usepackage{setspace}

\usepackage[unicode,psdextra]{hyperref}
\usepackage[depth=3,open,openlevel=1]{bookmark}  

\usepackage{ifpdf}
\usepackage[norelsize,ruled,vlined]{algorithm2e}
\usepackage{comment}

\ifx\productionversion\undefined
\makeatletter
\renewcommand\section{\clearpage\@startsection{section}{2}%
  \z@{1.5\linespacing\@plus.7\linespacing}{1.5em}%
  {\centering \LARGE\bf Section }}%
\makeatother
\fi

\newtheorem{theorem}{Theorem}[section]

\makeatletter
\newcommand\MkNewTheorem[2]{%
  \newtheorem{#1}{#2}[section]
  \expandafter\def\csname c@#1\endcsname{\c@theorem}
  \expandafter\def\csname p@#1\endcsname{\p@theorem}
  \expandafter\def\csname the#1\endcsname{\thetheorem}
  \expandafter\def\csname #1name\endcsname{#2}
}

\MkNewTheorem{corollary}{Corollary}
\MkNewTheorem{lemma}{Lemma}
\MkNewTheorem{proposition}{Proposition}
\MkNewTheorem{prop}{Proposition}
\MkNewTheorem{claim}{Claim}
\MkNewTheorem{observation}{Observation}
\MkNewTheorem{obs}{Observation}
\MkNewTheorem{conjecture}{Conjecture}
\MkNewTheorem{openquestion}{Open question}

\theoremstyle{definition}
\MkNewTheorem{example}{Example}
\MkNewTheorem{exercise}{Exercise}
\MkNewTheorem{notation}{Notation}
\MkNewTheorem{assumption}{Assumption}
\MkNewTheorem{definition}{Definition}
\MkNewTheorem{remark}{Remark}
\MkNewTheorem{goal}{Goal}

\makeatletter
\let\OurMathBbAux=\symbb
\DeclareRobustCommand\OurMathBb{\OurMathBbAux}
\let\mathbb=\OurMathBb
\let\symbb=\OurMathBb
\let\bfseries=\undefined
\DeclareRobustCommand\bfseries
{\not@math@alphabet\bfseries\mathbf
  \boldmath\fontseries\bfdefault\selectfont\let\OurMathBbAux=\mathbf}
\def\@thm#1#2#3{%
  \ifhmode\unskip\unskip\par\fi
  \normalfont
  \trivlist
  \let\thmheadnl\relax
  \let\thm@swap\@gobble
  \thm@notefont{\fontseries\mddefault\upshape\unboldmath}
  \thm@headpunct{.}
  \thm@headsep 5\p@ plus\p@ minus\p@\relax
  \thm@space@setup
  #1
  \@topsep \thm@preskip               
  \@topsepadd \thm@postskip           
  \def\@tempa{#2}\ifx\@empty\@tempa
    \def\@tempa{\@oparg{\@begintheorem{#3}{}}[]}%
  \else
    \refstepcounter{#2}%
    \def\@tempa{\@oparg{\@begintheorem{#3}{\csname the#2\endcsname}}[]}%
  \fi
  \@tempa
}
\makeatother

\makeatletter
\renewcommand{\pod}[1]
{\allowbreak\mathchoice{\mkern18mu}{\mkern8mu}{\mkern8mu}{\mkern8mu}(#1)}
\makeatother

\usepackage[normalem]{ulem}

\fi

\ifx\withrefcheck\relax
\usepackage{refcheck}
\fi

\usepackage{tikz}
\DeclareRobustCommand\sage[1]{\texttt{#1}}

\DeclareRobustCommand\Myunderscore{\scalebox{0.66}{\textunderscore}}
\newcommand\underscore{\Myunderscore\allowbreak}
\let\_=\underscore

\newcommand\githubsearchurl{https://github.com/mkoeppe/cutgeneratingfunctionology/search}
\pgfkeyssetvalue{/sagefunc/bccz_counterexample}{\href{\githubsearchurl?q=\%22def+bccz_counterexample(\%22}{\sage{bccz\underscore{}counterexample}}}
\pgfkeyssetvalue{/sagefunc/bcdsp_arbitrary_slope}{\href{\githubsearchurl?q=\%22def+bcdsp_arbitrary_slope(\%22}{\sage{bcdsp\underscore{}arbitrary\underscore{}slope}}}
\pgfkeyssetvalue{/sagefunc/bhk_gmi_irrational}{\href{\githubsearchurl?q=\%22def+bhk_gmi_irrational(\%22}{\sage{bhk\underscore{}gmi\underscore{}irrational}}}
\pgfkeyssetvalue{/sagefunc/bhk_irrational}{\href{\githubsearchurl?q=\%22def+bhk_irrational(\%22}{\sage{bhk\underscore{}irrational}}}
\pgfkeyssetvalue{/sagefunc/bhk_slant_irrational}{\href{\githubsearchurl?q=\%22def+bhk_slant_irrational(\%22}{\sage{bhk\underscore{}slant\underscore{}irrational}}}
\pgfkeyssetvalue{/sagefunc/chen_4_slope}{\href{\githubsearchurl?q=\%22def+chen_4_slope(\%22}{\sage{chen\underscore{}4\underscore{}slope}}}
\pgfkeyssetvalue{/sagefunc/dg_2_step_mir}{\href{\githubsearchurl?q=\%22def+dg_2_step_mir(\%22}{\sage{dg\underscore{}2\underscore{}step\underscore{}mir}}}
\pgfkeyssetvalue{/sagefunc/dg_2_step_mir_limit}{\href{\githubsearchurl?q=\%22def+dg_2_step_mir_limit(\%22}{\sage{dg\underscore{}2\underscore{}step\underscore{}mir\underscore{}limit}}}
\pgfkeyssetvalue{/sagefunc/drlm_2_slope_limit}{\href{\githubsearchurl?q=\%22def+drlm_2_slope_limit(\%22}{\sage{drlm\underscore{}2\underscore{}slope\underscore{}limit}}}
\pgfkeyssetvalue{/sagefunc/drlm_3_slope_limit}{\href{\githubsearchurl?q=\%22def+drlm_3_slope_limit(\%22}{\sage{drlm\underscore{}3\underscore{}slope\underscore{}limit}}}
\pgfkeyssetvalue{/sagefunc/drlm_backward_3_slope}{\href{\githubsearchurl?q=\%22def+drlm_backward_3_slope(\%22}{\sage{drlm\underscore{}backward\underscore{}3\underscore{}slope}}}
\pgfkeyssetvalue{/sagefunc/gj_2_slope}{\href{\githubsearchurl?q=\%22def+gj_2_slope(\%22}{\sage{gj\underscore{}2\underscore{}slope}}}
\pgfkeyssetvalue{/sagefunc/gj_2_slope_repeat}{\href{\githubsearchurl?q=\%22def+gj_2_slope_repeat(\%22}{\sage{gj\underscore{}2\underscore{}slope\underscore{}repeat}}}
\pgfkeyssetvalue{/sagefunc/gj_forward_3_slope}{\href{\githubsearchurl?q=\%22def+gj_forward_3_slope(\%22}{\sage{gj\underscore{}forward\underscore{}3\underscore{}slope}}}
\pgfkeyssetvalue{/sagefunc/gmic}{\href{\githubsearchurl?q=\%22def+gmic(\%22}{\sage{gmic}}}
\pgfkeyssetvalue{/sagefunc/hildebrand_2_sided_discont_1_slope_1}{\href{\githubsearchurl?q=\%22def+hildebrand_2_sided_discont_1_slope_1(\%22}{\sage{hildebrand\underscore{}2\underscore{}sided\underscore{}discont\underscore{}1\underscore{}slope\underscore{}1}}}
\pgfkeyssetvalue{/sagefunc/hildebrand_2_sided_discont_2_slope_1}{\href{\githubsearchurl?q=\%22def+hildebrand_2_sided_discont_2_slope_1(\%22}{\sage{hildebrand\underscore{}2\underscore{}sided\underscore{}discont\underscore{}2\underscore{}slope\underscore{}1}}}
\pgfkeyssetvalue{/sagefunc/hildebrand_5_slope_22_1}{\href{\githubsearchurl?q=\%22def+hildebrand_5_slope_22_1(\%22}{\sage{hildebrand\underscore{}5\underscore{}slope\underscore{}22\underscore{}1}}}
\pgfkeyssetvalue{/sagefunc/hildebrand_5_slope_24_1}{\href{\githubsearchurl?q=\%22def+hildebrand_5_slope_24_1(\%22}{\sage{hildebrand\underscore{}5\underscore{}slope\underscore{}24\underscore{}1}}}
\pgfkeyssetvalue{/sagefunc/hildebrand_5_slope_28_1}{\href{\githubsearchurl?q=\%22def+hildebrand_5_slope_28_1(\%22}{\sage{hildebrand\underscore{}5\underscore{}slope\underscore{}28\underscore{}1}}}
\pgfkeyssetvalue{/sagefunc/hildebrand_discont_3_slope_1}{\href{\githubsearchurl?q=\%22def+hildebrand_discont_3_slope_1(\%22}{\sage{hildebrand\underscore{}discont\underscore{}3\underscore{}slope\underscore{}1}}}
\pgfkeyssetvalue{/sagefunc/kf_n_step_mir}{\href{\githubsearchurl?q=\%22def+kf_n_step_mir(\%22}{\sage{kf\underscore{}n\underscore{}step\underscore{}mir}}}
\pgfkeyssetvalue{/sagefunc/kzh_10_slope_1}{\href{\githubsearchurl?q=\%22def+kzh_10_slope_1(\%22}{\sage{kzh\underscore{}10\underscore{}slope\underscore{}1}}}
\pgfkeyssetvalue{/sagefunc/kzh_28_slope_1}{\href{\githubsearchurl?q=\%22def+kzh_28_slope_1(\%22}{\sage{kzh\underscore{}28\underscore{}slope\underscore{}1}}}
\pgfkeyssetvalue{/sagefunc/kzh_28_slope_2}{\href{\githubsearchurl?q=\%22def+kzh_28_slope_2(\%22}{\sage{kzh\underscore{}28\underscore{}slope\underscore{}2}}}
\pgfkeyssetvalue{/sagefunc/kzh_3_slope_param_extreme_1}{\href{\githubsearchurl?q=\%22def+kzh_3_slope_param_extreme_1(\%22}{\sage{kzh\underscore{}3\underscore{}slope\underscore{}param\underscore{}extreme\underscore{}1}}}
\pgfkeyssetvalue{/sagefunc/kzh_3_slope_param_extreme_2}{\href{\githubsearchurl?q=\%22def+kzh_3_slope_param_extreme_2(\%22}{\sage{kzh\underscore{}3\underscore{}slope\underscore{}param\underscore{}extreme\underscore{}2}}}
\pgfkeyssetvalue{/sagefunc/kzh_4_slope_param_extreme_1}{\href{\githubsearchurl?q=\%22def+kzh_4_slope_param_extreme_1(\%22}{\sage{kzh\underscore{}4\underscore{}slope\underscore{}param\underscore{}extreme\underscore{}1}}}
\pgfkeyssetvalue{/sagefunc/kzh_5_slope_fulldim_1}{\href{\githubsearchurl?q=\%22def+kzh_5_slope_fulldim_1(\%22}{\sage{kzh\underscore{}5\underscore{}slope\underscore{}fulldim\underscore{}1}}}
\pgfkeyssetvalue{/sagefunc/kzh_5_slope_fulldim_2}{\href{\githubsearchurl?q=\%22def+kzh_5_slope_fulldim_2(\%22}{\sage{kzh\underscore{}5\underscore{}slope\underscore{}fulldim\underscore{}2}}}
\pgfkeyssetvalue{/sagefunc/kzh_5_slope_fulldim_3}{\href{\githubsearchurl?q=\%22def+kzh_5_slope_fulldim_3(\%22}{\sage{kzh\underscore{}5\underscore{}slope\underscore{}fulldim\underscore{}3}}}
\pgfkeyssetvalue{/sagefunc/kzh_5_slope_fulldim_4}{\href{\githubsearchurl?q=\%22def+kzh_5_slope_fulldim_4(\%22}{\sage{kzh\underscore{}5\underscore{}slope\underscore{}fulldim\underscore{}4}}}
\pgfkeyssetvalue{/sagefunc/kzh_5_slope_fulldim_5}{\href{\githubsearchurl?q=\%22def+kzh_5_slope_fulldim_5(\%22}{\sage{kzh\underscore{}5\underscore{}slope\underscore{}fulldim\underscore{}5}}}
\pgfkeyssetvalue{/sagefunc/kzh_5_slope_fulldim_covers_1}{\href{\githubsearchurl?q=\%22def+kzh_5_slope_fulldim_covers_1(\%22}{\sage{kzh\underscore{}5\underscore{}slope\underscore{}fulldim\underscore{}covers\underscore{}1}}}
\pgfkeyssetvalue{/sagefunc/kzh_5_slope_fulldim_covers_2}{\href{\githubsearchurl?q=\%22def+kzh_5_slope_fulldim_covers_2(\%22}{\sage{kzh\underscore{}5\underscore{}slope\underscore{}fulldim\underscore{}covers\underscore{}2}}}
\pgfkeyssetvalue{/sagefunc/kzh_5_slope_fulldim_covers_3}{\href{\githubsearchurl?q=\%22def+kzh_5_slope_fulldim_covers_3(\%22}{\sage{kzh\underscore{}5\underscore{}slope\underscore{}fulldim\underscore{}covers\underscore{}3}}}
\pgfkeyssetvalue{/sagefunc/kzh_5_slope_fulldim_covers_4}{\href{\githubsearchurl?q=\%22def+kzh_5_slope_fulldim_covers_4(\%22}{\sage{kzh\underscore{}5\underscore{}slope\underscore{}fulldim\underscore{}covers\underscore{}4}}}
\pgfkeyssetvalue{/sagefunc/kzh_5_slope_fulldim_covers_5}{\href{\githubsearchurl?q=\%22def+kzh_5_slope_fulldim_covers_5(\%22}{\sage{kzh\underscore{}5\underscore{}slope\underscore{}fulldim\underscore{}covers\underscore{}5}}}
\pgfkeyssetvalue{/sagefunc/kzh_5_slope_fulldim_covers_6}{\href{\githubsearchurl?q=\%22def+kzh_5_slope_fulldim_covers_6(\%22}{\sage{kzh\underscore{}5\underscore{}slope\underscore{}fulldim\underscore{}covers\underscore{}6}}}
\pgfkeyssetvalue{/sagefunc/kzh_5_slope_q22_f10_1}{\href{\githubsearchurl?q=\%22def+kzh_5_slope_q22_f10_1(\%22}{\sage{kzh\underscore{}5\underscore{}slope\underscore{}q22\underscore{}f10\underscore{}1}}}
\pgfkeyssetvalue{/sagefunc/kzh_5_slope_q22_f10_2}{\href{\githubsearchurl?q=\%22def+kzh_5_slope_q22_f10_2(\%22}{\sage{kzh\underscore{}5\underscore{}slope\underscore{}q22\underscore{}f10\underscore{}2}}}
\pgfkeyssetvalue{/sagefunc/kzh_5_slope_q22_f10_3}{\href{\githubsearchurl?q=\%22def+kzh_5_slope_q22_f10_3(\%22}{\sage{kzh\underscore{}5\underscore{}slope\underscore{}q22\underscore{}f10\underscore{}3}}}
\pgfkeyssetvalue{/sagefunc/kzh_5_slope_q22_f10_4}{\href{\githubsearchurl?q=\%22def+kzh_5_slope_q22_f10_4(\%22}{\sage{kzh\underscore{}5\underscore{}slope\underscore{}q22\underscore{}f10\underscore{}4}}}
\pgfkeyssetvalue{/sagefunc/kzh_5_slope_q22_f2_1}{\href{\githubsearchurl?q=\%22def+kzh_5_slope_q22_f2_1(\%22}{\sage{kzh\underscore{}5\underscore{}slope\underscore{}q22\underscore{}f2\underscore{}1}}}
\pgfkeyssetvalue{/sagefunc/kzh_6_slope_1}{\href{\githubsearchurl?q=\%22def+kzh_6_slope_1(\%22}{\sage{kzh\underscore{}6\underscore{}slope\underscore{}1}}}
\pgfkeyssetvalue{/sagefunc/kzh_6_slope_fulldim_covers_1}{\href{\githubsearchurl?q=\%22def+kzh_6_slope_fulldim_covers_1(\%22}{\sage{kzh\underscore{}6\underscore{}slope\underscore{}fulldim\underscore{}covers\underscore{}1}}}
\pgfkeyssetvalue{/sagefunc/kzh_6_slope_fulldim_covers_2}{\href{\githubsearchurl?q=\%22def+kzh_6_slope_fulldim_covers_2(\%22}{\sage{kzh\underscore{}6\underscore{}slope\underscore{}fulldim\underscore{}covers\underscore{}2}}}
\pgfkeyssetvalue{/sagefunc/kzh_6_slope_fulldim_covers_3}{\href{\githubsearchurl?q=\%22def+kzh_6_slope_fulldim_covers_3(\%22}{\sage{kzh\underscore{}6\underscore{}slope\underscore{}fulldim\underscore{}covers\underscore{}3}}}
\pgfkeyssetvalue{/sagefunc/kzh_6_slope_fulldim_covers_4}{\href{\githubsearchurl?q=\%22def+kzh_6_slope_fulldim_covers_4(\%22}{\sage{kzh\underscore{}6\underscore{}slope\underscore{}fulldim\underscore{}covers\underscore{}4}}}
\pgfkeyssetvalue{/sagefunc/kzh_6_slope_fulldim_covers_5}{\href{\githubsearchurl?q=\%22def+kzh_6_slope_fulldim_covers_5(\%22}{\sage{kzh\underscore{}6\underscore{}slope\underscore{}fulldim\underscore{}covers\underscore{}5}}}
\pgfkeyssetvalue{/sagefunc/kzh_7_slope_1}{\href{\githubsearchurl?q=\%22def+kzh_7_slope_1(\%22}{\sage{kzh\underscore{}7\underscore{}slope\underscore{}1}}}
\pgfkeyssetvalue{/sagefunc/kzh_7_slope_2}{\href{\githubsearchurl?q=\%22def+kzh_7_slope_2(\%22}{\sage{kzh\underscore{}7\underscore{}slope\underscore{}2}}}
\pgfkeyssetvalue{/sagefunc/kzh_7_slope_3}{\href{\githubsearchurl?q=\%22def+kzh_7_slope_3(\%22}{\sage{kzh\underscore{}7\underscore{}slope\underscore{}3}}}
\pgfkeyssetvalue{/sagefunc/kzh_7_slope_4}{\href{\githubsearchurl?q=\%22def+kzh_7_slope_4(\%22}{\sage{kzh\underscore{}7\underscore{}slope\underscore{}4}}}
\pgfkeyssetvalue{/sagefunc/ll_strong_fractional}{\href{\githubsearchurl?q=\%22def+ll_strong_fractional(\%22}{\sage{ll\underscore{}strong\underscore{}fractional}}}
\pgfkeyssetvalue{/sagefunc/mlr_cpl3_a_2_slope}{\href{\githubsearchurl?q=\%22def+mlr_cpl3_a_2_slope(\%22}{\sage{mlr\underscore{}cpl3\underscore{}a\underscore{}2\underscore{}slope}}}
\pgfkeyssetvalue{/sagefunc/mlr_cpl3_b_3_slope}{\href{\githubsearchurl?q=\%22def+mlr_cpl3_b_3_slope(\%22}{\sage{mlr\underscore{}cpl3\underscore{}b\underscore{}3\underscore{}slope}}}
\pgfkeyssetvalue{/sagefunc/mlr_cpl3_c_3_slope}{\href{\githubsearchurl?q=\%22def+mlr_cpl3_c_3_slope(\%22}{\sage{mlr\underscore{}cpl3\underscore{}c\underscore{}3\underscore{}slope}}}
\pgfkeyssetvalue{/sagefunc/mlr_cpl3_d_3_slope}{\href{\githubsearchurl?q=\%22def+mlr_cpl3_d_3_slope(\%22}{\sage{mlr\underscore{}cpl3\underscore{}d\underscore{}3\underscore{}slope}}}
\pgfkeyssetvalue{/sagefunc/mlr_cpl3_f_2_or_3_slope}{\href{\githubsearchurl?q=\%22def+mlr_cpl3_f_2_or_3_slope(\%22}{\sage{mlr\underscore{}cpl3\underscore{}f\underscore{}2\underscore{}or\underscore{}3\underscore{}slope}}}
\pgfkeyssetvalue{/sagefunc/mlr_cpl3_g_3_slope}{\href{\githubsearchurl?q=\%22def+mlr_cpl3_g_3_slope(\%22}{\sage{mlr\underscore{}cpl3\underscore{}g\underscore{}3\underscore{}slope}}}
\pgfkeyssetvalue{/sagefunc/mlr_cpl3_h_2_slope}{\href{\githubsearchurl?q=\%22def+mlr_cpl3_h_2_slope(\%22}{\sage{mlr\underscore{}cpl3\underscore{}h\underscore{}2\underscore{}slope}}}
\pgfkeyssetvalue{/sagefunc/mlr_cpl3_k_2_slope}{\href{\githubsearchurl?q=\%22def+mlr_cpl3_k_2_slope(\%22}{\sage{mlr\underscore{}cpl3\underscore{}k\underscore{}2\underscore{}slope}}}
\pgfkeyssetvalue{/sagefunc/mlr_cpl3_l_2_slope}{\href{\githubsearchurl?q=\%22def+mlr_cpl3_l_2_slope(\%22}{\sage{mlr\underscore{}cpl3\underscore{}l\underscore{}2\underscore{}slope}}}
\pgfkeyssetvalue{/sagefunc/mlr_cpl3_n_3_slope}{\href{\githubsearchurl?q=\%22def+mlr_cpl3_n_3_slope(\%22}{\sage{mlr\underscore{}cpl3\underscore{}n\underscore{}3\underscore{}slope}}}
\pgfkeyssetvalue{/sagefunc/mlr_cpl3_o_2_slope}{\href{\githubsearchurl?q=\%22def+mlr_cpl3_o_2_slope(\%22}{\sage{mlr\underscore{}cpl3\underscore{}o\underscore{}2\underscore{}slope}}}
\pgfkeyssetvalue{/sagefunc/mlr_cpl3_p_2_slope}{\href{\githubsearchurl?q=\%22def+mlr_cpl3_p_2_slope(\%22}{\sage{mlr\underscore{}cpl3\underscore{}p\underscore{}2\underscore{}slope}}}
\pgfkeyssetvalue{/sagefunc/mlr_cpl3_q_2_slope}{\href{\githubsearchurl?q=\%22def+mlr_cpl3_q_2_slope(\%22}{\sage{mlr\underscore{}cpl3\underscore{}q\underscore{}2\underscore{}slope}}}
\pgfkeyssetvalue{/sagefunc/mlr_cpl3_r_2_slope}{\href{\githubsearchurl?q=\%22def+mlr_cpl3_r_2_slope(\%22}{\sage{mlr\underscore{}cpl3\underscore{}r\underscore{}2\underscore{}slope}}}
\pgfkeyssetvalue{/sagefunc/rlm_dpl1_extreme_3a}{\href{\githubsearchurl?q=\%22def+rlm_dpl1_extreme_3a(\%22}{\sage{rlm\underscore{}dpl1\underscore{}extreme\underscore{}3a}}}
\pgfkeyssetvalue{/sagefunc/automorphism}{\href{\githubsearchurl?q=\%22def+automorphism(\%22}{\sage{automorphism}}}
\pgfkeyssetvalue{/sagefunc/multiplicative_homomorphism}{\href{\githubsearchurl?q=\%22def+multiplicative_homomorphism(\%22}{\sage{multiplicative\underscore{}homomorphism}}}
\pgfkeyssetvalue{/sagefunc/projected_sequential_merge}{\href{\githubsearchurl?q=\%22def+projected_sequential_merge(\%22}{\sage{projected\underscore{}sequential\underscore{}merge}}}
\pgfkeyssetvalue{/sagefunc/restrict_to_finite_group}{\href{\githubsearchurl?q=\%22def+restrict_to_finite_group(\%22}{\sage{restrict\underscore{}to\underscore{}finite\underscore{}group}}}
\pgfkeyssetvalue{/sagefunc/interpolate_to_infinite_group}{\href{\githubsearchurl?q=\%22def+interpolate_to_infinite_group(\%22}{\sage{interpolate\underscore{}to\underscore{}infinite\underscore{}group}}}
\pgfkeyssetvalue{/sagefunc/two_slope_fill_in}{\href{\githubsearchurl?q=\%22def+two_slope_fill_in(\%22}{\sage{two\underscore{}slope\underscore{}fill\underscore{}in}}}
\pgfkeyssetvalue{/sagefunc/generate_example_e_for_psi_n}{\href{\githubsearchurl?q=\%22def+generate_example_e_for_psi_n(\%22}{\sage{generate\underscore{}example\underscore{}e\underscore{}for\underscore{}psi\underscore{}n}}}
\pgfkeyssetvalue{/sagefunc/chen_3_slope_not_extreme}{\href{\githubsearchurl?q=\%22def+chen_3_slope_not_extreme(\%22}{\sage{chen\underscore{}3\underscore{}slope\underscore{}not\underscore{}extreme}}}
\pgfkeyssetvalue{/sagefunc/psi_n_in_bccz_counterexample_construction}{\href{\githubsearchurl?q=\%22def+psi_n_in_bccz_counterexample_construction(\%22}{\sage{psi\underscore{}n\underscore{}in\underscore{}bccz\underscore{}counterexample\underscore{}construction}}}
\pgfkeyssetvalue{/sagefunc/gomory_fractional}{\href{\githubsearchurl?q=\%22def+gomory_fractional(\%22}{\sage{gomory\underscore{}fractional}}}
\pgfkeyssetvalue{/sagefunc/not_minimal_2}{\href{\githubsearchurl?q=\%22def+not_minimal_2(\%22}{\sage{not\underscore{}minimal\underscore{}2}}}
\pgfkeyssetvalue{/sagefunc/not_extreme_1}{\href{\githubsearchurl?q=\%22def+not_extreme_1(\%22}{\sage{not\underscore{}extreme\underscore{}1}}}
\pgfkeyssetvalue{/sagefunc/kzh_2q_example_1}{\href{\githubsearchurl?q=\%22def+kzh_2q_example_1(\%22}{\sage{kzh\underscore{}2q\underscore{}example\underscore{}1}}}
\pgfkeyssetvalue{/sagefunc/zhou_two_sided_discontinuous_cannot_assume_any_continuity}{\href{\githubsearchurl?q=\%22def+zhou_two_sided_discontinuous_cannot_assume_any_continuity(\%22}{\sage{zhou\underscore{}two\underscore{}sided\underscore{}discontinuous\underscore{}cannot\underscore{}assume\underscore{}any\underscore{}continuity}}}
\pgfkeyssetvalue{/sagefunc/extremality_test}{\href{\githubsearchurl?q=\%22def+extremality_test(\%22}{\sage{extremality\underscore{}test}}}
\pgfkeyssetvalue{/sagefunc/plot_2d_diagram}{\href{\githubsearchurl?q=\%22def+plot_2d_diagram(\%22}{\sage{plot\underscore{}2d\underscore{}diagram}}}
\pgfkeyssetvalue{/sagefunc/nice_field_values}{\href{\githubsearchurl?q=\%22def+nice_field_values(\%22}{\sage{nice\underscore{}field\underscore{}values}}}
\pgfkeyssetvalue{/sagefunc/ParametricRealFieldElement}{\href{\githubsearchurl?q=\%22def+ParametricRealFieldElement(\%22}{\sage{ParametricRealFieldElement}}}
\pgfkeyssetvalue{/sagefunc/ParametricRealField}{\href{\githubsearchurl?q=\%22def+ParametricRealField(\%22}{\sage{ParametricRealField}}}
\pgfkeyssetvalue{/sagefunc/kzh_minimal_has_only_crazy_perturbation_1}{\href{\githubsearchurl?q=\%22def+kzh_minimal_has_only_crazy_perturbation_1(\%22}{\sage{kzh\underscore{}minimal\underscore{}has\underscore{}only\underscore{}crazy\underscore{}perturbation\underscore{}1}}}
\pgfkeyssetvalue{/sagefunc/bcds_discontinuous_everywhere}{\href{\githubsearchurl?q=\%22def+bcds_discontinuous_everywhere(\%22}{\sage{bcds\underscore{}discontinuous\underscore{}everywhere}}}

\DeclareRobustCommand\sage[1]{\textsf{\upshape #1}}
\DeclareRobustCommand\sagefunc[1]{\pgfkeys{/sagefunc/#1}}
\DeclareRobustCommand\sagefuncgraph[1]{\raisebox{-0.08ex}{\includegraphics[height=2ex,width=2.5em]{funcgraphs/#1}}}
\DeclareRobustCommand\sagefuncwithgraph[1]{\sagefunc{#1} \sagefuncgraph{#1}}

\DeclareRobustCommand\sagefuncwithgraphgomoryfractional{\sagefunc{gomory_fractional}\ \smash{\raisebox{-0.08ex}{\includegraphics[height=2.65ex,width=2.5em]{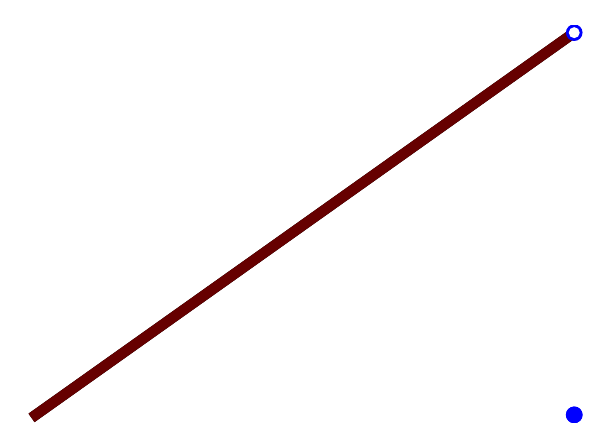}}}}

\providecommand\ISBN{ISBN }

\ifserialtitle
\title[Equivariant Perturbation VII]{Equivariant Perturbation in \\Gomory and Johnson's Infinite Group
  Problem\\[1ex] VII. Inverse Semigroup Theory, Closures, Decomposition of Perturbations}
\else
\title[On perturbation spaces of minimal valid functions]{On perturbation spaces of minimal valid functions:
  Inverse semigroup theory and equivariant decomposition theorem}
\fi

\author{Robert Hildebrand}
\address{\ifojmo\else Robert Hildebrand: \fi Grado Dept.\ of Industrial and Systems Engineering,
  Virginia Tech}
\email{rhil@vt.edu}

\author{Matthias K\"oppe}
\address{\ifojmo\else Matthias K\"oppe: \fi Dept.\ of Mathematics, University of California, Davis}
\email{mkoeppe@math.ucdavis.edu}

\author{Yuan Zhou}
\address{\ifojmo\else Yuan Zhou: \fi Dept.\ of Mathematics, University of Kentucky}
\email{yuan.zhou@uky.edu}

\thanks{%
  An extended abstract of 13 pages %
  \ifserialtitle{}titled \emph{On perturbation spaces of minimal valid functions:
    Inverse semigroup theory and equivariant decomposition theorem}
  \else
  with the same title
  \fi
  appeared in:
  Integer Programming and Combinatorial Optimization. IPCO 2019 (A.~Lodi and V.~Nagarajan, eds.),
  Lecture Notes in Computer Science, vol.~11480, Springer, Cham, 2019,
  pp.~247--260, \href {https://doi.org/10.1007/978-3-030-17953-3_19}
  {\path{https://doi.org/10.1007/978-3-030-17953-3_19}},
  \ISBN{978-3-030-17952-6} \cite{hildebrand-koeppe-zhou:algo-paper-abstract-ipco}.
  A~preliminary version of parts of the development in this paper appeared in the third
  author's 2017 Ph.D.~thesis~\cite{zhou:dissertation}.
  The authors gratefully acknowledge partial support from the National
  Science Foundation through grants DMS-0914873 (R.~Hildebrand, M.~K\"oppe)
  and DMS-1320051 (M.~K\"oppe, Y.~Zhou).  Part of this work was done while
  R.~Hildebrand and M.~K\"oppe were visiting the Simons Institute for the
  Theory of Computing in Fall 2017. It was partially supported by the
  DIMACS/Simons Collaboration on Bridging Continuous and Discrete Optimization
  through NSF grant CCF-1740425.}

\date{$\relax$Revision: 3403 $ - \ $Date: 2020-01-03 11:25:49 -0500 (Fri, 03 Jan 2020) $ $\!\!\!}

\newcommand{\old}[1]{{}}

\providecommand{\tgreen}[1]{\textsf{\textcolor {ForestGreen} {#1}}}
\providecommand{\tred}[1]{\texttt{\textcolor {red} {#1}}}
\providecommand{\tblue}[1]{\textcolor {blue} {#1}}

\ifx\withoutdoublestaruseless\undefined
\newenvironment{doublestaruseless}{\bgroup\ifmmode\else\small\fi\color{brown}}{\egroup}
\newcommand\useless[1]{\doublestaruseless #1\enddoublestaruseless}
\else
\excludecomment{doublestaruseless}
\newcommand\useless[1]{}
\fi

\ifx\withoutnextpaper\undefined

\else
\excludecomment{nextpaper}
\fi

\begin{document}

\begin{abstract}
\ifserialtitle
  In this self-contained paper, we
\else
  We 
\fi
  present a theory of the piecewise linear
  minimal valid functions for the 1-row Gomory--Johnson infinite group
  problem.  The non-extreme minimal valid functions are those that admit
  effective perturbations.  We give a precise description of the space of
  these perturbations as a direct sum of certain finite- and
  infinite-dimensional subspaces.  The infinite-dimensional subspaces have
  partial symmetries; to describe them, we develop a theory of inverse
  semigroups of partial bijections, interacting with the functional equations
  satisfied by the perturbations.  Our paper provides the foundation for
  grid-free algorithms for the Gomory--Johnson model, in particular for
  testing extremality of piecewise linear functions whose breakpoints are
  rational numbers with huge denominators.
\end{abstract}

\maketitle
\insert\footins{
  \normalfont\footnotesize
  \interlinepenalty\interfootnotelinepenalty
  \splittopskip\footnotesep \splitmaxdepth \dp\strutbox
  \floatingpenalty10000 \hsize\columnwidth
  \doclicenseThis\par}

\section{Introduction}

\subsection{Finite group relaxations $\Rf(P, \Z)$ of integer programs and hierarchies of
  valid inequalities
}
A powerful method to derive cutting planes for unstructured integer linear
optimization problems is to study relaxations with more structure and
convenient properties.  The pioneering relaxation in this line of research on
general-purpose cutting planes is
Gomory's \emph{finite group relaxation} \cite{gom}, whose convex hull is known
as the \emph{corner polyhedron}.  

The relaxations are structured around the
simplex method, applied to the continuous relaxation, and are therefore
suitable for generating cuts in a linear-programming-based cutting-plane
procedure.  The group relaxation is obtained by forgetting about the
nonnegativity of all basic variables, retaining only their integrality.
Viewed in the space of nonbasic variables, the equations of the simplex
tableau are replaced by congruences modulo the abelian group ($\Z$-module)
generated by the columns of the basis matrix.  Quotienting out by this group,
one obtains a ``group equation,'' which gives the relaxation its name.
Further relaxations are obtained by picking just one or a few rows of the
system, or more generally by condensing the system by means of group
homomorphisms; see \cite{gom} for its remarks on the use of (additive) group
characters.

In the present paper, we restrict ourselves to 1-row (``cyclic'') group
relaxations, which after aggregation of non-basic variables with identical
coefficients can be brought to the form
\begin{gather}
  \sum_{p \in P} y(p) \, p \in f + \Z  \label{GP-nonmaster} \\
  y(p) \in \Z_+ \ \ \textrm{for all $p \in P$}  \notag
\end{gather}
where $P$ is a finite subset of an additive group $G = \frac1q\Z \supset \Z$
and $f \in G\setminus\Z$, so $f + \Z$ is a coset of the subgroup $\Z$ in $G$.
This is called Gomory's \emph{finite (cyclic) group problem}.
We denote the convex hull of $y$ by $\Rf(P, \Z)$; it is a polyhedron of blocking
type. Therefore every nontrivial valid linear inequality can be written in the form
$\langle \pi, y\rangle := \sum_{p\in P} \pi(p) y(p) \geq 1$; then we call $\pi$
a \emph{valid function}.  If $\pi' \leq \pi$ are two valid functions, then the
valid inequality $\langle \pi, y\rangle \geq 1$ is a conic combination of
$\langle \pi', y\rangle \geq 1$ and nonnegativity inequalities $y(p) \geq 0$. Thus it suffices
to consider the \emph{minimal (valid) functions}~$\pi$, defined by the
property
\begin{equation}\tag{M}
  \begin{aligned}[t]
    \text{if $\pi'$ is valid and $\pi' \leq \pi$} \\
    \text{then $\pi' = \pi$.}
  \end{aligned}
  \label{eq:minimal-definition}
\end{equation}
A stronger notion is that of \emph{extreme functions} $\pi$, defined by the property
\begin{equation}\tag{E}
   \begin{aligned}[t]
     \text{if $\pi^+$ and $\pi^-$ are minimal and $\pi = \tfrac12 (\pi^++\pi^-)$}\\
     \text{then $\pi = \pi^+ = \pi^-$}. 
     \label{eq:facet-definition-like-extreme-function}
   \end{aligned}
\end{equation}
Extreme functions correspond to facet-defining inequalities for $\Rf(P,\Z)$.
Following the traditions of polyhedral combinatorics, we are interested in
describing families of extreme functions and making them available for
cutting-plane algorithms.

\subsection{Master problems $\Rf(G, \Z)$ and the subadditive characterization of minimal functions}
\label{s:intro:subadditive}
Gomory's approach was to consider \emph{master problems} for this purpose.
The sets of solutions $y$ to 1-row group relaxations $\Rf(P)$ for subsets $P$ of the
same group~$G$ inject into the \emph{master group relaxation}
\begin{gather}
  \sum_{p \in G} y(p) \, p \in f + \Z  \label{GP} \\
  y\colon G\to \Z_+ \textrm{ has finite support} \notag
\end{gather}
by setting $y(p) = 0$ for $p\notin P$.  We denote its convex hull by
$\Rf(G, \Z)$.  This is an infinite-dimensional set. By Gomory's master theorem
\cite[Theorem 13]{gom}, every extreme function $\pi'$ for $\Rf(P, \Z)$ is
obtained from some extreme function $\pi$ for a master problem $\Rf(G, \Z)$
with $P \subseteq G$ by restricting the function, $\pi' = \pi|_P$.  Moreover,
Gomory \cite{gom} gave a characterization of the minimal functions for the
master problem $\Rf(G, \Z)$ by the following functional inequalities and
equations:
\begin{subequations}\label{eq:minimal}
  \begin{align}
    &\pi(x) \geq 0 &&\text{for } x \in G, \label{eq:minimal:nonneg}\\
    &\pi(x+z) = \pi(x) && \text{for $x \in G$, $z\in \Z$} && \text{(periodicity)}\\
    &\pi(0)=0, \ \pi(f)=1, \label{eq:minimal:01}\\
    &\Delta\pi(x,y) \geq 0
    &&  \text{for } x,y \in G
                   && \text{(subadditivity),}\label{eq:minimal:subadd}\\
    &\Delta\pi(x, f-x) = 0
    && \text{for } x\in G 
                   && \text{(symmetry condition),} \label{eq:minimal:symm}
  \end{align}
\end{subequations}
where $\Delta\pi(x,y) = \pi(x)+\pi(y) -\pi(x+y)$ is the \emph{subadditivity
  slack} function.  By quotienting out by~$\Z$, this system describes a
polyhedron in $\R^{G/\Z}$.  Extreme functions are then simply the vertices of this
polyhedron; thus some of the subadditivity inequalities $\Delta\pi(x,y) \geq
0$ are tight, i.e., additivity holds.

\subsection{Continuous interpolations of extreme functions and the infinite
  group problem $\Rf(\R, \Z)$}

Gomory and Johnson, in their seminal papers \cite{infinite,infinite2}, noted
that many extreme functions for finite master group problems follow simple
patterns that become apparent in the piecewise linear interpolations of these functions.  The
simplest pattern is that of the well-known two-slope function giving the Gomory mixed
integer cut (\sagefuncwithgraph{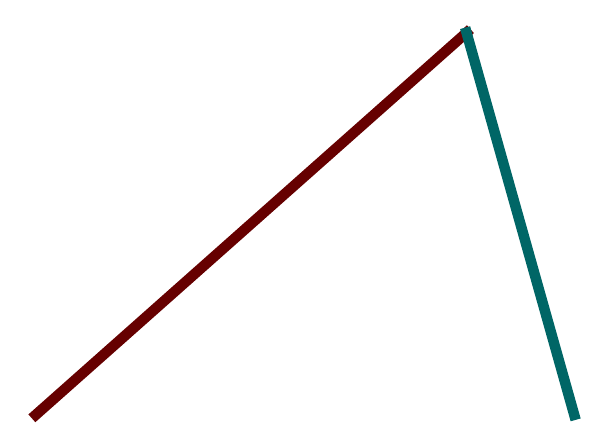}), which can be found in all finite group
problems; see \autoref{fig:gmic} (left).\footnote{A function name shown 
  in sans serif font refers to the software
  \cite{cutgeneratingfunctionology:lastest-release}, 
  which includes the 
  Electronic Compendium of Extreme Functions \cite{zhou:extreme-notes}.}
Gomory and Johnson
initiated a program to study such functions of a real variable systematically.  The technical
framework is that of the \emph{infinite group problem}, in which the group $G$
in \eqref{GP} is enlarged from $\frac1q\Z$ to $\R$.  Gomory and Johnson proved
that the characterization~\eqref{eq:minimal} of minimal functions still
holds in this setting. 
\begin{figure}[t]
  \centering
  \begin{tabular}{@{}l@{\qquad}l@{}}
    \includegraphics[width=.48\linewidth]{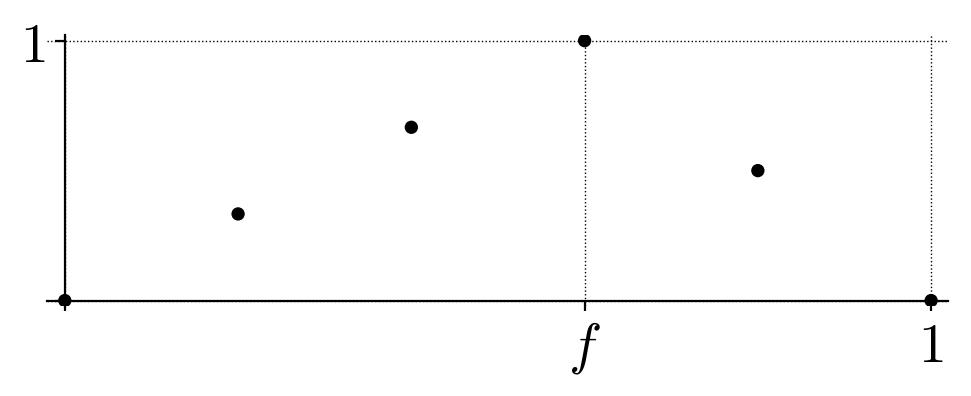} 
    & \includegraphics[width=.48\linewidth]{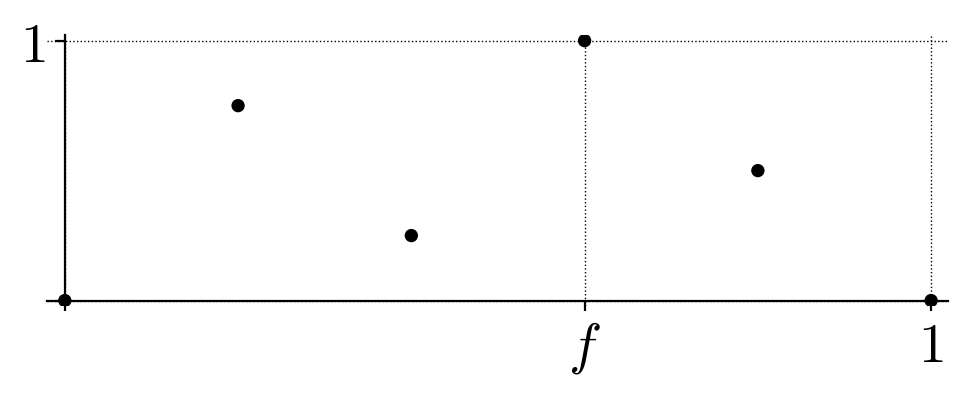}
    \\
    \includegraphics[width=.48\linewidth]{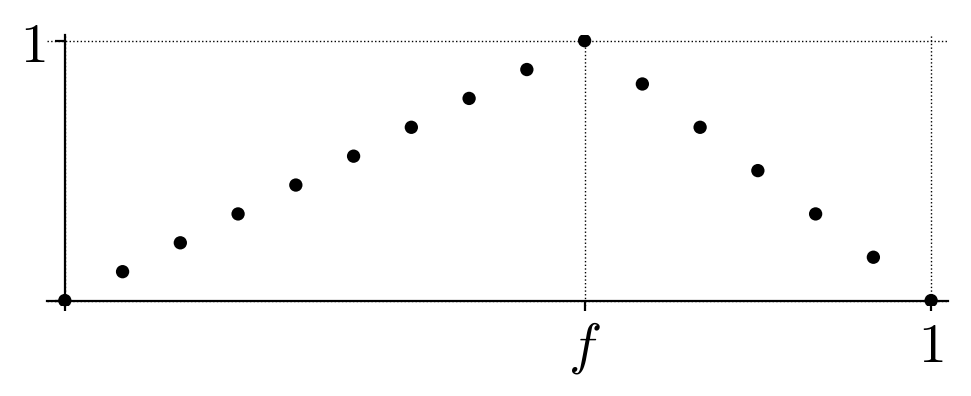} 
    & \includegraphics[width=.48\linewidth]{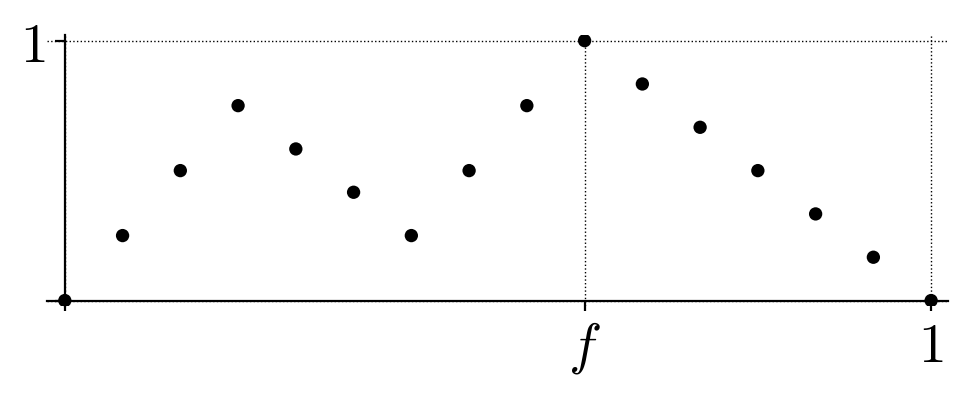}
    \\
    \includegraphics[width=.48\linewidth]{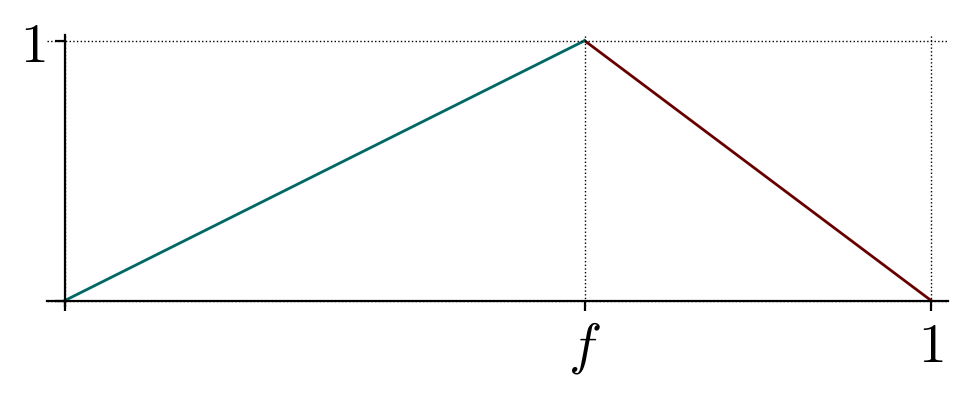} 
    & \includegraphics[width=.48\linewidth]{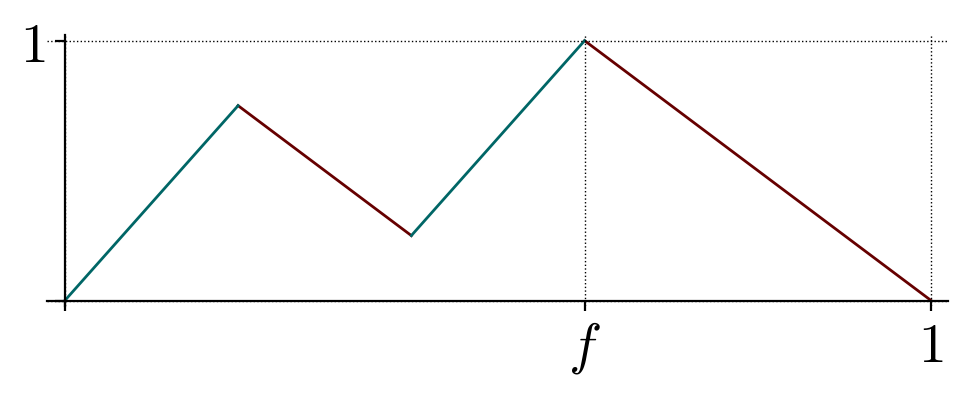}
  \end{tabular}    
  \caption[Extreme functions for finite master group problems following simple
    patterns that become apparent in the piecewise linear interpolations]
    {Extreme functions for finite master group problems following simple
    patterns that become apparent in the piecewise linear interpolations.
    \emph{Left,} the Gomory mixed-integer cut (\sagefunc{gmic}).
    \emph{Right,} another two-slope extreme function.}
  \label{fig:gmic}
\end{figure}

For an extreme function~$\pi|_{G}$ for a finite master problem $\Rf(G, \Z)$,
the piecewise linear interpolation $\pi =
\sagefunc{interpolate_to_infinite_group}(\pi|_G)$ is 
a minimal function, but not necessarily extreme.  (A partial converse is true;
the restriction of a continuous piecewise linear extreme function~$\pi$ for $\Rf(\R,
\Z)$ to a group $G$ that includes all breakpoints of~$\pi$ is extreme for
$\Rf(G, \Z)$.)
There is a possible viewpoint on the extreme functions for the infinite group problems as
``robust'' cut-generating functions that ignore the
specific number-theoretic properties of a particular group problem
$\Rf(\frac1q\Z,\Z)$.  As a matter of fact, in a numerical implementation, the
value $q$ and exact rational value of~$f$ would not be readily available.

A natural algorithmic focus lies on piecewise linear valid functions, though a
part of the literature \cite{bccz08222222,bcdsp:arbitrary-slopes} also studies more complicated
functions.  (When we discuss piecewise linear functions in this paper, we
include the discontinuous case with possible jumps at breakpoints, which
includes important examples such as the Gomory fractional cut,
\sagefuncwithgraphgomoryfractional.)

For $\Z$-periodic piecewise linear functions, the
characterization~\eqref{eq:minimal} of minimal functions gives a simple
algorithm for testing minimality, based on enumerating the vertices of a certain polyhedral complex;
see \cite[section 2.2]{basu-hildebrand-koeppe:equivariant} and \cite[section
5]{hong-koeppe-zhou:software-paper}.  For testing the extremality of a
piecewise linear minimal function, however, in contrast to the finite group
case, we cannot directly use polyhedral methods any more.  Since the quotient
$\R/\Z$ is not finite, we have to use infinite-dimensional methods of
functional equations and inequalities.  The most important lemma from this
theory is the
Gomory--Johnson interval lemma, variants of which has been used in virtually
all proofs of extremality in the literature.

\subsection{The space $\tildePi^{\texorpdfstring{\pi}{}}$ of effective perturbations $\tildepi$ of a minimal valid function}
\label{s:intro:effective-perturbations}

Recall that by definition~\eqref{eq:facet-definition-like-extreme-function}, a
minimal valid function $\pi$ is extreme if it cannot be written as a convex
combination of two other minimal valid functions $\pi^+$, $\pi^-$.  A fruitful
approach to extremality testing, introduced by Basu et al.\@{} \ifserialtitle in Part I of the present series of
papers \fi\cite{basu-hildebrand-koeppe:equivariant}, has been to consider the 
difference function (\emph{perturbation}) $\bar\pi = \pi^+ - \pi$, which
allows us to write $\pi^1 = \pi + \bar\pi$ and $\pi^2 = \pi - \bar\pi$.  
(Recently, Di Summa~\cite{DiSumma-2018:piecewise-smooth-piecewise-linear} obtained a
breakthrough result on the question of piecewise linearity of 
extreme functions using this approach.)
It is
convenient to build a space from the notion of perturbation functions.
Following \ifserialtitle Part~V \fi\cite[section 6]{hong-koeppe-zhou:software-paper}, we define the space 
\begin{equation}
  \label{eq:effective-perturbations}
\tildePi^{\pi}\RZ = \left\{\,\tilde{\pi} \colon \R \to \R \st \exists \, \epsilon>0 \text{ s.t.\ } \pi^{\pm} = \pi \pm \epsilon\tilde{\pi} \text{ are minimal valid}\,\right\}
\end{equation}
of
\emph{effective perturbation functions} for the minimal valid
function~$\pi$
. This is a vector space
(\autoref{lemma:effective-perturbation-vector-space}), a subspace of the space
of bounded functions.  The function~$\pi$ is
extreme if and only if the space $\tildePi^{\pi}\RZ$ is trivial.

If additivity ($\Delta\pi(x,y) = 0$) holds for some $(x,y)$, then by convexity
also $\Delta\tilde\pi(x,y) = 0$ holds for every effective perturbation
$\tilde\pi \in \tildePi^{\pi}\RZ$.  This is also true for additivity in
the limit \cite[Lemma 2.7]{basu-hildebrand-koeppe:equivariant}; see also
\cite[Lemma 6.1]{hong-koeppe-zhou:software-paper}.  Because $\pi$ is assumed
to be piecewise linear, the infinite system of functional equations describing
additivity and limit-additivity of~$\tilde\pi$ can be structured
(``combinatorialized'') according to a certain polyhedral complex
\cite{basu-hildebrand-koeppe:equivariant,hong-koeppe-zhou:software-paper}.

\subsection{Finite-dimensional and equivariant perturbations}
\label{s:intro:finite-dim-and-equivariant-perturbations}

\ifserialtitle In Part I of the present series, \fi 
Basu et
al.~\cite{basu-hildebrand-koeppe:equivariant} gave the first algorithm to
decide extremality of a piecewise linear function with rational breakpoints
in some
``grid'' (group) $G = \frac1q\Z$.  

In a
first step, one tests whether there exists a nontrivial perturbation for~$\pi$
in the
finite-dimensional subspace of $\tildePi^{\pi}\RZ$ that consists of the
functions $\sagefunc{interpolate_to_infinite_group}(\tilde\pi|_G)$, where
$\tilde\pi|_G$ is an effective perturbation function for the restriction
$\pi|_G$ to the finite group problem $\Rf(G,\Z)$.  

Otherwise, one may assume that $\tilde\pi|_G = 0$.  Under this assumption, the
interval lemma forces $\tilde\pi|_C = 0$ for certain \emph{directly covered}
intervals~$C$.  Basu et al.'s crucial observation was that if there are any
remaining \emph{uncovered} intervals, then one-dimensional families of
additivity equations impose a type of symmetry of the perturbation function.
By analyzing the required symmetry, one can construct a perturbation function
and prove nonextremality of~$\pi$.

Consider the additivity equations
\begin{equation}\label{eq:translation-equation}
  \Delta\tilde\pi(x, t) = \tilde\pi(x) + \tilde\pi(t) - \tilde\pi(x+t) = 0, 
  \quad \text{for $x \in D$}, 
\end{equation}
where $D$ is an interval and $t \in \frac1q\Z$ is a grid point.  Because
$\tilde\pi(t) = 0$, this simplifies to 
\begin{subequations}\label{eq:translation-reflection-equations-assuming-0}
\begin{equation}
  \tilde\pi(x) = \tilde\pi(x+t)\quad \text{for $x \in D$}.\label{eq:translation-equation-assuming-0}
\end{equation}
We then say that $\tilde\pi$ is \emph{invariant} under the action of the \emph{translation}
$\tau_t\colon x \mapsto x+t$ (restricted to the interval~$D$).
Likewise, a second type of one-dimensional families of additivity equations
simplifies to
\begin{equation}
  \tilde\pi(x) = - \tilde\pi(r-x)\quad \text{for $x \in
    D$}.\label{eq:reflection-equation-assuming-0} 
\end{equation}
\end{subequations}
Here a negative sign comes in.  We call $\rho_r\colon x\mapsto r-x$ a
\emph{reflection}.  By assigning a \emph{character} $\chi(\tau_t) = +1$ and
$\chi(\rho_r) = -1$ to the translations and reflections, we can unify
equations~\eqref{eq:translation-reflection-equations-assuming-0} as
\begin{equation}
  \tilde\pi(x) = \chi(\gamma)\, \tilde\pi(\gamma(x)) \quad \text{for $x \in D$},
\end{equation}
where $\gamma$ is either a translation or a reflection.  We then say that
$\tilde\pi$ is \emph{equivariant} under the action of $\gamma$. 

By analyzing the group~$\Gamma \subset \Aff(\R)$ generated by all relevant
translations and reflections, Basu et al. constructed a universal template
function $\psi\colon \R\to\R$, a continuous piecewise linear function with
breakpoints in $\frac1{4q}\Z$, which is equivariant under the action of the
group~$\Gamma$.  Taking
\begin{equation}
  \label{eq:masking-to-uncovered}
  \tilde\pi(x) =
  \begin{cases}
    \psi(x) & \text{for $x$ in uncovered intervals},\\
    0 & \text{for $x$ in covered intervals}
  \end{cases}
\end{equation}
then gives an effective perturbation function. 
(A revised construction in Basu et al.'s survey
\cite[section~8.2]{igp_survey_part_2} gives a continuous piecewise linear
function $\tilde\pi$ with breakpoints in $\frac1{3q}\Z$.)

\subsection{Contributions of the present paper}

\ifserialtitle
It has been a long-term research project to develop a complete, \emph{grid-free}
algorithmic theory and software implementation for piecewise linear minimal
valid functions, extending the reach of
the grid-based extremality test introduced in 
Part I of the series
\cite{basu-hildebrand-koeppe:equivariant}, which we described
in~\autoref{s:intro:finite-dim-and-equivariant-perturbations} above.  
While Parts II--IV develop a grid-based theory for 2-row relaxations, 
Part V of our series \cite{hong-koeppe-zhou:software-paper} returned to the one-row
case. It
introduced our software \cite{cutgeneratingfunctionology:lastest-release} and
prepared the grid-free theory with several results.
Part VI of the series \cite{koeppe-zhou:crazy-perturbation} discussed the case of piecewise
linear functions that are discontinuous on both sides of the origin and have
irrational breakpoints.  The present paper, part VII of the series, 
and a computational companion paper, part VIII of the series, 
are the culmination of the project for the case of piecewise linear functions
of one variable.
\fi

\subsubsection{Method: Inverse semigroups as the language of partial symmetries}

Group actions are the standard language to describe symmetries of mathematical
objects.  The use of group actions was fruitful in \ifserialtitle Part~I of
our series \else Basu et al.'s paper \cite{basu-hildebrand-koeppe:equivariant} \fi to
obtain the first algorithm for testing extremality.  However, group actions do
not provide a complete theory of the effective perturbations.  This becomes
most apparent in \cite[section~5]{basu-hildebrand-koeppe:equivariant}, where
Basu et al.~introduce a family of extreme functions with irrational
breakpoints, \sagefuncwithgraph{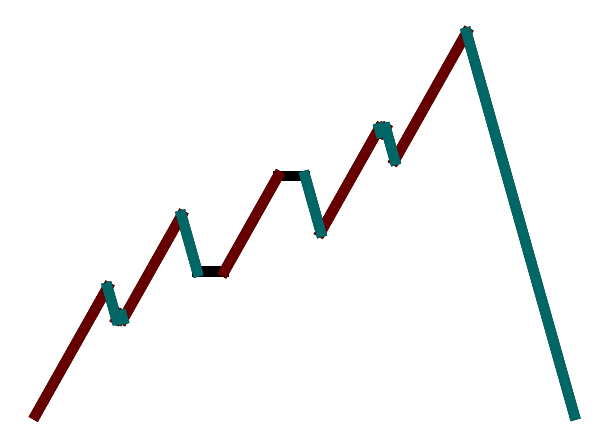}.  Here the group~$\Gamma$
generated by the translations and 
reflections only gives the correct result when a certain non-group-theoretic
reachability condition \cite[Assumption 5.1, Lemma
5.2]{basu-hildebrand-koeppe:equivariant} is satisfied.

The underlying reason is that the restriction of the translations and
reflections to the interval domains~$D$ in
\eqref{eq:translation-reflection-equations-assuming-0} is not considered in
the reflection group.  Indeed, what the translations and reflections describe
is not a full symmetry of the perturbation function, but only a \emph{partial
  symmetry} within the uncovered intervals.
\begin{figure}[t]
  \centering
  \Huge
  \hspace*{-1em}
  $
  \begin{aligned}
    \COMPOSE{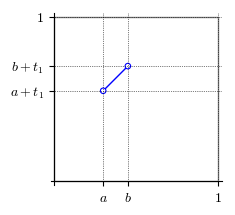}{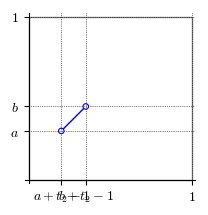} \\
    \COMPOSE{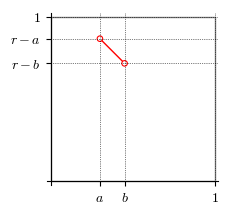}{tau2-} \\
    \COMPOSE{tau1+}{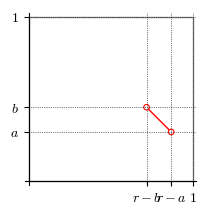}
  \end{aligned}
  $
  \caption{Operations of the inverse semigroup I: Composition \tred{Labels overlap in
  the figure. Fix this.}}
  \label{fig:semigroup-operations-composition}
\end{figure}

\begin{figure}[t]
  \centering
  \Huge
  $
  \begin{aligned}
    \VOP{tau1+}^{-1} &= \VOP{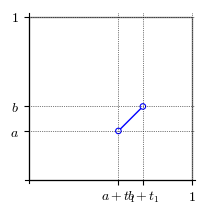} \\
    \VOP{rhoab+}^{-1} &= \VOP{rhoab-}
  \end{aligned}
  $
  \caption{Operations of the inverse semigroup II: Inverse}
  \label{fig:semigroup-operations-inverse}
\end{figure}

\begin{figure}[t]
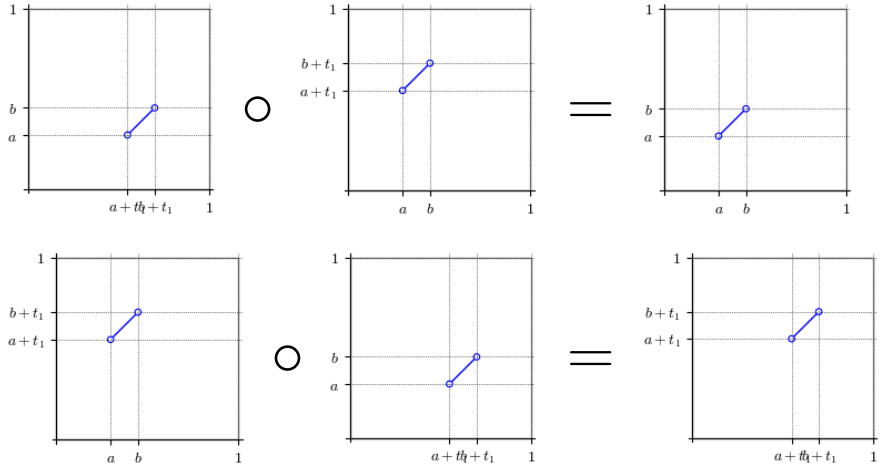

  \centering
  \Huge
  \hspace*{-1em}
  $
  \begin{aligned}
    \COMPOSE{tau1-}{tau1+} \\
    \COMPOSE{tau1+}{tau1-}
  \end{aligned}
  $
  \caption[Operations of the inverse semigroup III: Composition with inverse]
  {Operations of the inverse semigroup III: Composition with inverse.
    The partial identities $\tau_0|_D$ are the idempotents of the inverse semigroup.}
  \label{fig:semigroup-operations-composition-inverse}
\end{figure}


The correct language to describe partial symmetries is the less well-known
theory of inverse-semigroup actions.  An \emph{inverse
  semigroup}~$(\Gamma, \circ, \cdot^{-1})$, following \cite[page
7]{lawson-1998-inverse-semigroups}, is a semigroup, i.e., a set $\Gamma$
closed under an associative operation~$\circ$, satisfying the additional axiom that
\begin{equation}\tag{$\exists!$ inverse}
  \begin{gathered}[t]
    \text{for every $\omega \in \Gamma$, there exists a unique element
      $\omega^{-1}\in \Gamma$}\\
    \text{such that $ \omega = \omega \circ \omega^{-1} \circ \omega$
      and $ \omega^{-1} = \omega^{-1} \circ \omega \circ \omega^{-1}$}.
  \end{gathered}
\end{equation}
The equations in the axiom describe the familiar properties of a
pseudoinverse, but due to the required uniqueness, we will simply refer to
$\omega^{-1}$ as the \emph{inverse} of $\omega$.  In his
monograph~\cite{lawson-1998-inverse-semigroups}, Lawson points out that
\begin{quote}
  the relationship between inverse semigroups and partial symmetries is a
  generalization of the relation between groups and symmetries.
\end{quote}

Concretely, inverse semigroups arise as semigroups of partial bijections of a
set, where the operation $\circ$ is the composition and $\cdot^{-1}$ is the
inverse of a partial bijection.  We define the restrictions of the previously
defined translations and
reflections to open intervals~$D$.  We denote them by $\tau_t|_D$ and
$\rho_r|_D$ and consider them 
as partial bijections of $\R$ to itself, 
with domains $\dom(\tau_t|_D) = D = \dom(\rho_r|_D)$ and images
$\im(\tau_t|_D) = \tau_t(D) = D + t$ and $\im(\rho_r|_D) = \rho_r(D) = r -
D$.  We refer to these partial bijections as \emph{moves}. 
The composition of two moves $\gamma_1|_{D_1}$ and $\gamma_2|_{D_2}$ is
defined as 
\begin{equation}\label{eq:composition}
  \gamma_2|_{D_2}\circ\gamma_1|_{D_1}  = \gamma_2\circ\gamma_1|_{D_1 \cap \gamma_1^{-1}(D_2)};
\end{equation}
see \autoref{fig:semigroup-operations-composition}. 
The domain of the composition is either an open interval or the empty set. 
(By definition, there are exactly two empty moves: the empty translation
$\tau|_\emptyset$ and the empty reflection~$\rho|_\emptyset$.)
Note that the inverse of a move $\gamma|_D$, given by  $(\gamma|_D)^{-1} =
\gamma^{-1}|_{\gamma(D)}$, is not an inverse in a group-theoretic sense:
The compositions
\begin{equation}
  \label{eq:inverse}
  \gamma|_D \circ (\gamma|_D)^{-1} = \tau_0|_{\gamma(D)}
  \quad\text{and}\quad
  (\gamma|_D)^{-1} \circ \gamma|_D = \tau_0|_D
\end{equation}
are only \emph{partial identities} (restrictions of the identity $\tau_0$ to
intervals) and therefore not neutral elements but merely idempotents
(\autoref{fig:semigroup-operations-composition-inverse}).\medbreak

We develop methods that center around inverse semigroups of moves and their
generating sets.  We study the set of moves that are respected by the
effective perturbations of a given minimal function~$\pi$.  We analyze the
closure properties (axioms) that it satisfies: algebraically, it is an inverse
semigroup; but additional order-theoretic and analytic closure properties come
in.  Starting from an initial set (\emph{move ensemble})~$\Omegainit$, we can
then form the closure
with respect to the axioms.  We call it the \emph{moves closure} of $\Omegainit$
(or \emph{closed move semigroup} generated by~$\Omegainit$) and denote it by
$\clsemi{\Omegainit}$.

In the first part of the paper, we develop these methods in full generality,
without using any specific properties of the Gomory--Johnson model.  
Then we turn to the study of piecewise linear functions; here we make the
assumption of continuity from at least one side of the origin.

For all piecewise linear functions with rational breakpoints, we will show
that $\clsemi{\Omegainit}$ has a simple structure: Its graph consists of a finite
union of line segments and rectangles.  (We say that it is \emph{finitely
  presented}.)  It will become clear that we can compute $\clsemi{\Omegainit}$ in
finitely many steps using a completion-type algorithm, using only the
algebraic and order-theoretic axioms, by manipulating finite presentations of
generating systems.  However, this algorithm is not the focus of the present
paper: We defer all computational questions to the forthcoming companion paper
\cite{hildebrand-koeppe-zhou:algocomp-paper}.

Instead, an important point of our paper is that finitely presented closures
$\clsemi{\Omegainit}$ arise in a more general context, through the interplay of
the order-theoretic, algebraic, and analytic closure properties.  Move
ensembles whose graphs are connected open sets extend to open rectangles
already in the joined semigroup (\autoref{cor:open-sets-of-moves}).  
Our key theorem using the analytic properties is 
\autoref{thm:dense-boxes-in-joinlim}: Rectangles appear in the closure
whenever there is a convergent sequence of moves.  
\ifserialtitle
(In part I of our series, we have observed a glimpse of this phenomenon
already, in a specific arithmetic context
.)
\else
(In Basu et al.'s paper~\cite{basu-hildebrand-koeppe:equivariant}, 
a glimpse of this phenomenon was already observed in a specific arithmetic context.)
\fi
Empirically, for all
families of piecewise linear minimal valid functions in the literature
(see~\cite{zhou:extreme-notes} for an electronic compendium), even if the
breakpoints are irrational, the closure has a finite presentation.  This
includes the function \sagefuncwithgraph{bhk_irrational}, which we mentioned
above.  Again, we defer questions regarding the computation of this closure,
which then needs to use the additional axioms, to our forthcoming
paper~\cite{hildebrand-koeppe-zhou:algocomp-paper}.

\subsubsection{Result: Precise description of the space of equivariant perturbations}

Under the above assumptions, the finite presentation of $\clsemi{\Omegainit}$
allows us to read off a precise description of the space of
equivariant perturbations as a direct sum decomposition of vector subspaces
(\autoref{thm:decomposition-perturbation}).  

One component in the decomposition is a finite-dimensional space, consisting
of (possibly discontinuous) piecewise linear functions.  In contrast to the
grid-based algorithm, the set of breakpoints of these functions is not fixed,
but it is computed by our algorithm.  The finite-dimensional space is then
described by a system of finitely many linear equations
(\autoref{lemma:finite-dim-system}).

Then, for each of the finitely many \emph{uncovered components} (defined in
\autoref{sec:perturbation_space}), there is a component that is an
infinite-dimensional space isomorphic to the space
of Lipschitz functions on a compact interval that vanish on the boundary.
More specifically, our algorithm computes an open interval~$D$, the
\emph{fundamental domain}, on which we take the space of Lipschitz
functions~$\tilde\pi$ that vanish on the boundary~$\partial D$.  Additionally
there are finitely many moves $\gamma_j|_D$ with pairwise disjoint images
$\gamma_j(D)$ that together extend the functions equivariantly to the whole
uncovered component.  Outside of the component, the functions in this space
are zero.  This is \autoref{thm:equiv-perturbation-characterization}.

\smallbreak

This description of the space strengthens previous results. 
The method of \ifserialtitle Part~I \fi\cite{basu-hildebrand-koeppe:equivariant}, described in
\autoref{s:intro:finite-dim-and-equivariant-perturbations}, guarantees to
construct a piecewise linear effective perturbation if the space is nontrivial;
but it does not provide a complete description of the space. 
A theorem regarding direct sum decomposition appeared in \cite[Theorem
3.14]{igp_survey}, but it is limited to the grid case.

We remark that the precise description of the perturbation space of a minimal
function~$\pi$ enables us to strengthen (\emph{lift}) it by constructing a
direction in the space of effective perturbations.  By our theorem, the
problem of finding such a direction decomposes into subproblems; one finite-dimensional,
the others independent variational problems over Lipschitz functions.

\subsubsection{Computational implications: Grid-free algorithms, natural
  proofs}

We only sketch the computational implications of the present paper because we
will elaborate on them in our companion paper.  The inverse semigroup theory
lays the foundation for grid-free algorithms for minimal valid functions,
including automated extremality tests, which are detached from the finite
group problem.  A grid-free test is faster for functions whose breakpoints are
rational numbers with huge denominators; and it enables computations for
functions with irrational breakpoints.  More importantly, the grid-free
algorithms can give natural extremality proofs, similar to the general proof
pattern of extremality proofs in the published literature.  In this way, the
grid-free algorithms enable automated extremality proofs for smoothly
parameterized families of extreme functions, as described in
\cite{koeppe-zhou:param-abstract-isco}.

\subsection{Structure of the paper}
In sections
\ref{sec:translation_reflection_moves}--\ref{sec:semigroup}, we
introduce moves as partial bijections of~$\R$.  We study \emph{ensembles} (sets) of
such moves, which can be equipped with both an order-theoretic structure
(restriction and continuation) and an algebraic structure (inverse
semigroups).  In \autoref{s:equivariant} we describe how move
ensembles and semigroups describe partial symmetries of a function by a system
of functional equations.  Move ensembles for bounded functions have additional
properties, which we explore in \autoref{sec:directly-covered-intervals}.
Then, in \autoref{sec:moves-closure}, we study closure properties that capture
the additional properties of move ensembles for continuous functions.  This
development culminates in the notion of \emph{closed move semigroups} in
\autoref{s:ctscl-definition}.

We then apply this theory to compute the effective perturbation space of a
piecewise linear minimal valid function~$\pi$.
In~\autoref{sec:construct_initial_moves}, we introduce the \emph{initial
  additive move ensemble} $\Omegainit$, which describes functional equations
satisfied by every effective perturbation of~$\pi$.  For piecewise linear
functions~$\pi$, it is related to the \emph{additive faces} of a polyhedral
complex (\autoref{s:piecewise}).  Finally, in
\autoref{sec:perturbation_space}, working with a finite presentation
of the closed move semigroup $\clsemi{\Omegainit}$ generated by~$\Omegainit$, we
prove the main theorem of the paper, the decomposition theorem for the space
of effective perturbations of~$\pi$.

We end the paper in
\autoref{s:conclusion} with a discussion of the limitations of our approach
and an outlook on the computational companion paper
\cite{hildebrand-koeppe-zhou:algocomp-paper}.
See the next pages for a detailed table of contents.

\bigbreak

\clearpage
\hypertarget{toc}{}%
\bookmark[level=1,dest=toc,italic]{Contents}   
{
\tableofcontents}
\clearpage
\listoffigures

\clearpage
\bookmark[level=1,dest=notation,italic]{Notation for move ensembles and semigroups}
\let\reforpageref=\pageref
\begin{table}[t]\hypertarget{notation}{}
  \caption{Notation for move ensembles and semigroups}
  \label{tab:notation}
  \centering
    \begin{tabular}{p{.11\linewidth}p{.77\linewidth}c@{\,}}
      \toprule
      $\FullMoveGroup$
      & Group of unrestricted translations and reflections of~$\R$
      &\reforpageref{s:FullMoveGroup}
      \\
      \quad $\tau_t$, $\rho_r$
      & \quad  
        translation, reflection
      \\
      \quad $\gamma$
      & \quad some element
      \\[0.7ex] 
      $\FullMoveSemigroup$
      & Inverse semigroup of translations,
        reflections with domains
      & \reforpageref{def:FullMoveSemigroup}
      \\
      \quad $\tau_t|_D$
      &\quad translation restricted to open interval $D$
      \\
      \quad $\rho_r|_D$
      &\quad reflection restricted to open interval $D$
      \\
      \quad $\gamma|_D$
      &\quad unrestricted move restricted to open interval $D$
      \\[0.7ex] 
      $\Omega$
      & A move ensemble: a subset of $\FullMoveSemigroup$
      & \reforpageref{sec:ensembles}
      \\
      $\Omega^{\mathrm{inv}}$
      & \ldots\ satisfying \eqref{axiom:inverse}
      & \reforpageref{sec:semigroup:def}
      \\
      \noindent\hphantom{$\Omega^{\mergeplus/}$,} $\Gamma$
      & A move semigroup: an inverse subsemigroup of $\FullMoveSemigroup$
      & \reforpageref{sec:semigroup:def}
      \\
      \noindent\rlap{$\Omega^{\restrict/}$,}\hphantom{$\Omega^{\mergeplus/}$,} $\Gamma^{\restrict/}$
      & A move ensemble, or semigroup, satisfying \eqref{axiom:restrict}
      & \reforpageref{sec:restriction-closed}
      \\
      \noindent\rlap{$\Omega^{\join/}$,}\hphantom{$\Omega^{\mergeplus/}$,} $\Gamma^{\join/}$
      & 
        \ldots\ satisfying \eqref{axiom:restrict},
        \eqref{axiom:join}
      & \reforpageref{s:join-def}
      \\
      $\Omega^{\mergeplus/}$, $\Gamma^{\mergeplus/}$
      & 
        \ldots\ satisfying
        \eqref{axiom:restrict}, \eqref{axiom:join},
        \eqref{axiom:translation-reflection}
      & \reforpageref{s:kaleido-def}
      \\
      \noindent\rlap{$\IsLimit{\Omega}$,}\hphantom{$\Omega^{\mergeplus/}$,} $\IsLimit{\Gamma}$
      & 
        \ldots\ satisfying limit axiom \eqref{axiom:limits} 
        or \eqref{axiom:arblim}
      & \reforpageref{sec:limit-def}
      \\
      \noindent\rlap{$\IsJoinExtend{\Omega}$,}\hphantom{$\Omega^{\mergeplus/}$,} $\IsJoinExtend{\Gamma}$
      & \ldots\ satisying \eqref{axiom:joinextend}
      \\[0.7ex]
      $\Omega^{\fin/}$
      & A finite move ensemble \hphantom{ponents}
        \rdelim\}{2}{20pt}[~finite presentation]
      & \multirow{2}{*}{\reforpageref{s:omega-fin-notation}}
      \\
      $\C$
      & Connected covered components
      & 
      \\
      $\Omega^{\red/}$
      & A reduced finite move ensemble
      & \reforpageref{s:omega-red-notation}
      \\[0.7ex]
      $\posetL$, $\posetM$
      & Families of move ensembles
                          & \reforpageref{s:ctscl-definition}
      \\
      \bottomrule
      \vspace{4ex}
    \end{tabular}
\end{table}


\bookmark[level=1,dest=axioms,italic]{List of axioms for move ensembles}
\newcommand\rdelimbraceparbox[3]{\rdelim\}{#1}{20pt}[~\parbox{#2}{\begin{spacing}{1}\raggedright #3 \end{spacing}}]}
\begin{table}[t]\hypertarget{axioms}{}
  \caption{List of axioms for move ensembles}
  \label{tab:axioms}
  \centering
  \begin{spacing}{1.7}
    \begin{tabular}{@{\,}l@{\ }p{.24\linewidth}@{\,}p{.24\linewidth}@{ }p{.2\linewidth}l@{\,}}
      \toprule
      \eqref{axiom:composition}
      & \rdelimbraceparbox{2}{7em}{\emph{move semigroup}\\[0.5ex]$\Gamma = \semi{\Omega}$}
      & \rdelimbraceparbox{4}{6.5em}{\emph{joined semigroup}\\[1ex]$\Gamma^{\join/} = \joinsemi{\Omega}$}
      & \rdelimbraceparbox{7}{6.5em}{\emph{closed move
        semigroup}\\[1ex]$\IsExtend{\Gamma}^{\join/} = \clsemi{\Omega}$}
      & \pageref{axiom:composition} \\
      \eqref{axiom:inverse}
      &&&
      & \pageref{axiom:inverse} \\
      \eqref{axiom:restrict}
      & \rdelimbraceparbox{2}{7em}{\emph{joined ensemble}\\[0.5ex]$\Omega^{\join/} = \join{\Omega}$}
      & &
      & \pageref{axiom:restrict} \\
      \eqref{axiom:join}
      &&&
      & \pageref{axiom:join} \\
      \eqref{axiom:translation-reflection}
      & \multicolumn{1}{l}{\emph{kaleidoscopic ens.}} & $\Omega^\mergeplus/$
      & & \pageref{axiom:translation-reflection} \\
      \eqref{axiom:limits},  \eqref{axiom:arblim}
      & \multicolumn{1}{l}{\emph{limits-closed ens.}} & $\bar\Omega = \arblim{\Omega}$
      & & \pageref{axiom:limits} \\
      \eqref{axiom:joinextend}
      & \multicolumn{1}{l}{\emph{extended ensemble}} & \multicolumn{1}{@{}l}{\rlap{$\IsJoinExtend{\Omega} = \joinextend{\Omega}$}}
      & & \pageref{axiom:joinextend}\\
    \bottomrule
  \end{tabular}
\end{spacing}
\end{table}

\clearpage

\section{Translation and reflection moves. Their algebraic and
  order-theoretic structure}
\label{sec:translation_reflection_moves}

\tred{Goal of generality for the initial sections: (1) Avoid 'piecewise
  linear' and 'finitely many intervals'.  (2) Introduce notation to replace
  (0,1) so that everything can be used for other models, not just GJ.}

\subsection{Group $\FullMoveGroup$ of unrestricted translations $\tau_{\texorpdfstring{t}{\unichar{"1D64}}}$ and
  reflections $\rho_{\texorpdfstring{r}{\unichar{"1D63}}}$, character $\chi$}
\label{s:FullMoveGroup}
\begin{definition}
  For a point $r \in \R$, define the (unrestricted) \emph{reflection}
  $\rho_r\colon \R\to \R$, $x \mapsto r-x$.  For a vector $t \in \R$, define
  the (unrestricted) \emph{translation}
  $\tau_t\colon \R \to \R$, $x \mapsto x + t$.  
\end{definition}
The set
$ \FullMoveGroup =
\{\, \rho_r, \tau_t \st r \in \R,\; t \in \R\,\}$  of all translations and reflections, with the operations
of function composition $\circ$ and inverse $\cdot^{-1}$, has the structure of
a group.  
It is a subgroup of
the group $\Aff(\R)$ of regular affine transformations of~$\R$.

To denote an element that can be either a translation or a
reflection, we will usually use the letter $\gamma$.
To recover whether an element $\gamma$ is a translation or a reflection, 
we assign a \emph{character} $\chi(\rho_r) = -1$ to every reflection and
$\chi(\tau_t) = +1$ to every translation.
The map $\gamma\mapsto\chi(\gamma)$
is a group character, i.e., a homomorphism, so compositions of elements follow
the rule 
\begin{equation}
  \chi(\gamma_1 \circ \gamma_2) = \chi(\gamma_1) \cdot \chi(\gamma_2).\label{eq:group-character}
\end{equation}

\subsection{Restricted moves
  $\gamma\texorpdfstring{|_D}{|\unichar{"208D}\unichar{"1D64}\unichar{"2009}\unichar{"0326}\unichar{"1D65}\unichar{"208E}} \in{} \FullMoveSemigroup$ as partial bijections of~$\R$}

As we mentioned in the introduction, compared to
\cite{basu-hildebrand-koeppe:equivariant}, where finitely generated subgroups
of $\FullMoveGroup$ were used for the grid-based extremality test algorithm, in this
paper we develop a more detailed theory using restricted moves with domains.
Our terminology is based on the monograph
\cite{lawson-1998-inverse-semigroups} on inverse semigroups.

We begin by restricting translations and reflections $\gamma \in \FullMoveGroup
$ to open
interval domains $D \subseteq \R$.
\begin{definition}\label{def:FullMoveSemigroup}
  \label{def:move}
  Let $\gamma \in \FullMoveGroup
  $ be a translation or reflection, and let $D \subseteq \R$ be an open interval
  or the empty set.
  \begin{enumerate}[\rm(a)]
  \item The \emph{move} $\gamma|_D$ is the partial function with domain $D$
    and image $\gamma(D)$, defined by $\gamma|_D(x) = \gamma(x)$ for
    $x \in D$.
  \item The \emph{character} of $\gamma|_D$ is the character of $\gamma$.
  \item Two moves $\gamma_1|_{D_1}, \gamma_2|_{D_2}$ with open interval
    domains $D_1, D_2$ are \emph{equal} if
    $\gamma_1 = \gamma_2$ and $D_1 = D_2$.  A move with an open interval
    domain is not \emph{equal} to a move with an empty domain.
    We identify all translations with empty domain and denote this object by
    $\tau|_\emptyset$. Likewise, we identify all reflections with empty domain
    and denote this object by $\rho|_\emptyset$.  The empty translation and the
    empty reflection are not \emph{equal}; they are distinct objects with
    $\chi(\tau|_\emptyset) = +1$ and $\chi(\rho|_\emptyset) = -1$.
  \item The set of all moves is denoted by $\FullMoveSemigroup$.
  \end{enumerate}
\end{definition}

\subsubsection{Remark on the relation to pseudogroups}

Inverse semigroups of partial homeomorphisms between open subsets of a
topological space are known as \emph{pseudogroups}
\cite[section 1.2]{lawson-1998-inverse-semigroups}.  However, our theory differs in the
following ways: (1) We only allow open intervals (and the empty set) as
domains of the partial functions, rather than arbitrary open subsets. The
reason for our choice will become clear in \autoref{s:equivariant}, where we
will use moves to describe systems of functional equations.
(2) Less importantly, we have two empty moves, one for each possible
character, rather than a unique empty move.

\subsection{Graphs of moves}
\label{s:graphs-of-moves}

We find it convenient to describe the graphs of moves. The \emph{graph} of
$\gamma|_D$ is the set \[\graph(\gamma|_D) = \{\, (x,y) \in \R\times\R \st x \in D,\;
  \gamma(x) = y \,\}.\]   

Figures showing the graphs have already appeared in
\autoref{fig:semigroup-operations-composition} and
\autoref{fig:semigroup-operations-inverse}.  To emphasize that the domains of all
nonempty moves are open intervals, we decorate the endpoints of the moves by
hollow circles, indicating that the endpoints are not part of the graphs.

\subsection{Restriction partial order $\subseteq$ on moves}
\label{s:restriction-partial-order}

The set of all moves comes with a natural partial order.
\begin{definition}
  $\gamma_1|_{D_1}$ is a \emph{restriction} of $\gamma_2|_{D_2}$, denoted $\gamma_1|_{D_1}
  \subseteq \gamma_2|_{D_2}$, if $D_1 \subseteq D_2$, $\chi(\gamma_1) =
  \chi(\gamma_2)$, and $\gamma_1(x) = \gamma_2(x)$ for $x \in D_1$.
\end{definition}
Thus, in this partial order, translations and reflections are incomparable.  We have
$\tau|_\emptyset \subseteq \tau_t|_D$ for all translations and likewise
$\rho|_\emptyset \subseteq \rho_r|_D$. 

\begin{definition}
Given $\gamma|_{D}$ and an open interval (or empty set) $D' \subseteq D$, the
\emph{restriction} of $\gamma|_{D}$ to $D'$ is the move
$(\gamma|_D)\big|_{D'} =\gamma|_{D'}$.  Given an open interval (or empty set)
$I' \subseteq \gamma(D)$, the \emph{corestriction} of $\gamma|_{D}$ to $I'$ is
the move ${}_{I'}\big| (\gamma|_D) = \gamma|_{D\cap \gamma^{-1}(I')}$. 
\end{definition}

\subsection{Inverse semigroup structure $(\FullMoveSemigroup, \circ, \texorpdfstring{\cdot^{-1}}{\unichar{"2219}\unichar{"207B}\unichar{"B9}})$}

Let $\gamma_1|_{D_1}$ and $\gamma_2|_{D_2}$ be two moves. As noted in the
introduction, their composition
$\gamma_2|_{D_2}\circ\gamma_1|_{D_1} $ is defined as
$\gamma_2\circ\gamma_1|_{D_1 \cap \gamma_1^{-1}(D_2)}$
(\autoref{fig:semigroup-operations-composition}).  The domain of this partial
bijection is either an open interval or the empty set; so it is again a move.
It is clear that the composition operation $\circ$ is associative. Hence the
moves form a semigroup~$(\FullMoveSemigroup, \circ)$.

As we have noted already, a move $\gamma|_D$ also has a (unique) inverse given by
$(\gamma|_D)^{-1} = \gamma^{-1}|_{\gamma(D)}$
(\autoref{fig:semigroup-operations-inverse}) satisfying the
laws~\eqref{eq:inverse}
(\autoref{fig:semigroup-operations-composition-inverse}).  Hence the moves
form an inverse semigroup $(\FullMoveSemigroup, \circ, \cdot^{-1})$.  Its
idempotent elements are exactly the 
\emph{partial identities}, which are restrictions of the identity translation
$\tau_0$ to open intervals~$D$ together with the empty
translation~$\tau|_\emptyset$.  (The empty reflection is not idempotent; we
have $\rho|_\emptyset \circ \rho|_\emptyset = \tau|_\emptyset$.)

The inverse semigroup structure interacts with the restriction partial order
(\autoref{s:restriction-partial-order}) as follows \cite[Proposition
1.1.4]{lawson-1998-inverse-semigroups}.  If
$\gamma|_{D'} \subseteq \gamma|_{D}$, then
$\gamma|_{D'}^{-1} \subseteq \gamma|_{D}^{-1}$; moreover, this restriction can
be expressed as a composition with an idempotent:
$\gamma|_{D'} = (\gamma|_{D})\big|_{D'} = \gamma|_{D} \circ \tau_0|_{D'}$.
Finally, if $\gamma_i|_{D'_i} \subseteq \gamma_i|_{D_i}$ for $i = 1,2$, then
$\gamma_2|_{D'_2} \circ \gamma_1|_{D'_1} \subseteq \gamma_2|_{D_2} \circ
\gamma_1|_{D_1}$.

\section{Ensembles $\Omega$ of moves. Their order-theoretic
  structure}
\label{sec:ensembles}

Now we consider \emph{move ensembles} $\Omega$, i.e., arbitrary subsets of the inverse
semigroup~$\FullMoveSemigroup$.  We denote elements
by $\gamma|_D$, where $\gamma
\in \FullMoveGroup$ is an unrestricted move and $D$ is the domain
.

\subsection{Order-theoretic structure}

\subsubsection{Restriction-closed move ensembles $\Omega^{\restrict/}$}
\label{sec:restriction-closed}

\begin{definition}
  A move ensemble~$\Omega^{\restrict/}$ is \emph{restriction-closed} if it
  satisfies the following axiom.
  \begin{equation}
    \tag{restrict}\label{axiom:restrict}
    \begin{gathered}[t]
      \text{If $\gamma|_{D} \in \Omega^{\restrict/}$ and $D' \subseteq D$ is an
        open interval or the empty set,}\\
      \text{then $\gamma|_{D'} \in \Omega^{\restrict/}$.}
    \end{gathered}
  \end{equation}
  For a move ensemble~$\Omega$, the restriction closure $\restrict{\Omega}$ is
  the smallest restriction-closed move ensemble containing~$\Omega$.  (It
  consists of all restrictions of moves of $\Omega$.)
\end{definition}

\begin{remark}
  Throughout the paper, a superscript like $\restrict/$ in
  $\Omega^{\restrict/}$ indicates an axiom that the set $\Omega^{\restrict/}$
  satisfies. See \autoref{tab:notation} for an overview of notation used in
  the paper.
\end{remark}

\begin{example}
  The inverse semigroup $\FullMoveSemigroup$ of all restricted translations
  and reflections is a restriction-closed move ensemble.
\end{example}

\subsubsection{Join-closed move ensembles $\IsJoin{\Omega}$}

\label{s:join-def}
\begin{definition}
  A move ensemble $\IsJoin{\Omega}$ is \emph{(completely) join-closed} if it
  satisfies \eqref{axiom:restrict} and the following condition.
  \begin{equation}
    \tag{continuation}\label{axiom:join}
    \begin{gathered}[t]
      \text{If there is a family
        $ \Omega_{\mathfrak I} = \{\, \gamma|_I \st I \in \mathfrak{I} \,\} \subseteq
        \IsJoin{\Omega}$}\\
      \text{such that $D = \textstyle\bigcup_{I \in \mathfrak{I}} I$ is an interval,}\\
      \text{then $\gamma|_D \in \IsJoin{\Omega}$.} 
    \end{gathered}
  \end{equation}
\end{definition}

\tblue{Note that \eqref{axiom:restrict} is part of the definition of
  ``join-closed''.  To avoid confusion of the axiom by itself and the axiom
  combined with \eqref{axiom:restrict}, we have renamed the axiom
  ``continuation''. That seems to be used in part of pseudogroup literature.
  Then whenever we say join, it means continuation and restriction together.}

\begin{definition}
  We define the \emph{joined ensemble} $\join{\Omega}$ of~$\Omega$ as the
  smallest set of moves containing $\Omega$ that satisfies
  \eqref{axiom:join} and \eqref{axiom:restrict}.
\end{definition}

\begin{lemma}
  \label{lemma:joined-ensemble-explicitly}
  For a move ensemble~$\Omega$, the joined ensemble $\join{\Omega}$ consists
  of the following moves.
  \begin{equation}
    \label{eq:join-generation}
    \begin{gathered}[t]
      \Bigl\{\,
      \gamma|_{D} \Bigst \text{$D \subseteq \bigcup_{I \in
          \mathfrak{I}} I$, where $\gamma|_I \in \Omega$ for $I \in
        \mathfrak{I}$,  $D$ empty or open interval}
      \,\Bigr\}.
    \end{gathered}
  \end{equation}
\end{lemma}
\begin{proof}
  This set clearly satisfies \eqref{axiom:join} and
  \eqref{axiom:restrict}. On the other hand, $\join{\Omega}$ needs to contain
  this set.
\end{proof}

\tred{Add lemmas/remarks regarding moves in joined ensemble here? 
  Such as: (because of infinite join) endpoints of moves of $\join{\Omega}$ are limits of endpoints of
  moves of $\Omega$. 
}

\subsubsection{Presentation by the set\/ $\maxdom(\Omega^{\join/})$ of maximal elements}

\begin{definition}
  For a move ensemble $\Omega$, let $\maxdom(\Omega)$ denote the set of maximal
  elements of~$\Omega$ in the restriction partial order.
\end{definition}

\begin{lemma}\label{lemma:joined-is-presented-by-max}
  A join-closed move ensemble $\Omega^{\vee}$ is equal to the restriction closure and to the
  joined ensemble of its maximal elements in the restriction partial order:
  \begin{equation}
    \Omega^{\vee}  = \restrict{\maxdom(\Omega^{\vee})} = \join{\maxdom(\Omega^{\vee})}
  \end{equation}
\end{lemma}
\begin{proof}
  Let $\gamma|_D \in \Omega^{\vee}$.  Let
  $\mathfrak D = \{\, D' \supseteq D \st \gamma|_{D'} \in
  \Omega^{\vee} \,\}$.
  Let $\bar D = \bigcup \mathfrak D$, an interval.  Then
  $\bar D \in \Omega^{\vee}$ because $\Omega^{\vee}$ satisfies \eqref{axiom:join}. Moreover,
  $\gamma|_D \subseteq \gamma|_{\bar D} \in \maxdom(\Omega^{\vee})$
  and thus $\gamma|_D \in \restrict{\maxdom(\Omega^{\vee})}$.  The
  other inclusions are trivial.
\end{proof}

\subsection{Move ensembles as set-valued maps $\R{} \Rightarrow{}
  \R$. Domains, images, restrictions}
\label{s:restrictions}

\begin{definition}
  Let $\Omega$ be a move ensemble and $R$ be a
  disjoint union of open intervals, $R = \bigcup_{R'\in \mathfrak I} R'$. 
  The \emph{restriction} $\Omega|_R$ is the move
  ensemble consisting of the 
  restrictions $\gamma|_{D\cap R'}$ whenever $\gamma|_D \in \Omega$, $R' \in
  \mathfrak I$, and either $D = \emptyset$ or $D \cap R' \neq \emptyset$. 
  Similarly, we define 
  the \emph{corestriction} ${}_R|\Omega$ and 
  the \emph{double restriction} ${}_R|\Omega|_R$.
\end{definition}
In the restrictions, domains of moves are restricted to subintervals of~$R$.
Note that by our definition, the restrictions contain empty moves if and only if
$\Omega$ contains empty moves.  Therefore we have the following two convenient
properties:
\begin{lemma}\label{lemma:restriction-of-restrict/join-is-restrict/join}
  For a move ensemble $\Omega^{\restrict/}$ that satisfies \eqref{axiom:restrict},
  the restrictions satisfy \eqref{axiom:restrict}, and
  we have
  \begin{align*}
    \Omega^{\restrict/}|_R &= \bigl\{\, \gamma|_D \in \Omega^{\restrict/} \bigst D \subseteq R \,\bigr\}, \\
    {}_R|\Omega^{\restrict/}\hphantom{|_R} &= \bigl\{\, \gamma|_D \in \Omega^{\restrict/} \bigst \gamma(D) \subseteq R \,\bigr\}, \\
    {}_R|\Omega^{\restrict/}|_R &= \bigl\{\, \gamma|_D \in \Omega^{\restrict/} \bigst D, \gamma(D) \subseteq R \,  \,\bigr\}.
  \end{align*}
  Likewise, restrictions also preserve \eqref{axiom:join}.
\end{lemma}
\begin{lemma}\label{lemma:restriction-of-maxdom-of-join-is-max}
  Let $\Omega^{\max} = \maxdom(\Omega^{\join/})$, where $\Omega^{\join/}$ is a
  joined ensemble.  Then each of the restrictions $\Omega^{\max}|_R$,
  ${}_R|\Omega^{\max}$, and ${}_R|\Omega^{\max}|_R$ consists of the maximal
  elements of $\Omega^{\join/}|_R$, ${}_R|\Omega^{\join/}$, and
  ${}_R|\Omega^{\join/}|_R$, respectively.
\end{lemma}

We define the domain and image of a move ensemble~$\Omega$, as well as the
image of a set under the ensemble.  
\begin{definition}
  Let $\Omega$ be a move ensemble.
  The \emph{domain} of $\Omega$ is
  \begin{displaymath}
    \dom(\Omega) = \bigcup\bigl\{\, D \bigst \text{$\gamma|_D \in \Omega$ for some
      $\gamma$} \,\bigr\};
  \end{displaymath}
  the \emph{image} of $\Omega$ is 
  \begin{displaymath}
    \im(\Omega) = \bigcup\bigl\{\, \gamma(D) \bigst \text{$\gamma|_D \in \Omega$ for some
      $\gamma$} \,\bigr\}.
  \end{displaymath}
\end{definition}

\begin{definition}
  Let $\Omega$ be a move ensemble and $X \subseteq \R$ be a set.  Define
  \begin{displaymath}
    \Omega(X) = \bigl\{\, \gamma(x) \bigst \gamma|_D \in \Omega,\; x\in X\cap D
    \,\bigr\}.
  \end{displaymath}
\end{definition}

\begin{remark}
  In these notions, a move ensemble behaves like a set-valued map
  $\Omega\colon \R \Rightarrow \R$.
\end{remark}

\subsection{Graphs $\graph(\Omega), \graph\unisubscript+(\Omega), \graph\unisubscript-(\Omega)$ of move ensembles $\Omega$}
\label{s:graphs-of-ensembles}

We introduced graphs of moves in \autoref{s:graphs-of-moves}.
For a move ensemble $\Omega$ we define the \emph{translation moves graph}
\begin{align*}
  \graph_{+}(\Omega) &= \bigcup \bigl\{\, \graph(\gamma|_D) \bigst \gamma|_D \in \Omega
                       \text{ and } \chi(\gamma) = 1 \,\bigr\},\\
  \intertext{consisting of line segments
  with slopes $+1$, and the \emph{reflection moves graph}}
  \graph_{-}(\Omega) &=  \bigcup \bigl\{\, \graph(\gamma|_D) \bigst \gamma|_D \in
                       \Omega \text{ and } \chi(\gamma) = -1\,\bigr\}, 
\end{align*}
consisting of line segments with slopes $-1$. 
The \emph{graph} of $\Omega$ is 
\begin{displaymath}
  \graph(\Omega) = \graph_{+}(\Omega) \cup \graph_{-}(\Omega).
\end{displaymath}
We also define the \emph{character conflict graph},
\begin{displaymath}
  \graph_{\pm}(\Omega) = \graph_{+}(\Omega) \cap \graph_{-}(\Omega).
\end{displaymath}

The map $\Omega\mapsto (\graph_{+}(\Omega), \graph_{-}(\Omega))$ becomes an
injection if restricted to the join-closed move
ensembles $\Omega^{\vee}$.  Hence these pairs of graphs faithfully
represent all join-closed move ensembles.  (In our
figures showing these graphs, we superimpose the translation graph (blue) and
reflection graph (red).)

We can go back from graphs to ensembles using the following notation.
\begin{definition}
Let $O \subseteq \R^2$. We define the (join-closed) move ensembles 
  \begin{align*}
    \tMovesOfGraph(O) &= \bigl\{\, \tau_t|_D \bigst \graph(\tau_t|_D) \subset O,\;
                        D \text{ an  open interval or empty}\,\bigr\},\\
    \rMovesOfGraph(O) &= \bigl\{\, \rho_r|_D \bigst \graph(\rho_r|_D)  \subset
                        O,\; D \text{ an  open interval or empty}\,\bigr\},\\
    \MovesOfGraph(O) &= \bigl\{\, \gamma|_D \bigst  \graph(\gamma|_D)  \subset
                       O,\; D \text{ an  open interval or empty} \,\bigr\}.
  \end{align*}
Thus, $\MovesOfGraph(O)  = \tMovesOfGraph(O) \cup \rMovesOfGraph(O)$.
\end{definition}

\section{Inverse semigroups generated by move ensembles}
\label{sec:semigroup}

Now we turn to the study of inverse semigroups
generated by move ensembles.

\subsection{Move semigroups~$\Gamma$; move semigroups $\semi{\Omega}$
  generated by ensembles $\Omega$}
\label{sec:semigroup:def}

\begin{definition}
  A move ensemble $\Gamma$ is a \emph{move semigroup} (or, an \emph{inverse subsemigroup} of
  $\FullMoveSemigroup$) 
  if it satisfies the following axioms:
  \begin{align}
    \tag{composition}& \label{axiom:composition} \gamma'|_{D'} \circ \gamma|_D \in \Gamma \text{ for all }\gamma|_D, \gamma'|_{D'} \in \Gamma, \\
    \tag{inv}& \label{axiom:inverse} (\gamma|_D)^{-1} \in
               \Gamma  \text{ for all  } \gamma|_D \in
               \Gamma.
  \end{align}
\end{definition}

\begin{definition} 
  For a move ensemble $\Omega$, the \emph{move semigroup} $\semi{\Omega}$
  generated by $\Omega$ is the smallest move semigroup containing
  $\Omega$.
\end{definition}

\begin{definition}
  A move semigroup $\Gamma$ is \emph{finitely generated} if there exists a
  finite set $\Omega$ such that $\Gamma = \semi{\Omega}$.
\end{definition}

\begin{lemma}\label{lemma:inv-semigroup-words}
  Let \(\IsInv{\Omega}\) be a move ensemble satisfying
  \eqref{axiom:inverse}. 
  Then $\semi{\IsInv{\Omega}}$ is the set of all finite compositions
  $\gamma^k|_{D_k} \circ \dots \circ \gamma^1|_{D_1}$ of moves $
  \gamma^i|_{D_i} \in \IsInv{\Omega}$.
\end{lemma}

\begin{remark}
  Since the domains of moves in $\Omega$ are empty or open intervals, any move
  $\gamma|_D \in \semi{\Omega}$ also has a domain $D$ that is empty or an open
  interval.  If $\gamma|_D \in \Omega$, then the idempotent
  $(\gamma|_D)^{-1} \circ \gamma|_D^{} = \tau_0 |_D^{}$ is an element of $\semi{\Omega}$.  The
  inverse semigroup generated by the empty set is the empty set.
\end{remark}

\subsection{Move semigroups and joins; joined move semigroups
  $\joinsemi{\Omega}$ generated by ensembles $\Omega$}
\label{sec:semigroup:join}

Move semigroups generated by joined ensembles are not automatically
join-closed. 
On the other hand, joining does preserve the semigroup properties.

\begin{lemma}
\label{lemma:joinsemi-is-semi}
  Let $\Gamma$ be a move semigroup.  
  Then the joined ensemble $\join{\Gamma}$ is a move semigroup.
  In particular, for a move ensemble $\Omega$, we have 
  $$\join{\semi{\Omega}} = \semi{\join{\semi{\Omega}}}.$$
\end{lemma}
\begin{proof}
Let $\gamma|_{D}, \gamma'|_{D'} \in  \join{\Gamma}$. 
We first show that $\join{\Gamma}$ satisfies the axiom \eqref{axiom:composition}.
By equation~\eqref{eq:join-generation}, 
there exist collections $\mathfrak{I}$ and $\mathfrak{I}'$ of open intervals, such that 
$D \subseteq \bigcup_{I \in \mathfrak{I}} I$, $D' \subseteq \bigcup_{I' \in \mathfrak{I}'} I'$, and $\gamma|_I, \gamma'|_{I'} \in \Gamma$ for all $I \in \mathfrak{I},  I' \in \mathfrak{I}'$.
We know that
\[\gamma'|_{I'} \circ \gamma|_{I} = (\gamma' \circ \gamma)|_{\gamma^{-1}(I') \cap I} \in \Gamma, \text{ for all } I \in \mathfrak{I} \text{ and } I' \in   \mathfrak{I}',\]
since $\Gamma$~satisfies \eqref{axiom:composition}, and that
\[\gamma^{-1}(D') \cap D \subseteq \gamma^{-1}\bigl(\bigcup_{I' \in \mathfrak{I}'} I'\bigr) \cap \bigl(\bigcup_{I \in \mathfrak{I}} I\bigr) =  \bigcup_{I \in \mathfrak I,\, I' \in \mathfrak I'} \bigl(\gamma^{-1}(I') \cap I\bigr).\]
Therefore, by equation~\eqref{eq:join-generation}, 
\(\gamma'|_{D'} \circ \gamma|_{D} = (\gamma' \circ \gamma)|_{\gamma^{-1}(D') \cap D} \in  \join{\Gamma}\).

We will now show that $\join{\Gamma}$ satisfies axiom~\eqref{axiom:inverse}.
We know that \((\gamma|_I)^{-1} = \gamma^{-1}|_{\gamma(I)} \in \Gamma \text{ for all } I \in \mathfrak{I},\) since $\Gamma$ satisfies \eqref{axiom:inverse}, and that \(\gamma(D) \subseteq \gamma(\bigcup_{I \in \mathfrak{I}} I) = \bigcup_{I \in \mathfrak{I}} \gamma(I)\).
Therefore,
$(\gamma|_D)^{-1} = \gamma^{-1}|_{\gamma(D)} \in  \join{\Gamma}$. 
We conclude that $ \join{\Gamma}$ is a move semigroup, so $\join{\Gamma} = \semi{\join{\Gamma}}$.
\end{proof}

\begin{definition}
  Let $\Omega$ be a move ensemble. 
  Then the \emph{joined move semigroup} of~$\Omega$ is defined as
  $$ \joinsemi{\Omega} = \join{\semi{\Omega}}.$$
\end{definition}

\subsection{Move semigroups $\MovesOfGraph(O)$, $\tMovesOfGraph(O)$,
  $\rMovesOfGraph(O)$ generated by connected open ensembles}
\label{sec:semigroup:open}

Finitely generated inverse semigroups, as defined in \autoref{sec:semigroup:def},
are not general enough for our purposes.  As we will see later, we need to
consider move ensembles $\Omega$ whose graphs are open connected sets.  They
generate inverse semigroups $\semi{\Omega}$ that are not finitely generated.
However, they have the following simple structure (see \autoref{fig:open-set-of-moves-sampled}).

\begin{theorem}
\label{thm:open-sets-of-moves}
  Let $\Omega$ be an ensemble of moves.
  Let $O \subseteq \R^2$ be a connected open set. Let $D = \dom(O) :=
  \dom(\MovesOfGraph(O))$ and $I = \im(O) := \im(\MovesOfGraph(O))$.
  \begin{enumerate}
  \item If $\graph_+(\Omega)$ contains $O$, then $\graph_+(\semi{\Omega})$ contains $(D \cup I) \times (D \cup I)$.
  \item If $\graph_-(\Omega)$ contains $O$, then $\graph_-(\semi{\Omega})$ contains $(D \times I) \cup (I \times D)$ and $\graph_+(\Omega)$ contains $(D \times D) \cup (I \times I)$.
  \item If $\graph_\pm(\Omega)$ contains $O$, then
    $\graph_\pm(\semi{\Omega})$ contains $(D\cup I) \times (D \cup I)$.  
  \end{enumerate}
\end{theorem}
\begin{figure}
  \Huge \centering
    \mbox{$
  \vcenter{\hbox{\includegraphics[height=.33\linewidth]{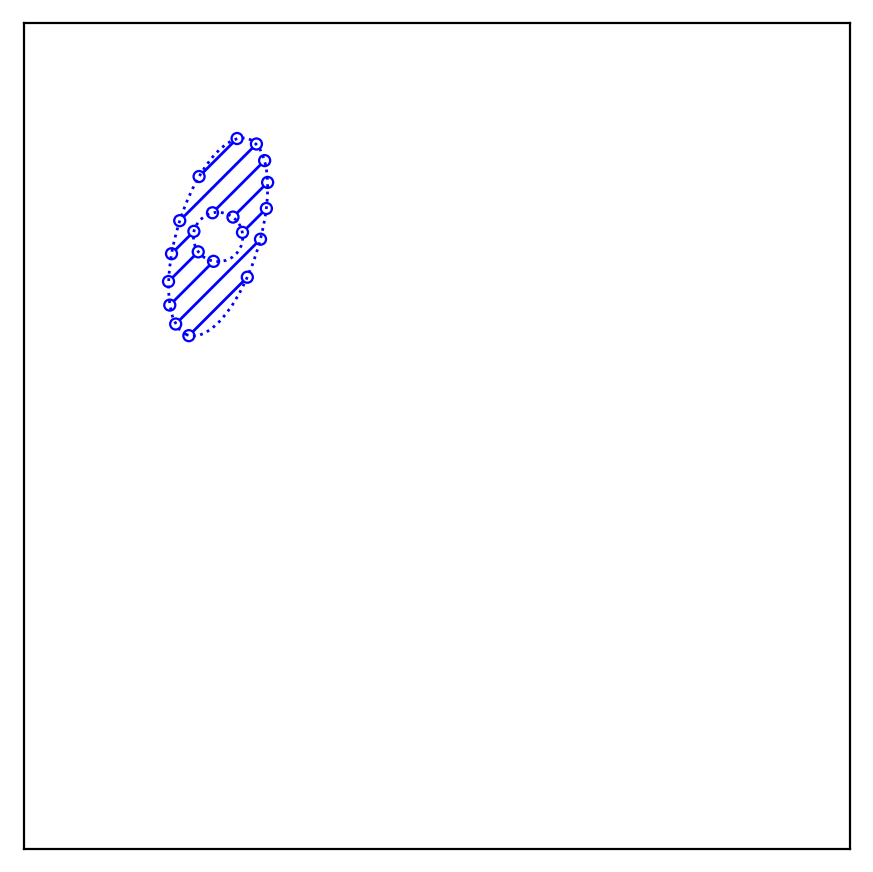}}}
  \xmapstojoinsemi
  \vcenter{\hbox{\includegraphics[height=.33\linewidth]{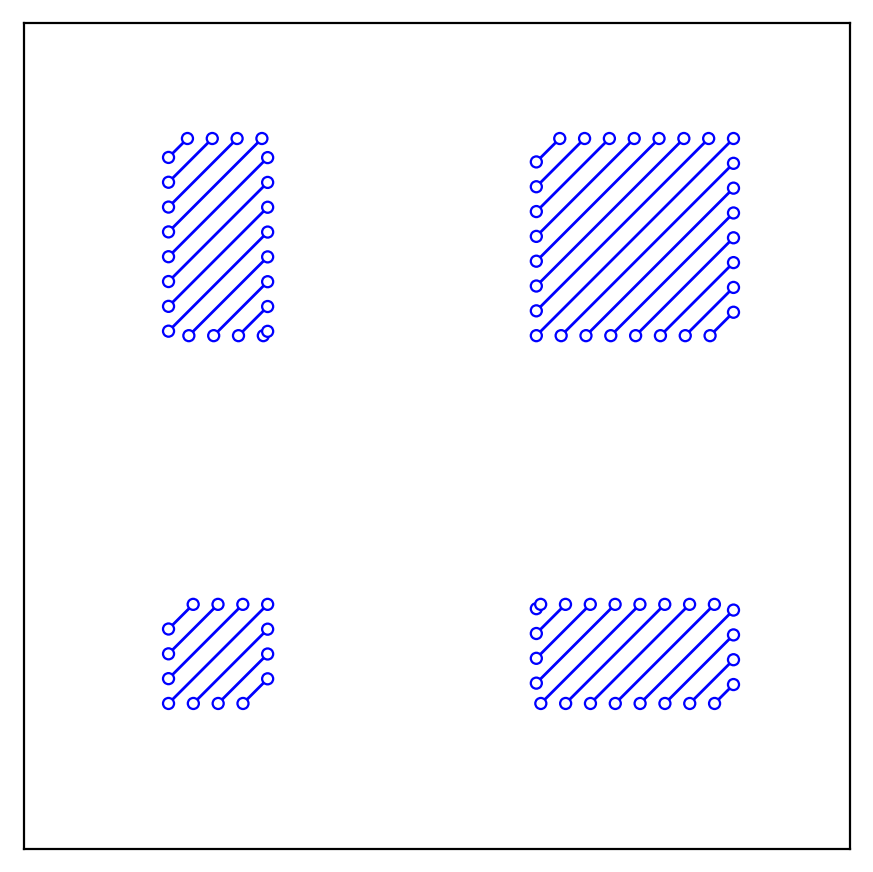}}}
  $}
      \mbox{$
  \vcenter{\hbox{\includegraphics[height=.33\linewidth]{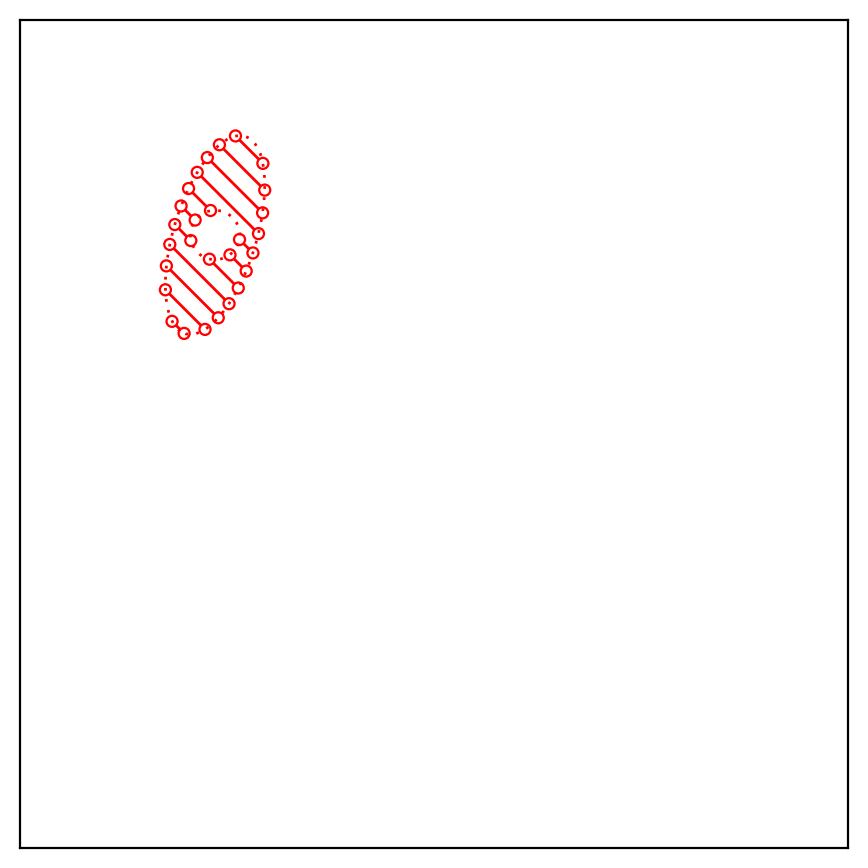}}}
  \xmapstojoinsemi
  \vcenter{\hbox{\includegraphics[height=.33\linewidth]{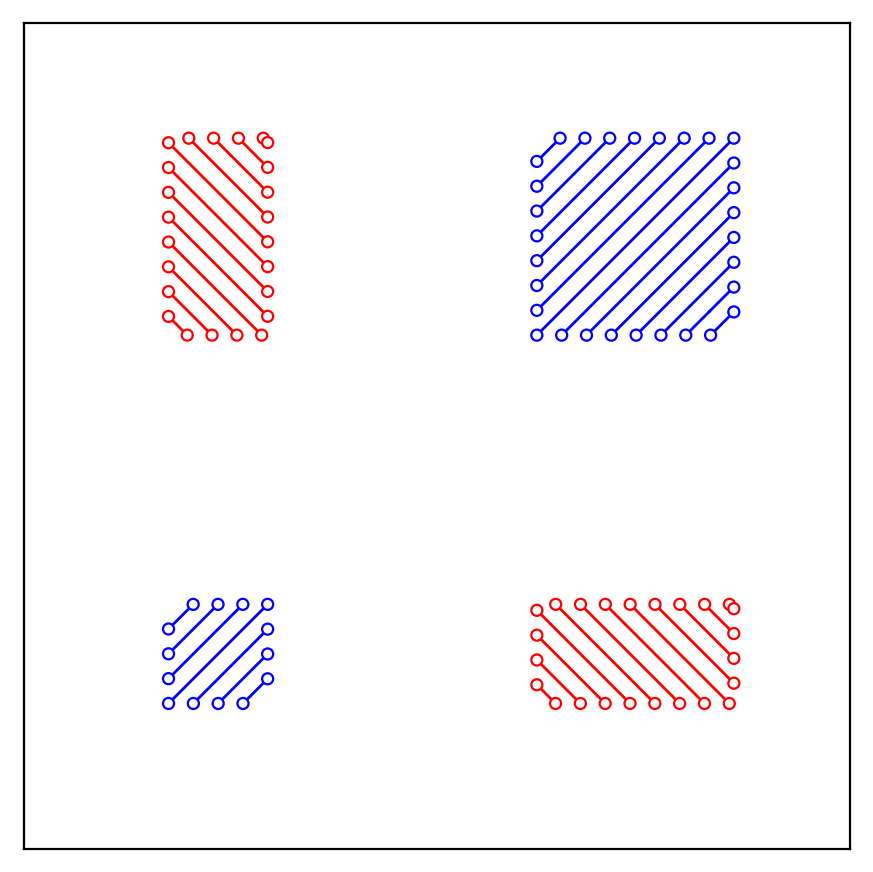}}}
  $}
      \mbox{$
  \vcenter{\hbox{\includegraphics[height=.33\linewidth]{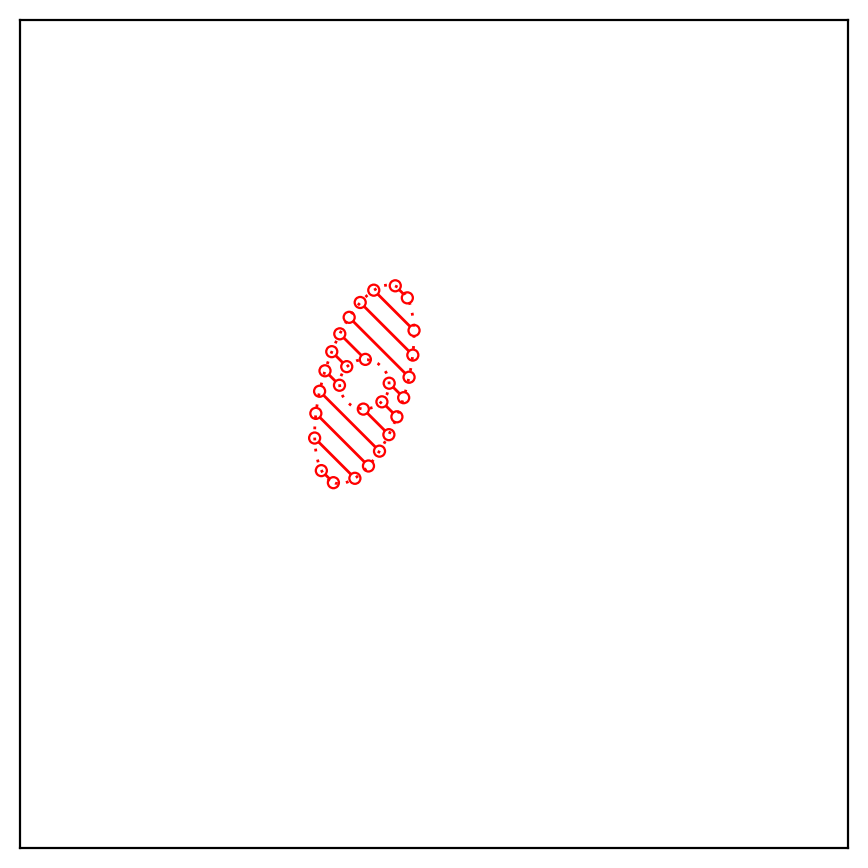}}}
  \xmapstojoinsemi
  \vcenter{\hbox{\includegraphics[height=.33\linewidth]{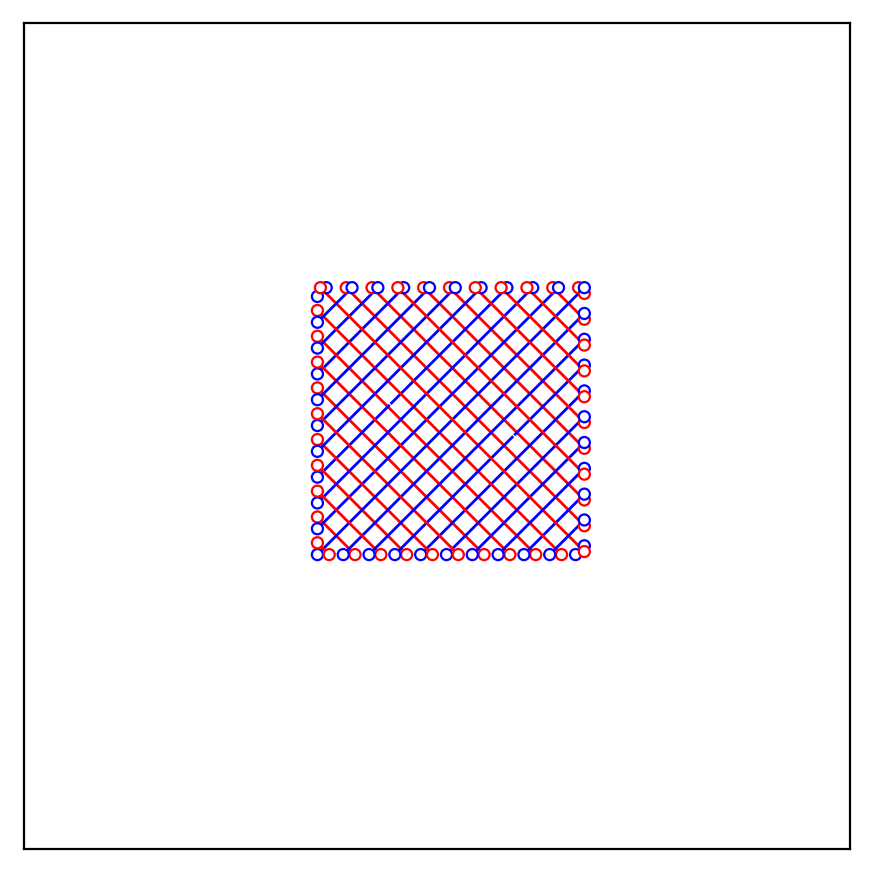}}}
  $}
      \mbox{$
  \vcenter{\hbox{\includegraphics[height=.33\linewidth]{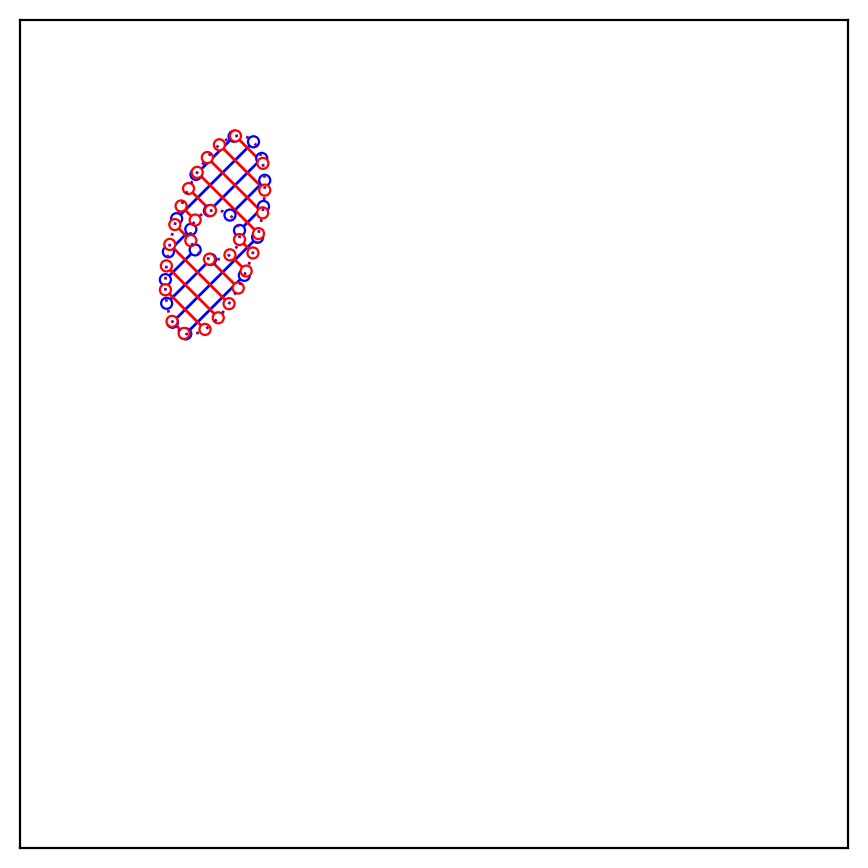}}}
  \xmapstojoinsemi
  \vcenter{\hbox{\includegraphics[height=.33\linewidth]{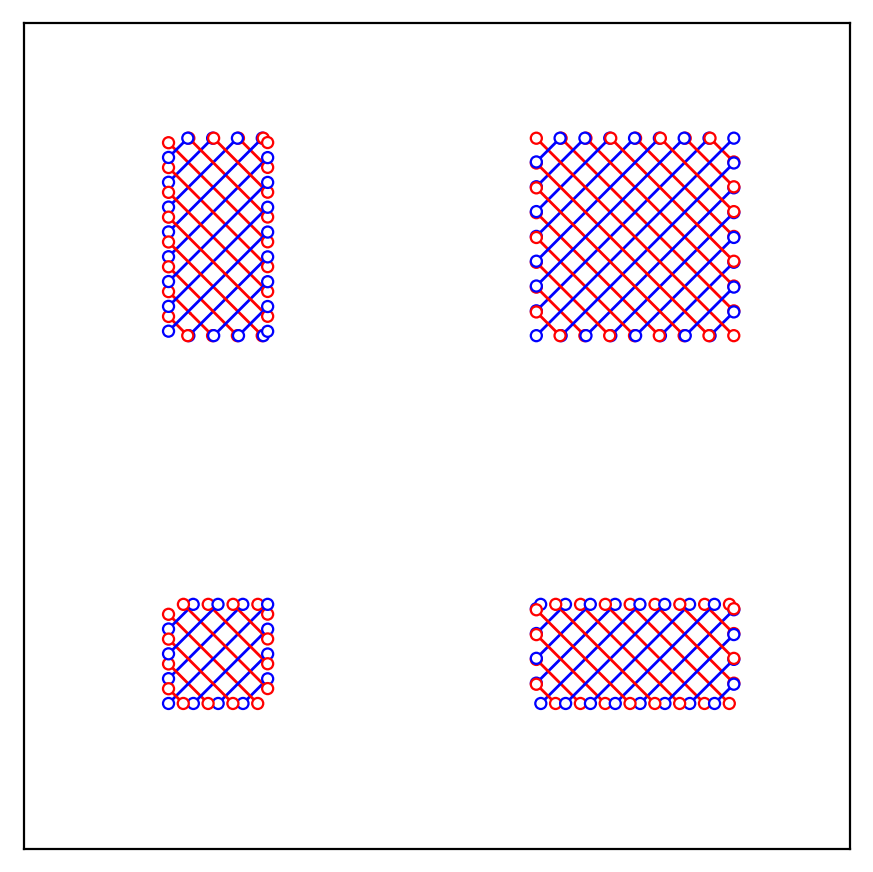}}}
  $}
  \caption[Illustrations for \autoref{thm:open-sets-of-moves} and
    \autoref{cor:open-sets-of-moves}]
    {Illustrations for \autoref{thm:open-sets-of-moves} and
    \autoref{cor:open-sets-of-moves}.
  Here, only a finite set of moves is considered.  If, however, an
  infinite set is used by considering all moves in the $O$-shaped set in the
  left plots, then the entire rectangles would be filled in on the right plots. }
  \label{fig:open-set-of-moves-sampled}
\end{figure}
\begin{proof}
Part 2. We show that (2a) $\graph_-(\semi{\Omega})$ contains $D \times I$ and (2b) $\graph_+(\semi{\Omega})$ contains $D \times D$; the other two containments of  $I \times D$ and  $I \times I$ follow from the fact that $\semi{\Omega}$ is closed under inverse.

Let $(x,y), (x',y') \in O$ be two arbitrary points in the connected open set $O$. Since there is a path between $(x,y)$ and $(x',y')$ contained in $O$, and the path is compact, it is covered by finitely many open $\ell_\infty$-balls $O_1, \dots, O_n \subseteq O$ with $(x_1,y_1):=(x,y) \in O_1$, 
$(x_2, y_2) \in O_1 \cap O_2, \dots, (x_n, y_n) \in O_{n-1} \cap O_n$ and $(x_{n+1},y_{n+1}):=(x',y') \in O_n$. 
 Since $(x_1, y_1), (x_2, y_1), (x_2, y_2), \allowbreak \dots, \allowbreak (x_n, y_n), \allowbreak (x_{n+1}, y_n), \allowbreak(x_{n+1}, y_{n+1})\in O$, 
 there exist
 $\rho_{r_1}|_{D_1},  \, \rho_{r_1'}|_{D_1'}, \, \rho_{r_2}|_{D_2}, \allowbreak \dots, \allowbreak \rho_{r_n}|_{D_n}, \, \allowbreak\rho_{r_n'}|_{D_n'} \allowbreak\text{ and } \allowbreak\rho_{r_{n+1}}|_{D_{n+1}} \in \Omega$
 such that 
 $\rho_{r_i}|_{D_i}(x_i) = y_i$ for $i=1,\dots, n+1$ and $\rho_{r_i'}|_{D_i'}(x_{i+1}) = y_i$ for $i=1,\dots, n$. Notice that the inverse restricted reflections ${(\rho_{r_i'}|_{D_i'})}^{-1} \in \semi{\Omega}$ with ${(\rho_{r_i'}|_{D_i'})}^{-1}(y_i) = x_{i+1}$ for $i=1,\dots, n$. 
 We have
 $$x_1 \xmapsto{\rho_{r_1}|_{D_1}} y_1\xmapsto{{(\rho_{r_1'}|_{D_1'})}^{-1} } x_2
  \mapsto \cdots \mapsto 
  y_n \xmapsto{{(\rho_{r_n'}|_{D_n'})}^{-1}} x_{n+1} \xmapsto{\rho_{r_{n+1}}|_{D_{n+1}}} y_{n+1}.$$ 
The composition of the $2n+1$ reflections 
\[\rho_r|_{D_r} :=\rho_{r_{n+1}}|_{D_{n+1}} \circ {(\rho_{r_n'}|_{D_n'})}^{-1}  \circ \rho_{r_n}|_{D_n} \circ \dots \circ (\rho_{r_1'}|_{D_1'})^{-1} \circ \rho_{r_1}|_{D_1}\] is a restricted reflection, satisfying that $\rho_r|_{D_r} \in \semi{\Omega}$ and $\rho_r|_{D_r}(x) = y'$. Therefore, (2a) holds.
The composition of the $2n$ reflections  
\[\tau_t|_{D_t} := {(\rho_{r_n'}|_{D_n'})}^{-1}  \circ \rho_{r_n}|_{D_n} \circ \dots \circ (\rho_{r_1'}|_{D_1'})^{-1} \circ \rho_{r_1}|_{D_1}\] 
is a restricted translation, satisfying that $\tau_t|_{D_t}  \in \semi{\Omega}$ and $\tau_t|_{D_t} (x) = x'$. Therefore, (2b) holds.

Part 1 follows exactly the same proof as part 2 using instead restricted translations 
$\tau_{t_1}|_{D_1},  \, \tau_{t_1'}|_{D_1'}, \, \tau_{t_2}|_{D_2},  \dots, \tau_{t_n}|_{D_n}, \, \tau_{t_n'}|_{D_n'}, \, \tau_{t_{n+1}}|_{D_{n+1}} \in \Omega$.

Part 3. Let $(x,y), (x',y') \in O$. By part 1 and 2, there exist restricted translation and reflection $\tau_t|_{D_t}, \rho_r|_{D_r} \in \semi{\Omega}$ such that $x \xmapsto{\tau_t|_{D_t}} y\xmapsto{\rho_r|_{D_r}} x'$. The composition $\rho_r|_{D_r} \circ \tau_t|_{D_t}$ is a restricted reflection in $\semi{\Omega}$. Therefore,  $\graph_-(\semi{\Omega})$ contains $D \times D$. By part 1, part 2 and the fact that $\semi{\Omega}$ is closed under inverse, we obtain that part 3 holds.
\end{proof}

The following corollary sharpens the result.

\begin{corollary}\label{cor:open-sets-of-moves}
Let $O \subseteq \R^2$ be a connected open set, with 
$D = \dom(O) = \dom(\MovesOfGraph(O))$ and $I = \im(O) = \im(\MovesOfGraph(O))$.
\begin{equation*}
\begin{array}{rr@{\:}ll}
$(1)$ \,& \joinsemi{\tMovesOfGraph(O)} = & \tMovesOfGraph\left((D \cup I)\times (D \cup I)\right). & \\
$(2a)$ \,& \joinsemi{\rMovesOfGraph(O)} =& \rMovesOfGraph((D \times I) \cup (I \times D)) \cup {}\\
       \,&     & \tMovesOfGraph((D \times D) \cup (I \times I)), &\text { if }D \cap I = \emptyset. \\
$(2b)$ \,& \joinsemi{\rMovesOfGraph(O)} =& \MovesOfGraph\left((D \cup I)\times (D \cup I)\right), &\text { if } D \cap I \neq \emptyset.\\
$(3)$ \,& \joinsemi{\MovesOfGraph(O)} =& \MovesOfGraph\left((D \cup I)\times (D \cup I)\right). &
\end{array}
\end{equation*}
\end{corollary}
\begin{proof}
By applying \autoref{thm:open-sets-of-moves}-(1), (2) and (3) to $\Omega=\tMovesOfGraph(O)$, $\Omega=\rMovesOfGraph(O)$ and $\Omega=\MovesOfGraph(O)$, we obtain that $\joinsemi{\Omega}$ on the left-hand side of the equation in (1), (2a) and (3) contains the move ensemble on the right-hand side, respectively. In case (2b) where $D \cap I \neq \emptyset$, by applying \autoref{thm:open-sets-of-moves}-(2) to $\Omega=\rMovesOfGraph(O)$, we have that $\joinsemi{\Omega}$ contains $\MovesOfGraph\left((D \cup I)\times (D \cap I)\right)$. It then follows from \autoref{thm:open-sets-of-moves}-(3) that  $\joinsemi{\Omega}$ contains 
the right-hand side of (2b).
Conversely, the right-hand side of the equation in each case is a joined move semigroup 
that contains~$\Omega$. Hence, the equality holds. 
\end{proof}

\begin{remark}\label{rem:lines+boxes-nonpurple}
  \autoref{thm:open-sets-of-moves} suggests to consider the following class of
  generating ensembles for inverse semigroups.  Take a finite ensemble
  $\Omega^{\fin/} = \{ \gamma^1|_{D_1}, \dots, \gamma^n|_{D_n} \} $ together with 
  a finite list of infinite ensembles of the form $\MovesOfGraph_+(D_{i}\times
  I_{i})$, $i = n+1, \dots, n+m$ and
  $\MovesOfGraph_-(D_{i}\times I_{i})$, $i = n+m+1, \dots, n+m+\ell$%
  , where $D_i$ and $I_i$ are open intervals.  However, we
  suppress the details of this.  In \autoref{sec:directly-covered-intervals}, an additional assumption
  will allow us to use a more convenient class of generating ensembles.
\end{remark}

\section{$\Omega$-equivariant functions}
\label{s:equivariant}

\subsection{Spaces of $\Omega$-equivariant functions}
\label{s:space-equivariant}

Move ensembles encode a system of functional equations as follows.

\begin{definition}
  Let $\Omega$ be a move ensemble and let $\theta\colon \R\to\R$ be a function.
  \begin{enumerate}[\rm(a)]
  \item 
    We say that $\theta$ is \emph{affinely $\Omega$-equivariant} (in short,
    $\theta$ \emph{respects} $\Omega$) provided that for every
    $\gamma|_D \in \Omega$ %
    there exists a constant $c_{\gamma|_D}^\theta$ such
    that
    \begin{equation}
      \label{eq:additivity-equation-for-move}
      \theta(\gamma|_D(x))=\chi(\gamma)\theta(x)+c^{\theta}_{\gamma|_D} \quad  \text{ for } x \in D,
    \end{equation} 
    where $\chi(\gamma)=\pm 1$ is the character of $\gamma$.
  \item 
    If all constants $c_{\gamma|_D}^\theta$ can be chosen to be zero, 
    then we say that $\theta$ is \emph{$\Omega$-equivariant} (or,
    \emph{equivariant under the action of~$\Omega$}).
  \end{enumerate}
\end{definition}

Throughout the paper, we will be working with affinely $\Omega$-equivariant
functions.  At the very end, in \autoref{sec:perturbation_space}, an
important space of $\Omega$-equivariant functions will appear.

\begin{remark}\label{rem:no-singleton-domains}It now becomes clear why singletons $\{x\}$ are not allowed as
  the domain~$D$ of a move.  The functional
  equation~\eqref{eq:additivity-equation-for-move} would degenerate to a
  single equation with an independent constant $c^\theta_{\gamma|_{\{x\}}}$.
  The equation and the constant can be eliminated from the system.
\end{remark}
Some trivial relations between the constants $c^{\theta}_{\gamma|_D}$ are
induced by the restriction partial order on moves
(\autoref{s:restriction-partial-order}).  If $\emptyset \neq D \subset D'$,
thus $\gamma|_D \subseteq \gamma|_{D'}$ and $D\neq \emptyset$, then
necessarily $c^{\theta}_{\gamma|_D} = c^{\theta}_{\gamma|_{D'}}$.

Thus it is natural to work with restriction-closed ensembles, as defined in
\autoref{sec:restriction-closed}.

\begin{lemma}
    For a space $\Theta$ of functions, we denote by $\Theta^\Omega$ the
    set of affinely $\Omega$-equivariant functions in~$\Theta$.
  If $\Theta$ is a vector space, then so is $\Theta^\Omega$.
\end{lemma}
\begin{proof}
  Let $\theta_1, \theta_2 \in \Theta$ and $a_1, a_2\in \R$.  Let
  $\theta = a_1 \theta_1 + a_2 \theta_2$.  Then $\theta \in \Theta$.
  Moreover, let $c^{\theta_1}_{\gamma|_D}$ for $\gamma|_D \in \Omega$ and
  $c^{\theta_2}_{\gamma|_D}$ for $\gamma|_D \in \Omega$ be the families of
  constants that satisfy~\eqref{eq:additivity-equation-for-move} for
  $\theta_1$ and $\theta_2$, respectively.  Then
  $c^{\theta}_{\gamma|_D} = a_1 c^{\theta_1}_{\gamma|_D} + a_2
  c^{\theta_2}_{\gamma|_D}$
  for $\gamma|_D \in \Omega$ is a family of constants that
  satisfy~\eqref{eq:additivity-equation-for-move} for~$\theta$.
\end{proof}

\subsection{Join-closed semigroup $\Omegaresp/$ of moves respected by
  given functions}
\label{s:Omegaresp}

\begin{definition}\label{def:Omegaresp}
  For a function $\theta\colon \dom(\theta) \to \R$, we denote the ensemble of moves respected
  by~$\theta$ as
  \begin{displaymath}
    \Omegaresp{\theta} = \bigl\{\, \gamma|_D \in \FullMoveSemigroup
    \bigst D, \gamma(D) \subseteq \dom(\theta),\;
    \text{$\exists c^{\theta}_{\gamma|_D} \in \R$
      s.t.~\eqref{eq:additivity-equation-for-move} holds} \,\bigr\}.
  \end{displaymath}
  (Clearly $\Omegaresp{\theta}$ is the largest move ensemble that $\theta$
  respects.)  For a space $\Theta'$ of functions, we denote
  $\Omegaresp{\Theta'} = \bigcap_{\theta\in\Theta'} \Omegaresp{\theta}$.
\end{definition}

\begin{theorem}
  \label{thm:join-semigroup}
  \label{thm:additivity-equation-for-move}
  Let $\Omega$ be a move ensemble.
  If a function $\theta$ respects $\Omega$, then $\theta$
  respects the joined semigroup $\joinsemi{\Omega}$.  
\end{theorem}

To prove this, we use the following lemma. 

\begin{lemma}
Let $\mathfrak{I}$ be a collection of open intervals that cover the open interval $(l,u)$. If a function $g$ is constant over each interval $I$ from the collection $\mathfrak{I}$, then $g$ is constant over $(l,u)$.
\label{lemma:constant-over-open-cover}
\end{lemma}
\tred{In the FULL version, this one should be rewritten using posets -- for generality to
  $\R^k$. Note this lemma is pretty much an equivalent definition of 'connected set'.}
\begin{proof}
Let $m = \frac{l+u}{2}$ and $a = g(m)$. Consider the interval $J = \{\,y \in (l,m) \mid g(x)=a \text{ for all } x \in [y, m]\,\}$. Since $m$ is contained in some open interval $I \in \mathfrak{I}$ and $g(x)=a$ for $x \in I$, we know that $J$ is non-empty. Let $l'=\inf J$. We now show that $l = l'$. Suppose that $l \neq l'$. Then there exists an open interval $I \in \mathfrak{I}$ such that $l' \in I$, and $g$ is constant over $I$. Since $I \cap J \neq \emptyset$ and $g(x)=a$ for $x \in J$, we have that $g(x)=a$ for $x\in I$, a contradiction to $l'=\inf J$. Hence $g(x) = a$ for all $l < x \leq m$. Similarly, one shows that $g(x) = a$ for all $m \leq x < u$. Therefore, $g$ is constant over $(l,u)$.
\end{proof}

\begin{proof}[Proof of \autoref{thm:additivity-equation-for-move}]
Let $\gamma|_D \in \joinsemi{\Omega}$. 
Thus, there exists a collection $\mathfrak{I}$ of open intervals, such that  $D = \bigcup_{I \in \mathfrak{I}} I$ and $\gamma|_I \in\semi{\Omega}$ for each $I \in \mathfrak{I}$. 

Define $g(x) = \theta(\gamma(x))- \chi(\gamma)\theta(x)$ for $x \in D$.
We first show that $g$ is constant over each interval $I \in \mathfrak{I}$.
Let $I \in \mathfrak{I}$.  Since $\gamma|_I \in \semi{\Omega}$, we can write it in the form $\gamma|_{I} = \gamma_{k}|_{D_{k}}\circ \gamma_{k-1}|_{D_{k-1}} \circ\dots\circ\gamma_{1}|_{D_{1}}$, where $\gamma_{1}|_{D_{1}}, \gamma_{2}|_{D_{2}}, \dots, \gamma_{k}|_{D_{k}} \in \Omega$.
Let $x_0 \in I$ and denote $x_i = \gamma_i(x_{i-1})$ for $i=1,2,\dots,k$.  Then, $x_i \in D_{i+1}$ for $i=0,1,\dots,k-1$, and $x_k =\gamma|_{I}(x_0)= \gamma(x_0)$. 
Since $\theta$ respects $\Omega$, for $i=1,2,\dots,k$, we have that
\[\theta(x_i)=\chi(\gamma_i)\theta(x_{i-1})+c^{\theta}_ i,\] where the constants $ c^{\theta}_ i$ are  independent of the choice of $x_0 \in I$.  We also know that $\chi(\gamma) = \chi(\gamma_1) \chi(\gamma_2) \dots \chi(\gamma_k)$. Therefore, 
\begin{align*}
g(x_0) &= \theta(\gamma(x_0))- \chi(\gamma)\theta(x_0) \\
  &= \theta(x_k) - \chi(\gamma_k)\chi(\gamma_{k-1})\dots \chi(\gamma_1)\theta(x_0) = \sum_{j=1}^{k} {\left(
    \prod_{i=j+1}^{k}{\chi(\gamma_i)} \right) c_j^{\theta}}
\end{align*}
is constant  for $x_0 \in I$.

Then, it follows from \autoref{lemma:constant-over-open-cover} that $g$ is constant over $D$.
\end{proof}

\begin{corollary}
  For a function $\theta$, the ensemble $\Omegaresp{\theta}$ defined
  in~\autoref{def:Omegaresp} is a join-closed move semigroup. 
  The same holds for the
  ensemble $\Omegaresp{\Theta'}$, where $\Theta'$ is a space of functions.
\end{corollary}

\section{Kaleidoscopic joined ensembles and bounded functions.  Finite
  presentations by moves and components}
\label{sec:directly-covered-intervals}

\subsection{Cauchy--Pexider functional equation $f(x) + g(y) = h(x+y)$}
\label{s:cauchy-pexider}

Recall from \autoref{s:space-equivariant} that move ensembles encode systems
of functional equations.  We now bring a first result on functional equations
to use.  The following result on the Cauchy--Pexider functional equation on
bounded domains appeared in \cite[Theorem 4.3]{igp_survey}.  Here we state it
for functions of a single real variable.  It is a variant of the
Gomory--Johnson interval lemma, which has been used throughout the extreme
functions literature. Note that it requires a weak assumption regarding the
function space. Boundedness is sufficient; see \cite{igp_survey} for a more
detailed discussion.
\begin{lemma}[Convex additivity domain lemma]
\label{lem:Pexider}
Let $f,g,h\colon \R\to\R$ be bounded functions and let $E \subseteq \R^2$ be
open, convex, and bounded.  Suppose that
$$
f(x) + g(y) = h(x+y) \quad \text{ for all } (x,y) \in E.
$$
Define the projections
\begin{displaymath}
  p_1(x,y) = x, \quad p_2(x,y) = y, \quad p_3(x,y) = x+y
\end{displaymath}
as functions  from $\R^2$ to $\R$.
Then $f,g,h$ are affine with the same slopes on the domains $p_1(E), p_2(E), p_3(E)$, respectively.
\end{lemma}

\subsection{Kaleidoscopic move ensembles}

When we are only interested in \emph{bounded} functions that respect a move
ensemble~$\Omega$, then it follows from \autoref{lem:Pexider} that we can
replace $\Omega$ by a move ensemble $\Omega^{\mergeplus/}$ with more
convenient properties.

\tred{Is \mergeplus/ 'all local information from the function space'? (does this
  notion make sense)}

\begin{lemma}
\label{lem:IL-Moves}
Let $\theta\colon \R\to\R$ be a bounded function.  
Let  $D, I \subseteq \R$ be open intervals.  The following are equivalent:
\begin{enumerate}
\item $\theta$ respects $\tMovesOfGraph(D \times I)$,
\item $\theta$ respects $\rMovesOfGraph(D \times I)$,
\item $\theta$ respects $\MovesOfGraph(D \times I)$,
\item $\theta$ is affine on $D$ and $I$ with the same slope.\label{item:IL-Moves:affine}
\end{enumerate}
\end{lemma}
\begin{proof}
We first show that (1) implies (\ref{item:IL-Moves:affine}).
By assumption, the function $\theta$ satisfies
equation~\eqref{eq:additivity-equation-for-move} for all $\tau_t|_{D_t}$,
where $t \in \{\, y-x \st x \in D,\; y \in I\,\}$ and
$D_t = \{\, x \in D \st x+t \in I \,\}$. Thus, there exists a function
$c\colon I + (-D) \to \R$ such that
$$
\theta(x+t) = \theta(x) + c(t) \qquad\text{for all $(x, x+t) \in D \times I$}.
$$
The function $c$ is bounded because $\theta$ is.
Then, by Lemma~\ref{lem:Pexider} with $f = h = \theta$ and $g = c$, we have that
$\theta$ affine on $D$ and~$I$ with the same slope.
The proofs that each of (2) and (3) 
implies (\ref{item:IL-Moves:affine}) are similar; we omit them.

Now we show that (\ref{item:IL-Moves:affine}) implies (1).
Fix $t = y - x$ for some $x\in D, y \in I$.  
Since $\theta$ is affine on $D$ and $I$ with the same slope, there exist scalars $a,b,b'$ such that $\theta(x) = a \cdot x + b$ for all $x \in D$ and $\theta(x) = a \cdot x + b'$ for all $x \in I$.  But then for all $x \in D$ such that $x+t \in I$, we have that $\theta(x+t) - \theta(x) = a\cdot t$, which is constant.  Therefore, $\theta$ respects $\tau_t|_{D_t}$. 
Again the proofs that (\ref{item:IL-Moves:affine}) also implies (2) and (3) is
similar and we omit them.
\end{proof}

\label{s:kaleido-def}
Motivated by these results, we make the following definitions. 
\begin{definition}
  A move ensemble $\Omega^{\mergeplus/}$ is a \emph{\mergeplus/ joined ensemble} if it satisfies 
  \eqref{axiom:restrict}, \eqref{axiom:join}, and the following axiom:
  \begin{equation}
    \tag{kaleido} \label{axiom:translation-reflection}
    \begin{gathered}[t]
      \text{ for open intervals }D,I \subseteq \R\\
      \tMovesOfGraph(D\times I) \subseteq \Omega^{\mergeplus/}
      \text{ if and only if } \rMovesOfGraph(D\times I) \subseteq \Omega^{\mergeplus/}.
    \end{gathered}
  \end{equation}
\end{definition}

\subsection{Covered intervals, connected covered components}
\label{s:covered-intervals-components}

\begin{definition}
  For a \mergeplus/ joined ensemble $\Omega^{\mergeplus/}$ and an open interval $D$
  such that $\MovesOfGraph(D\times D) \subseteq \Omega^{\mergeplus/}$, we say
  that $D$ is a \emph{covered interval} in $\Omega^{\mergeplus/}$.
\end{definition}

Let $\Gamma^{\mergeplus/}$ be a \mergeplus/ joined move semigroup.  For two open intervals
$D_1, D_2$, if $$\MovesOfGraph\bigl((D_1 \cup D_2)\times(D_1 \cup D_2)\bigr) \subseteq
\Gamma^{\mergeplus/},$$ then we 
say that both $D_1$ and $D_2$ are covered intervals in the same
\emph{connected covered component} of $\Gamma^{\mergeplus/}$.  (Here the word
``connected'' does not refer to the topology of~$\R$, in contrast
to~\autoref{sec:semigroup:open}.) 
It follows from
\autoref{cor:open-sets-of-moves} that this is an equivalence relation.
However, we want to define the notion of a 
\emph{connected covered component} also for \mergeplus/ joined ensembles
$\Omega^{\mergeplus/}$ that are not semigroups.  In this case there is no
equivalence relation (transitivity fails), but we still use the word
``components'' in the following definition.
\begin{definition}
  Let $\Omega^{\mergeplus/}$ be a \mergeplus/ joined ensemble.  Let $C$ be an open
  set such that $\MovesOfGraph(C\times C) \subseteq \Omega^{\mergeplus/}$.
  Then $C$ is called a \emph{connected covered component} of
  $\Omega^{\mergeplus/}$.  Any two covered intervals $D_1, D_2 \subseteq C$ are said
  to be \emph{connected} by the component~$C$.
\end{definition}
The connected covered components of~$\Omega^{\mergeplus/}$ are partially
ordered by set inclusion.  The maximal elements in this partial order suffice
to describe all covered intervals.
\begin{corollary}
  \label{cor:affine-on-components}
  Let $\theta$ be a bounded function.  Suppose $\theta$ respects
  a \mergeplus/ joined ensemble $\Omega^{\mergeplus/}$.  Let $C$ be a connected
  covered component of~$\Omega^{\mergeplus/}$.
  Then $\theta$ is affine on all
  open intervals in $C$ with a common slope.
\end{corollary}
\begin{proof}
  Let $D,I \subseteq C$ be open intervals.  Then
  $D \times I \subseteq C\times C$, and hence $\theta$ respects
  $\MovesOfGraph(D\times I)$.  By
  \autoref{lem:IL-Moves}\,(\ref{item:IL-Moves:affine}), $\theta$ is affine on
  $D$ and $I$ with the same slope.
\end{proof}

(Later in \autoref{sec:perturbation_space}, we will also consider so-called
connected uncovered components.)

\subsection{Presentations by moves $\Omega^{\fin/}$ and components $\CC=\{\texorpdfstring{C_1}{C\unichar{"2081}},\dots,\texorpdfstring{C_k}{C\unichar{"1D63}}\}$}
\tred{Could also be a separate section later}

Now we are prepared to define a convenient finite presentation for a large
class of \mergeplus/ joined ensembles, which we announced in
\autoref{rem:lines+boxes-nonpurple}.

\begin{definition}
  Take a finite list of connected covered components
  $\CC = \{\Comp[1], \dots, \Comp[k]\}$, where each $\Comp[i]$ is a
  finite union of disjoint open intervals.
  Define
  \begin{align*}
    \moves(\CC) &= \bigcup_{i=1}^k \moves(\Comp[i] \times \Comp[i]) \\
                &= \bigl\{\, \gamma|_D \in \FullMoveSemigroup \bigst D, \gamma(D) \subseteq \Comp[i] \text{ for some } i =1,
                  \dots, k  \,\bigr\}.  
  \end{align*}
\end{definition}
The graph $\graph(\moves(\CC))$ is a union of open rectangles.  See
\autoref{fig:connected-covered-components} for a visualization.  We plot
the components with different colors. 
\begin{figure}[t]
  \centering
  $\vcenter{\hbox{\includegraphics[height=.4\linewidth]{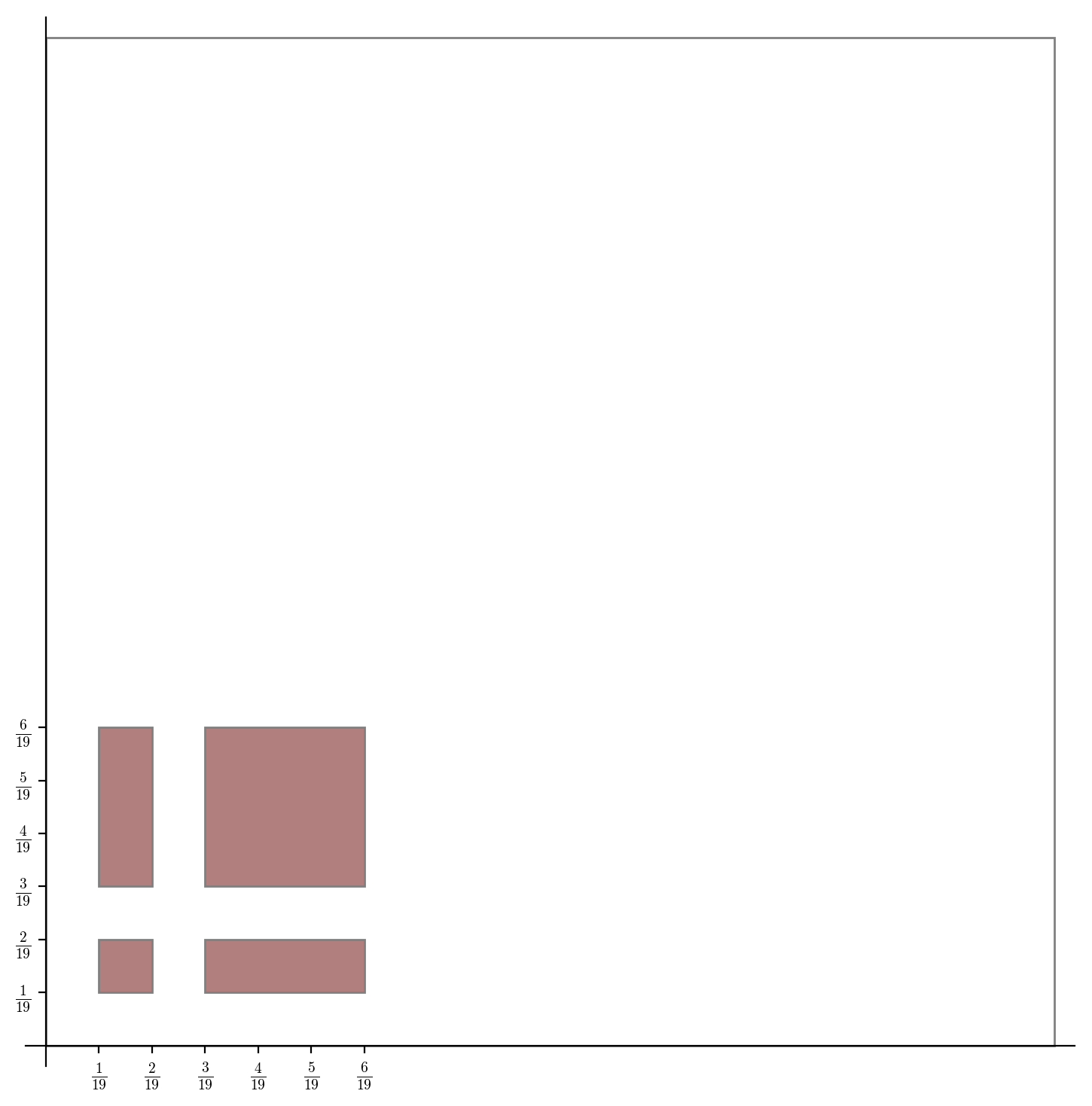}}}
  \qquad
  \vcenter{\hbox{\includegraphics[height=.4\linewidth]{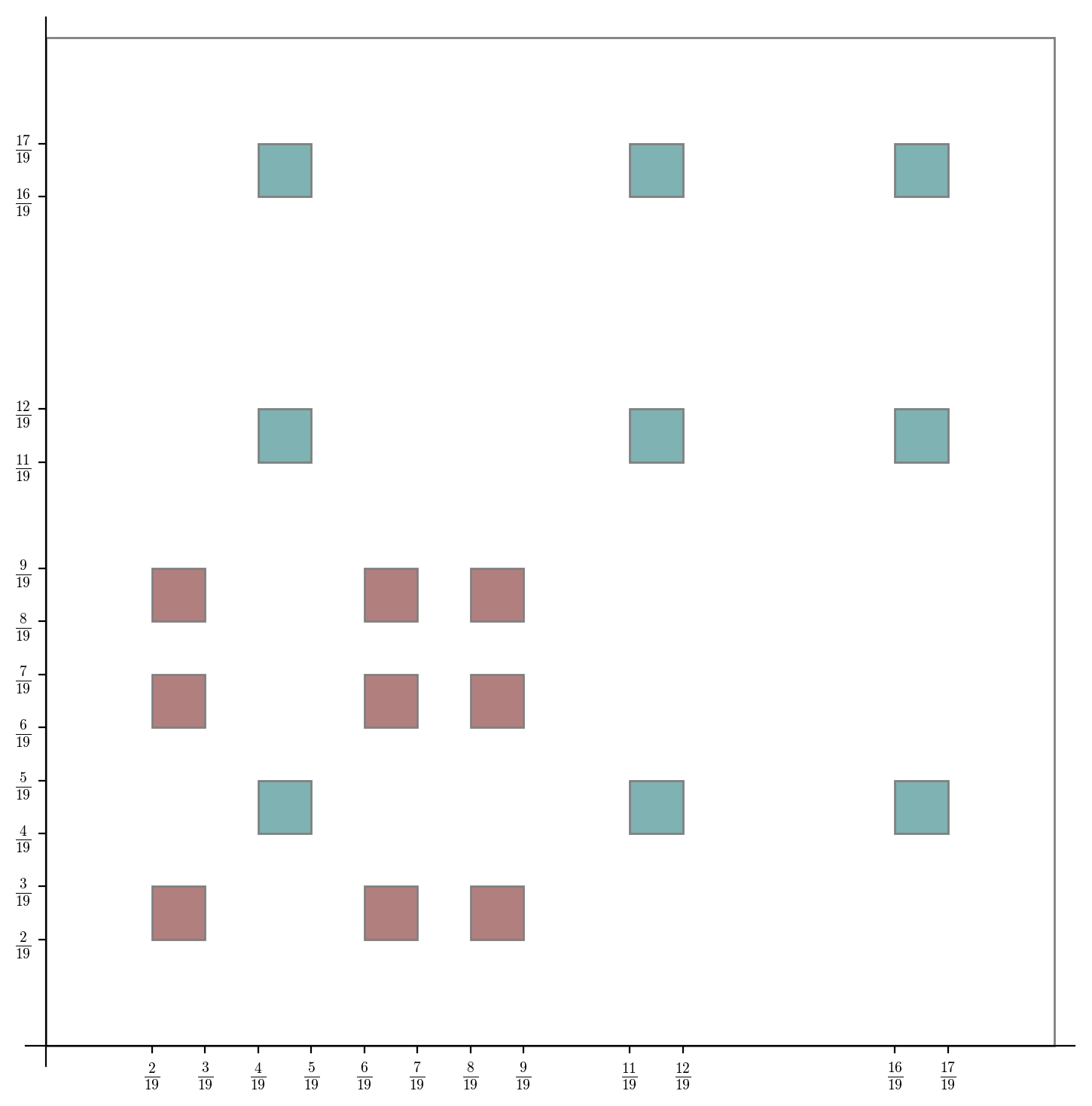}}}$
  \caption[Move ensemble $\moves(\C)$ from connected covered components $\C$]
  {Move ensemble $\moves(\C)$ from connected covered components $\C$. \emph{Left,}
    $\C = \{ C_1 \}$ (one component), where
    $C_1 = (\frac1{19}, \frac2{19}) \cup (\frac3{19}, \frac6{19})$, shown in \emph{red}.
    \emph{Right,} $\C = \{ C_1, C_2 \}$ (two components), where
    $C_1 = (\frac2{19}, \frac3{19}) \cup (\frac6{19}, \frac7{19}) \cup
    (\frac8{19}, \frac9{19})$
    is shown in \emph{red}
    and
    $C_2 = (\frac4{19}, \frac5{19}) \cup (\frac{11}{19}, \frac{12}{19}) \cup
    (\frac{16}{19}, \frac{17}{19})$
    is shown in \emph{cyan}.}
  \label{fig:connected-covered-components}
\end{figure}

Note that any ensemble of the form $\moves(\CC)$ or
$\Omega^{\fin/} \cup \moves(\CC)$, where $\Omega^{\fin/}$ is a finite move
ensemble, satisfies \eqref{axiom:restrict} and
\eqref{axiom:translation-reflection}, but is not necessarily join-closed.  To
make a \mergeplus/ joined ensemble, we use the
following.

\label{s:omega-fin-notation}
\begin{definition}
  For any finite move ensemble $\Omega^{\fin/}$ and a finite list $\C$
  of connected covered components, define 
  \begin{displaymath}
    \joinmoves(\Omega^{\fin/}, \CC) = \join{\Omega^{\fin/} \cup \moves(\CC)}.
  \end{displaymath}
  If $\Omega^{\fin/} = \emptyset$, we simply write $\joinmoves(\CC)$. 
\end{definition}

\begin{definition}
  The ordered pair $(\Omega^{\fin/}, \CC)$ is said to be a \emph{finite
    presentation} (by moves $\Omega^{\fin/}$ and components~$\CC$) of the
  \mergeplus/ joined ensemble $\joinmoves(\Omega^{\fin/},\CC)$.
\end{definition}

\begin{corollary}
\label{cor:affine-on-component-finite-presentation}
  Let $\theta$ be a bounded function.  Suppose $\theta$ respects
  a move ensemble $\Omega^{\mergeplus/}$ that has the finite presentation
  $(\Omega^{\fin/},\CC)$. 
  Then $\theta$ is affine on all
  intervals in $\C$ and shares a common slope on all intervals of each
  component $\Comp[i]$ of $\C$. 
\end{corollary}
\begin{proof}
  This is a restatement of \autoref{cor:affine-on-components}.
\end{proof}

It is clear that these presentations are not unique, which motivates 
the next subsection.

\subsection{Finite presentation in reduced form $(\Omega^{\red/}, \CC)$}
\label{s:finite-presentation-normal-form}

\begin{figure}
  \centering
  {\Huge $
  \begin{aligned}
    \MOVEDIAG{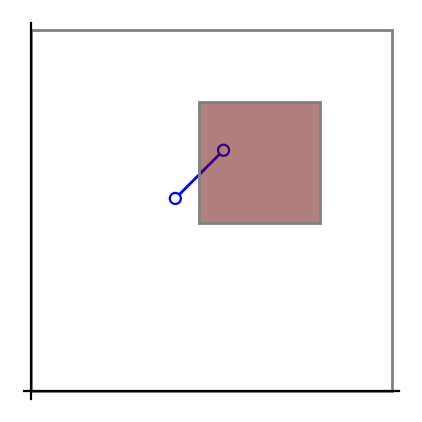}
    &\mapsto
    \MOVEDIAG{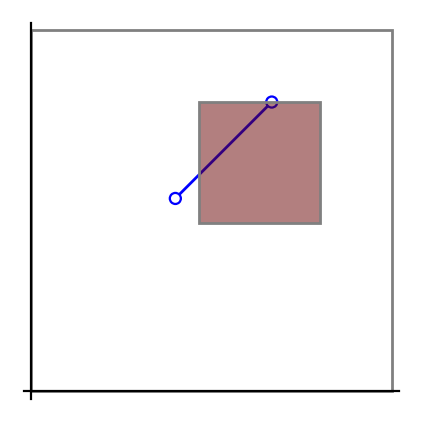}
    \\
    \MOVEDIAG{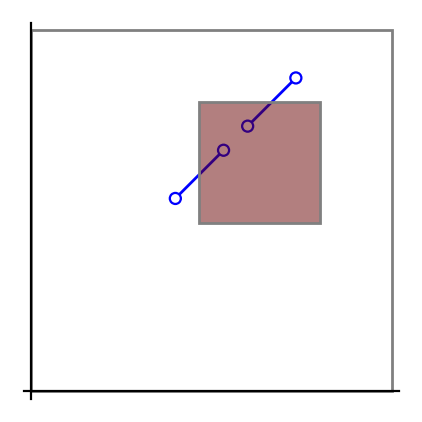}
    &\mapsto
    \MOVEDIAG{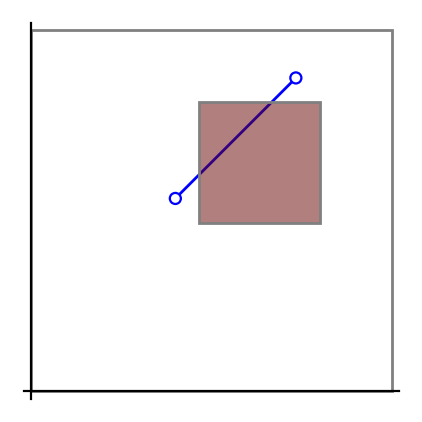}
    \\
    \MOVEDIAG{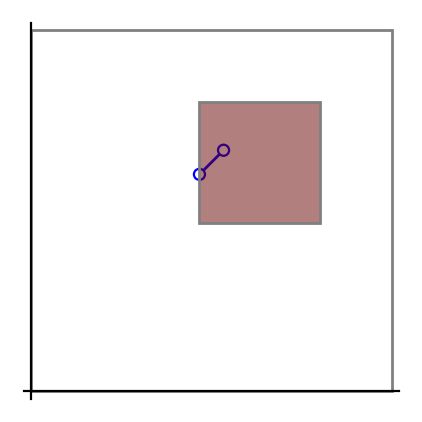}
    &\mapsto
    \MOVEDIAG{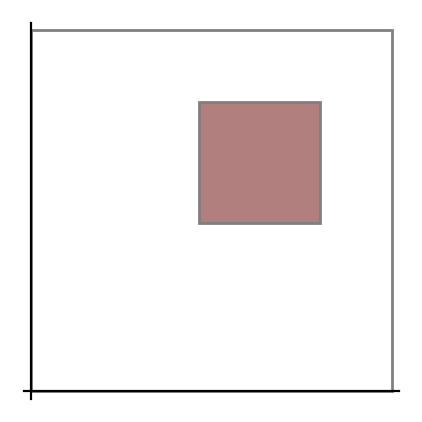}
  \end{aligned}
  $}
  \caption[Finite presentation in reduced form. 
    \emph{Left,} finite presentations $(\Omega^{\fin/}, \CC)$ of \mergeplus/ joined ensembles~$\Omega^{\mergeplus/}$.
    \emph{Right,} finite presentations $(\Omega^{\red/}, \CC)$ in reduced form of the
    same ensembles]
    {Finite presentation in reduced form. 
    \emph{Left,} finite presentations $(\Omega^{\fin/}, \CC)$ of \mergeplus/ joined ensembles~$\Omega^{\mergeplus/}$.
    \emph{Right,} finite presentations $(\Omega^{\red/}, \CC)$ in reduced form of the same ensembles.
    (a) A move poking into a component is extended to become a maximal move
    of~$\Omega^{\mergeplus/}$. 
    (b) Two restrictions of the same move are extended to become a maximal
    move of~$\Omega^{\mergeplus/}$. 
    (c) A move that lies completely in a component is removed.}
  \label{fig:reduce_moves_by_components}
\end{figure}

\label{s:omega-red-notation}
\begin{definition}
  A finite presentation $(\Omega^{\red/}, \CC)$ of a \mergeplus/ joined ensemble
  $\Omega^{\mergeplus/}$ is said to be in \emph{(long) reduced form} if the following
  holds:
  \begin{equation}
    \tag{reduce}\label{axiom:reduced}
    \Omega^{\red/} \subseteq \maxdom(\Omega^{\mergeplus/}) \setminus \joinmoves(\CC),
  \end{equation}
  that is, each move $\gamma|_D\in\Omega^{\red/}$ is maximal in
  $\Omega^{\mergeplus/}$ with respect to the restriction partial
  order~$\subseteq$, 
  and the graph $\graph(\gamma|_D)$ is not covered by the
  union of open rectangles $C_i\times C_i$, $C_i \in \C$.
\end{definition}  

\begin{lemma}
  If a \mergeplus/ joined ensemble $\Omega^{\mergeplus/}$ has a finite presentation
  $(\Omega^{\fin/}, \CC)$, then there is a unique finite
  ensemble~$\Omega^{\red/}$ such that $(\Omega^{\red/}, \CC)$ is in reduced
  form and $\Omega^{\mergeplus/} = \joinmoves(\Omega^{\red/}, \CC)$.  
\end{lemma}

\autoref{fig:reduce_moves_by_components} illustrates the operation of going
from a finite presentation to a reduced presentation of the same ensemble.

\begin{remark}
  As the examples in \autoref{fig:reduce_moves_by_components} illustrate,
  the domains of moves in $\Omega^{\fin/}$ may be extended.
\end{remark}

\subsection{Finite presentations of generating ensembles of move semigroups}

Move ensembles have a crucial r\^ole as generating sets of move semigroups.
We now describe an operation that changes the generating ensemble, but preserves
the move semigroup that is generated by it.

\begin{lemma}[Extend component by move]
\label{lem:indirectly_covered_from_move}
Let $\CC$ be a list of connected components and let $\Omega$ be a move ensemble such that $\moves (\CC) \subseteq \Omega$.  If $\gamma|_D \in \Omega$ and $D \subseteq \Comp[i]$ for some $\Comp[i] \in \CC$, then $\moves(\CC') \subseteq \semi{\Omega}$, where $\Comp[i]' = \Comp[i] \cup \gamma(D)$ and all other components of $\CC'$ are the same as $\CC$.
\end{lemma}
See \autoref{fig:extend_components_by_moves} for an illustration.

\begin{proof}
Let $x \in \Comp[i]$, $z \in \gamma(D)$, and $y = \gamma^{-1}(z) \in D$.  Since $x \in \Comp[i]$, $x$ is in the domain of moves $\tau_0$ and $\rho_0$ in $\Omega$.  Thus, we can both translate and reflect $x$ to $z$ by
$$x \xmapsto{\tau_{t_{y - x}} } y  \xmapsto{\gamma} z,$$
and
$$x \xmapsto{\rho_{r_{0}}} x \xmapsto{\tau_{t_{y - x}} } y  \xmapsto{\gamma} z.$$
Note that which one above is a translation or reflection depends on the character $\chi(\gamma)$.
\end{proof}

\begin{figure}
  \centering
  \Huge $
  \begin{aligned}
    \MOVEDIAG{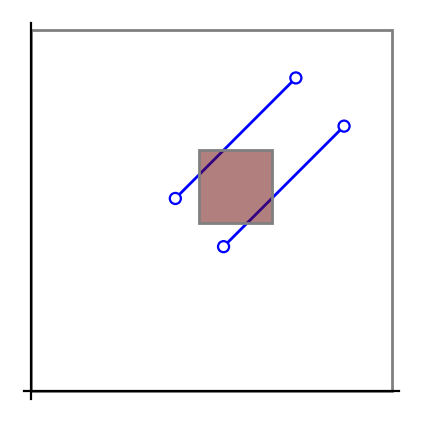} 
    &\xmapstosemi
    \MOVEDIAG{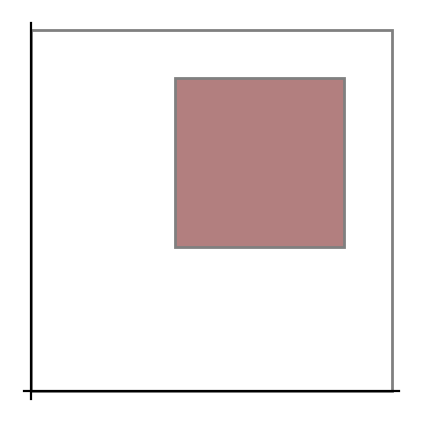}
    \\
    \MOVEDIAG{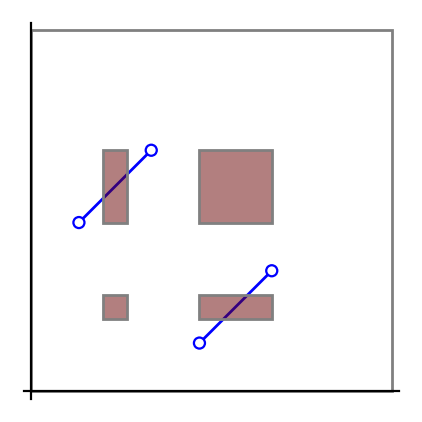} 
    &\xmapstosemi
    \MOVEDIAG{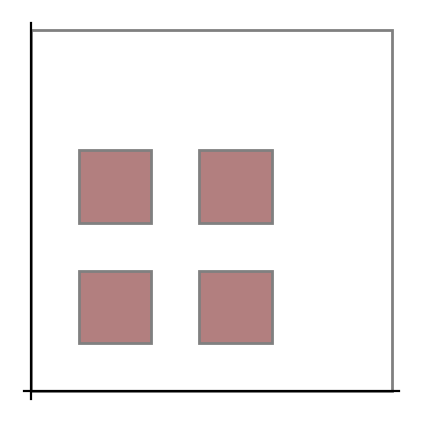}
    \\
    \MOVEDIAG{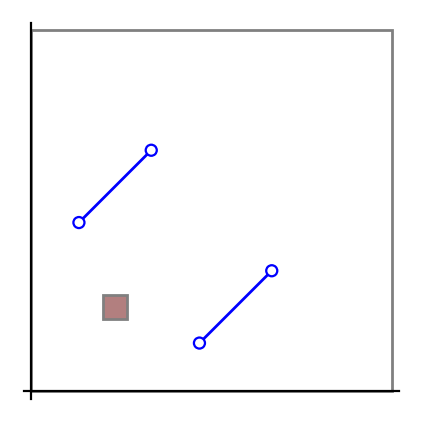} 
    &\xmapstosemi
    \MOVEDIAG{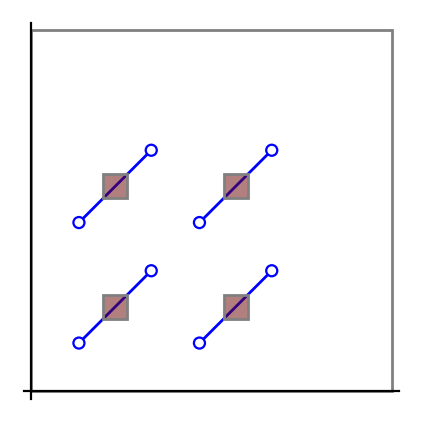}
    \\
    \MOVEDIAG{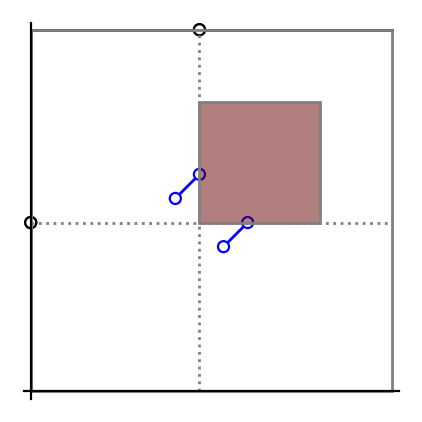} 
    &\xmapstosemi
    \MOVEDIAG{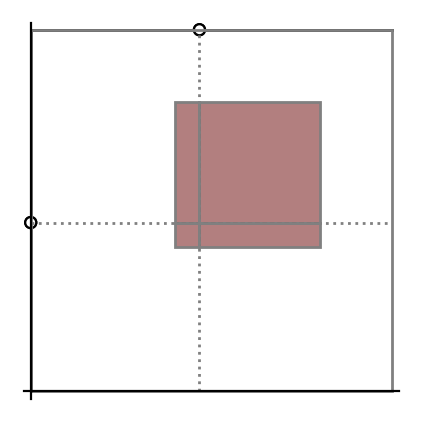}
  \end{aligned}    
  $
  \caption{
    Extending components by moves,
    \autoref{lem:indirectly_covered_from_move}.
    \emph{Left,} reduced finite presentations of a kaleidoscopic joined
    ensemble $\Omega^{\mergeplus/}$. \emph{Right,} reduced finite
    presentations of $\semi{\Omega^{\mergeplus/}}$.}
  \label{fig:extend_components_by_moves}
\end{figure}

\section{Limit-closed ensembles and continuous functions. Closed move semigroups}
\label{sec:moves-closure}
\label{sec:order-defined}

Let $A \subseteq \R$ be an open set.  We now consider the space
$\BoundedContinuous{A}$ of bounded continuous functions on $A$.
For $\BoundedContinuous{A}$, some notions of convergence of moves are natural
to study.

\subsection{Limit-closed move ensembles $\IsLimit{\Omega}$; closures
  $\limit{\Omega}$,
  $\arblim{\Omega}$}
\label{sec:limit-def}

\subsubsection{Convergence of unrestricted moves}
\begin{definition}
  A sequence $\{ \gamma^i \}_{i\in\N} \subseteq \FullMoveGroup$ of
  unrestricted moves \emph{converges} 
  \begin{enumerate}[\rm(a)]
  \item to an unrestricted translation $\tau_t \in \FullMoveGroup$ if all
    but finitely many $\gamma^i$ are translations $\tau_{t^i}$ and $t^i\to t$.
  \item to an unrestricted reflection $\rho_r \in \FullMoveGroup$ if all
    but finitely many $\gamma^i$ are reflections $\rho_{r^i}$ and $r^i\to r$.
  \end{enumerate}
\end{definition}

\subsubsection{Limits closure}

\begin{definition}
  We define the \emph{limits closure} $\limit{\Omega}$
  of a moves
  ensemble~$\Omega$ to be the smallest (by set inclusion) moves
  ensemble~$\IsLimit{\Omega}$ containing~$\Omega$ that satisfies the following axiom.
  \tblue{(Do NOT include \eqref{axiom:restrict}, \eqref{axiom:join} here, \eqref{axiom:extend} and \eqref{axiom:translation-reflection} comes later.)}
  \begin{equation}
    \tag{lim} \label{axiom:limits}
    \begin{gathered}[t]
      \text{Let $D$ be an open interval.}\\
      \text{If $\gamma^i\to\gamma$ and $\gamma^i|_D \in \bar{\Omega}$ for all $i$,}
      \ \text{then $\gamma|_D \in \bar{\Omega}$.}
    \end{gathered}
  \end{equation}
\end{definition}

We note that the domain $D$ is fixed for all moves in the sequence.  Thus, the
limits closure will in general not satisfy~\eqref{axiom:join} and~\eqref{axiom:inverse}.  Instead we can consider the following axiom%
\useless{s}.

\begin{definition}
  Define $\arblim{\Omega}$ to be the smallest moves ensemble $\IsLimit{\Omega}$ containing $\Omega$
  that satisfies the following axiom.
  \begin{equation}
    \tag{arblim} \label{axiom:arblim}
    \begin{gathered}[t]
    \text{If $\gamma^i \to \gamma$, $l^i \to l$, $u^i \to u$ and $\gamma^i|_{(l^i, u^i)} \in \bar{\Omega}$ for all
      $i$,}\\ \text{then $\gamma|_{(l,u)} \in \bar{\Omega}$.}
  \end{gathered}
  \end{equation}
\end{definition}
\tred{Perhaps a-lim (and s-lim) looks nicer?}
For our purposes, when considered together with \eqref{axiom:join}, the
notions turn out to be equivalent.

\begin{theorem}\label{thm:lim-symlim-equiv-1round}
  Let $\Omega^{\join/}$ be a join-closed move ensemble.
  Then $$\join{\limit{\Omega^{\join/}}}
  \useless{{}= \join{\symlim{\Omega^{\join/}}}}
  = \join{\arblim{\Omega^{\join/}}}.$$
\end{theorem}
\begin{proof}
It is clear that $\limit{\Omega^{\join/}}
\useless{{}\subseteq \symlim{\Omega^{\join/}}}
\subseteq \arblim{\Omega^{\join/}}$, so it suffices to show that
\begin{equation}
\label{eq:arblim-in-joinlim}
\arblim{\Omega^{\join/}} \subseteq \join{\limit{\Omega^{\join/}}}. 
\end{equation}  
Let $\tau_t|_{(l,u)} \in \arblim{\Omega^{\join/}}$. By \eqref{axiom:arblim},
there is a convergent sequence $\{\tau_{t^i}|_{(l^i, u^i)}\}_{i \in\N}$ of
moves  in $\Omega^{\join/}$ such that  $l^i \to l$, $u^i \to u$ and $t^i \to
t$. For every integer $j>\frac{2}{u-l}$, there exists a large integer $n_j$
such that for any $i\geq n_j$, we have $l_i < l+\frac{1}{j}$ and
$u-\frac{1}{j} < u_i$. Since $\Omega^{\join/}$ satisfies \eqref{axiom:join},
$\tau_{t^i} |_{D_j} \in \Omega^{\join/}$ for any $i\geq n_j$, where $D_j :=
(l+\frac{1}{j}, u-\frac{1}{j})$. Since $t_i \to t$, we have  $\tau_t|_{D_j}
\in \limit{\Omega^{\join/}}$ for every $j$, hence $\tau_t|_{(l,u)} \in
\join{\limit{\Omega^{\join/}}}$.  We showed that \eqref{eq:arblim-in-joinlim}
holds for translations. The proof for reflections is similar.
\end{proof}
\begin{theorem}
\label{thm:lim-symlim-equiv}
  Let $\Omega^{\join/}$ be a join-closed move ensemble.
  The following are equivalent.
  \begin{enumerate}
  \item $\Omega^{\join/}$ satisfies \eqref{axiom:limits}.
    \useless{\item $\Omega^{\join/}$ satisfies \eqref{axiom:symlim}.}
  \item $\Omega^{\join/}$ satisfies \eqref{axiom:arblim}.\label{th:lim-symlim-equiv:arblim}
\end{enumerate}
\end{theorem}

The proof is essentially the same and we omit it.

\tred{Should we nevertheless introduce notation jlim?}

\subsubsection{Respecting limits}

\begin{lemma}[Limits]
\label{lem:limits-of-moves}
Let $D$  be an open interval and let $\theta$ be continuous on~$D$.  If there
exists a sequence $\gamma^i \to \gamma$ such that $\theta$ 
respects $\gamma^i|_D$ for all $i$, then $\theta$ also respects $\gamma|_D$. 
\end{lemma}
\begin{proof}
  We prove the lemma for a sequence $t_i \to t$ such that $\theta$ respects
  the translations $\tau_{t_i}|_D$ for all $i$.  We will show that $\theta$
  also respects $\tau_t|_D$.

Since $\theta$ is continuous on $D$, $\theta$ is also continuous on $\tau_{t_i}(D)$ for all $i$.  
Fix $\bar x \in D$.  Since $t_i \to  t$, and $\bar x \in \intr(D)$ since $D$
is open, there exists an $\hat i$ such that for all $i \geq \hat i$, we have
$\bar x + t \in D + t_i$.  Hence, for a neighborhood $N_{\bar x}$ of~$\bar x$,
$\theta$ is continuous in $N_{\bar x} + t$. 
Now, for all $x \in N_{\bar x}$, 
$$
\theta(x + t) - \theta(x) = \lim_{t_i \to t} \theta(x + t_i) - \theta(x) = \lim_{i \to \infty} c^\theta_{\tau_{t_i}|_D}.
$$
Since the limit on the right-hand side is independent of $x$, we define
$c^\theta_{\tau_{t}|_{N_{\bar x}}}$ to be this limit.  Thus, $\theta$ respects
$\tau_t|_{N_{\bar x}}$.

Now the connected open set~$D$ is covered by the open neighborhoods $N_{\bar
  x}$ of each $\bar x \in D$.  It follows that 
$c^\theta_{\tau_{t}|_{N_{\bar x}}} = c^\theta_{\tau_{t}|_{N_{\bar x'}}}$ for
all $\bar x, \bar x' \in D$.  Therefore, $\theta$ respects $\tau_t|_D$.
Moreover, $\theta$ is continuous on $\tau_t|_D$.

The proof for a sequence of reflections is the same.
\end{proof}

\subsubsection{Limit-closed move semigroups}

\begin{lemma}
\label{lemma:symlimsemi-is-semi}
Let $\Gamma$ be a move semigroup. Then $\arblim{\Gamma}$ is also a move
semigroup.
\end{lemma}
\begin{proof}
It is clear that  $\arblim{\Gamma}$ satisfies \eqref{axiom:inverse}, as $\Gamma$  satisfies \eqref{axiom:inverse}. We now show that $\arblim{\Gamma}$ satisfies \eqref{axiom:composition}.  Let $\gamma_1|_{D_1}, \gamma_2|_{D_2} \in  \arblim{\Gamma}$ such that  $\gamma_1|_{D_1} \circ \gamma_2|_{D_2}$ is not an empty move. $\gamma_1|_{D_1}$ and $\gamma_2|_{D_2}$ are the \eqref{axiom:arblim} of sequences of moves $\{\gamma_1^i|_{D_1^i}\}_{i \in \N}$ and  $\{\gamma_2^i|_{D_2^i}\}_{i \in \N}$ in $\Gamma$. Since $\Gamma$ satisfies \eqref{axiom:composition}, $\gamma^i|_{D^i} := (\gamma_1^i|_{D_1^i})\circ (\gamma_2^i|_{D_2^i}) \in \Gamma$ for every $i$. The \eqref{axiom:arblim} of the sequence $\{\gamma^i|_{D^i}\}_{i \in \N}$ is $\gamma_1|_{D_1} \circ \gamma_2|_{D_2}$. Thus, we obtain that  $\gamma_1|_{D_1} \circ \gamma_2|_{D_2} \in  \arblim{\Gamma}$. This show that  $\arblim{\Gamma}$ is a semigroup.
\end{proof}

\begin{lemma}
\label{lemma:joinlimjoinsemi-is-semi}
Let $\Gamma^\vee$ be a join-closed semigroup. Then $\join{\limit{\Gamma^{\join/}}}
\useless{{}= \join{\symlim{\Gamma^{\join/}}}}
= \join{\arblim{\Gamma^{\join/}}}$ is a semigroup.
\end{lemma}
\begin{proof}
It follows from \autoref{lemma:symlimsemi-is-semi}, \autoref{lemma:joinsemi-is-semi} and \autoref{thm:lim-symlim-equiv-1round}.
\end{proof}

\begin{theorem}[Limits imply components]
\label{thm:dense-boxes-in-joinlim}
Let $\Gamma^\vee$ be a join-closed move semigroup. Assume that $\gamma|_D$ is the
limit move (in the sense of $\limit/$ or \useless{$\symlim/$ or }$\arblim/$)
of a sequence $\{\gamma^i|_{D^i}\}_{i \in \N}$ of moves in $\Gamma^\vee$ with
$\gamma^i \neq \gamma$ for every $i$. Let $I = \gamma(D)$.  
Then
the following holds.
\begin{enumerate}
\item If $\gamma$ is a translation, then $\tMovesOfGraph((D \cup I)\times(D \cup I )) \subseteq \join{\limit{\Gamma^\vee}}$. 
\item If $\gamma$ is a reflection, then 
\[\rMovesOfGraph((D \times I) \cup (I \times D)), \tMovesOfGraph((D \times D) \cup (I \times I))  \subseteq \join{\limit{\Gamma^\vee}}.\]
\end{enumerate}
\end{theorem}

\begin{proof}
Let $D = (l, u)$. If a sequence $\{\gamma^i|_{D^i}\}_{i \in \N}$ of moves in
$\Gamma^\vee$ with $\gamma^i \neq \gamma$  converges to $\gamma|_D$ in the
sense of \useless{$\symlim/$ or }%
$\arblim/$, then $\gamma^i|_{D^i \cap (l+\epsilon,
  u-\epsilon)} \to \gamma|_{(l+\epsilon, u-\epsilon)}$ in the sense of
$\limit/$ for any small $\epsilon>0$. Thus, it suffices to prove the statement
for a limit move $\gamma|_D$ in the sense of $\limit/$; the statement for  \useless{$\symlim/$ and }%
$\arblim/$ follows from \autoref{lemma:joinlimjoinsemi-is-semi} and continuation.

We first show that $\tMovesOfGraph(D \times D)\subseteq \join{\limit{\Gamma^\vee}}$. Let $\epsilon >0$ be an arbitrary small number. Since $\gamma|_D$ is a limit move, there exist $\gamma^i|_D, \gamma^j|_D \in \Gamma^\vee$ in the convergent sequence such that  the constants $\gamma -\gamma^i$ and $\gamma -\gamma^j$ have the same sign, and $0 < \gamma^j - \gamma^i <\epsilon$.  
Let $\delta =  \gamma^j - \gamma^i$ and $D^1 = (l,u)\cap (l-\delta, u-\delta)$. We notice that $(\gamma^i|_D)^{-1} \circ \gamma^j|_D=\tau_\delta|_{D^1}$ when $\gamma$ is a translation, and $(\gamma^j|_D)^{-1} \circ \gamma^i|_D=\tau_\delta|_{D^1}$ when $\gamma$ is a reflection. Therefore, $\tau_\delta|_{D^1}  \in \Gamma^\vee$. 
Let $D^k := (l,u)\cap (l-k\delta, u-k\delta)$ for $k \in \Z$. For $k \geq 1$,  $\tau_{k\delta}|_{D^k}$ is the $k$ times composition of $\tau_\delta |_{D^1}$, hence it is in $\Gamma^\vee$. For $k=-1$, $\tau_{-\delta}|_{D^{-1}} =  (\tau_\delta |_{D^1})^{-1} \in \Gamma^\vee$.  For $k \leq -2$, $\tau_{k\delta}|_{D^k}$ is the $-k$ times composition of $\tau_{-\delta}|_{D^{-1}}$, and hence is in $\Gamma^\vee$. Finally, for $k=0$, we have $(\tau_\delta |_{D^1}) \circ (\tau_{-\delta}|_{D^{-1}}),  (\tau_{-\delta}|_{D^{-1}}) \circ (\tau_\delta |_{D^1})\in \Gamma^\vee$, so their join $\tau_0|_{D^0}$ is also in  $\Gamma^\vee$. Therefore, for every $k \in \Z$ such that $D^k$ is not empty, we have $\tau_{k\delta}|_{D^k} \in \Gamma^\vee$. By letting $\epsilon \to 0$, we obtain that $\tMovesOfGraph(D \times D)\subseteq \join{\limit{\Gamma^\vee}}$.

Since $\gamma|_D \in \limit{\Gamma^\vee} \subseteq \join{\limit{\Gamma^\vee}}$ and $\join{\limit{\Gamma^\vee}}$ is a semigroup by \autoref{lemma:joinlimjoinsemi-is-semi}, we have that $\tMovesOfGraph(D \times I)  \subseteq \join{\limit{\Gamma^\vee}}$ when $\gamma$ is a translation, and $\rMovesOfGraph(D \times I)  \subseteq \join{\limit{\Gamma^\vee}}$ when $\gamma$ is a reflection.

The other two subsets follow from
applying the above argument to $(\gamma|_D)^{-1}$ instead of $\gamma|_D$.
\end{proof}

\subsection{Continuous domain extension $\joinextend{\Omega}$}

Next we introduce a topological version of axiom~\eqref{axiom:join}.

\subsubsection{Extended move ensembles $\IsJoinExtend{\Omega}$}

\begin{definition}
  Let $\Omega$ be a move ensemble with $\dom(\Omega), \im(\Omega) \subseteq
  A$, where $A \subseteq \R$ is an open set.  Then the \emph{extended move
    ensemble} $\joinextend{\Omega}$ of~$\Omega$ is defined to be the smallest
  set $\IsJoinExtend{\Omega}$ containing~$\Omega$ that satisfies the
  following axiom
  \begin{equation}\tag{extend${}_A$}\label{axiom:joinextend}
    \begin{gathered}[t]
      \text{Let $\gamma \in \FullMoveGroup$ and
        $D$ empty or an open interval.}\\
      \text{If there is an ensemble
        $\{ \gamma|_{D^i} \}_{i \in \mathfrak I} \subseteq
        \IsJoinExtend{\Omega}$ such that} \\
      \text{
        $D \subseteq \cl(\textstyle\bigcup_{i\in \mathfrak I} D^i) \cap A \cap
        \gamma^{-1}(A)$, then $\gamma|_D \in \IsJoinExtend{\Omega}$.}
    \end{gathered}
  \end{equation}
\end{definition}

\tred{Notation: Do we want the subscript? Otherwise, remove it from ALL
  notations consistently. (Yuan prefers to remove it.  Matthias prefers to
  keep it.)}
\begin{remark}
  An ensemble satisfying \eqref{axiom:joinextend} is join-closed.
\end{remark}

The most simple application of \eqref{axiom:joinextend} allows us to join two
adjacent moves across a point of continuity; see \autoref{fig:extend}. 
\begin{lemma}\label{lem:extend-pair}
  Let $\IsJoinExtend{\Omega}$ be a move ensemble that satisfies~\eqref{axiom:joinextend}. 
  Then we have:
  \begin{equation}
    \tag{\text{\upshape{2-}}extend${}_A$} \label{axiom:extend}
    \begin{gathered}[t]
      \text{If $\gamma|_{(l,m)}, \gamma|_{(m,u)} \in \IsJoinExtend{\Omega}$, where $l < m < u$,
        and $m, \gamma(m) \in A$,}\\
      \text{then $\gamma|_{(l,u)} \in \IsJoinExtend{\Omega}.$}
    \end{gathered}
  \end{equation}
\end{lemma}

\begin{figure}
  \centering
  {\Huge $
  \begin{aligned}
    \MOVEDIAG{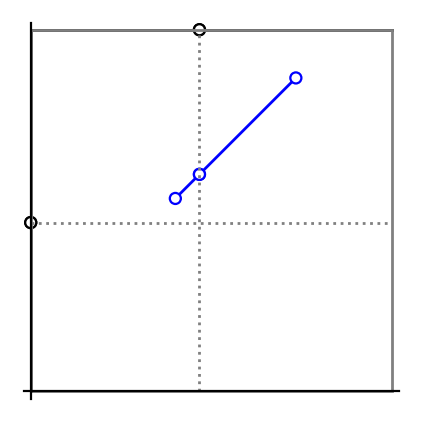}
    &\xmapstoextend
    \MOVEDIAG{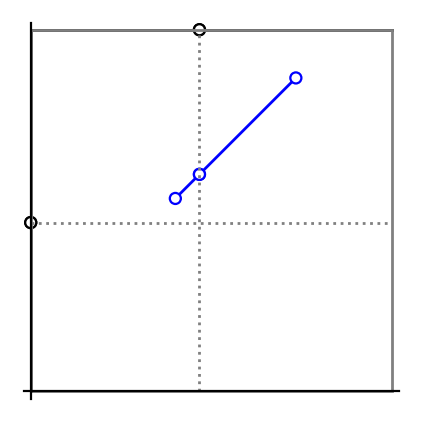}
    \\
    \MOVEDIAG{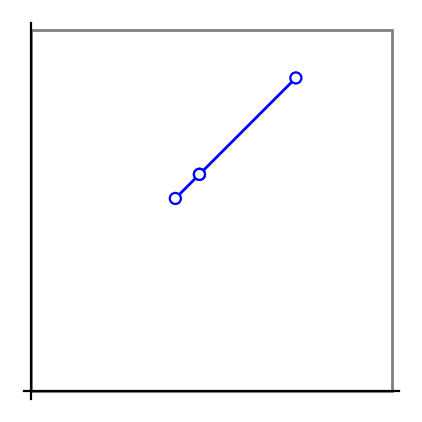}
    &\xmapstoextend
    \MOVEDIAG{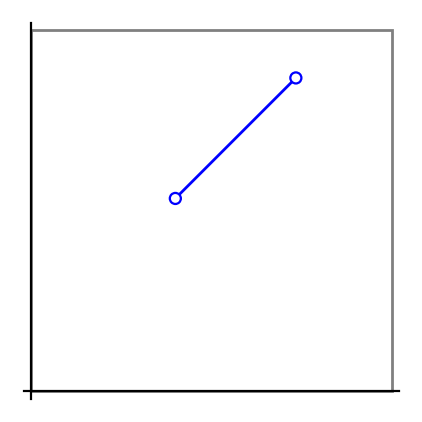}
    \\
    \MOVEDIAG{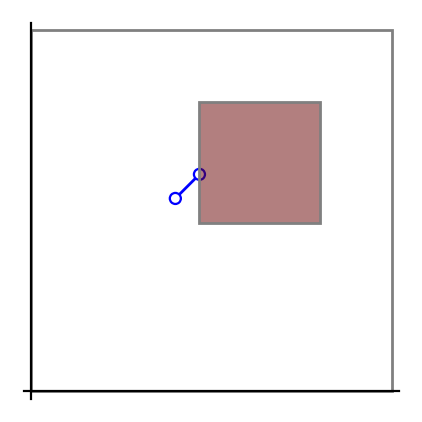}
    &\xmapstoextend
    \MOVEDIAG{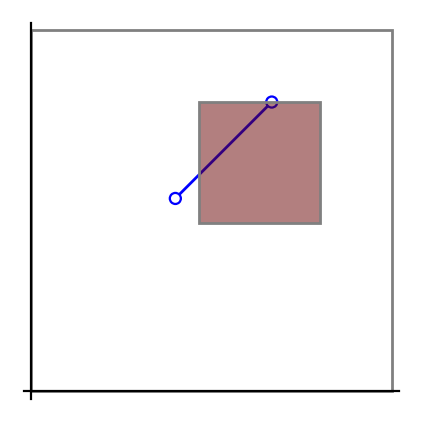}
    \\
  \end{aligned}
  $}
\caption[Extended move ensembles $\joinextend{\Omega}$ of ensembles~$\Omega$]
{Extended move ensembles $\joinextend{\Omega}$ of ensembles~$\Omega$.
  Points not in the continuity set~$A$ are indicated by \emph{black circles at
    the top and left border}. 
}
\label{fig:extend}
\end{figure}

The following is clear from the definition.
\begin{lemma}
  Let $\Omega$ be a move ensemble with
  $\dom(\Omega) = \im(\Omega) \subseteq A$.  Let
  $\IsJoinExtend{\Omega} = \joinextend{\Omega}$.  Then
  $\dom(\IsJoinExtend{\Omega}) = \im(\IsJoinExtend{\Omega}) \subseteq A$.
\end{lemma}

\begin{remark}
  \label{rem:extend-infinite}
  If $\Omega^{\join/}$ is a joined ensemble with finite
  $\maxdom(\Omega^{\join/})$, then repeated application of
  \eqref{axiom:extend}, followed by applying \eqref{axiom:join}, suffices to obtain
  $\joinextend{\Omega^{\join/}}$.

  However, this is not true for arbitrary joined ensembles
  $\Omega^{\join/}$.  As an example, let $A=\R$ and consider
  $\Omega^{\join/}$ consisting of the restrictions of a move~$\gamma$ to all
  subintervals of $(-1,0)$ and $(\frac1{n+1}, \frac1{n})$ for $n\in\N$.
  (These maximal intervals are disjoint.)  Domains of moves of
  $\joinextend{\Omega^{\join/}}$ are all subintervals of $(-1,1)$.  The
  domains of moves of 
  $\extend{\Omega^{\join/}}$ are $(-1,0)$ and its subintervals and the
  infinite chain $(\frac1m, 1)$ for $m\in \N$ and some of its subintervals;
  the supremum of the chain, $(0,1)$ is not an element.
  Then the domains of maximal moves of $\join{\extend{\Omega^{\join/}}}$ are
  $(-1,0)$ and $(0,1)$.  It takes \emph{another round} of $\extend\relax$ to
  arrive at $\joinextend{\Omega^{\join/}}$.  
\end{remark}

We have an explicit description of the moves in the extended move ensemble
$\joinextend{\Omega}$, similar to \autoref{lemma:joined-ensemble-explicitly}
for $\join{\Omega}$.
\begin{remark}\label{rem:jextend-equiv-def}
  For a move ensemble $\Omega$ with $\dom(\Omega), \im(\Omega) \subseteq
  A$, where $A \subseteq \R$ is an open set, we have
\begin{equation}
\label{eq:jextend-equiv-def}
\begin{aligned}
  \joinextend{\Omega} = \bigl\{ \, \gamma|_D \bigst \, &\gamma \in \FullMoveGroup,\; \text{$D$ empty or open interval}, \\
  &D\subseteq \cl (C_\gamma) \cap A \cap \gamma^{-1}(A)
                    \,\bigr\}, 
 \end{aligned}
\end{equation}
where $C_\gamma := \bigcup \{\, I \st \gamma|_I \in \Omega \,\}$, which is a subset of~$A \cap \gamma^{-1}(A)$.
\end{remark}

\subsubsection{Domain extension and semigroups}

\begin{lemma}
\label{lemma:jextendsemi-closed}
 Let $\Gamma$ be a move semigroup with $\dom(\Gamma), \im(\Gamma) \subseteq
 A$, where $A \subseteq \R$ is an open set. Then $\joinextend{\Gamma}$ is a
 move semigroup that satisfies \eqref{axiom:joinextend}. 
 \tred{Should it say 'the smallest'?}
\end{lemma}
\begin{proof}
Since $\Gamma$ satisfies \eqref{axiom:inverse}, it is clear that $\joinextend{\Gamma}$  satisfies \eqref{axiom:inverse}. We now show that  $\joinextend{\Gamma}$ satisfies \eqref{axiom:composition}, too.

Let $\gamma_1|_{D_1}, \gamma_2|_{D_2} \in \joinextend{\Gamma}$.  
Let 
\[C_1 = C_{\gamma_1} = \bigcup \{\, I \st \gamma_1|_I \in \Gamma \,\} \; \text{ and } \; C_2 = C_{\gamma_2} = \bigcup \{\, I \st \gamma_2|_I \in \Gamma \,\}.\]
By equation \eqref{eq:jextend-equiv-def}, the open set $D_1$ and $D_2$ satisfy that
\[D_1\subseteq \cl (C_1) \cap A \cap \gamma_1^{-1}(A) \; \text{ and } \; D_2\subseteq \cl (C_2) \cap A \cap \gamma_2^{-1}(A).\]
Let $\gamma = \gamma^2 \circ \gamma^1$, $C= C_\gamma= \bigcup \{\, I \st
\gamma|_I \in \Gamma \,\}$ and let $D = \gamma_1^{-1}(D_2) \cap D_1$ be a non-empty open set. We will show that 
\begin{equation}
\label{eq:proof-jextendsemi-closed}
D \subseteq \cl (C) \cap A \cap \gamma^{-1}(A).
\end{equation}
 It then follows again from \eqref{eq:jextend-equiv-def} that $\gamma_2|_{D_2} \circ \gamma_1|_{D_1} = \gamma|_{D} \in  \joinextend{\Gamma}$, and hence $ \joinextend{\Gamma}$ is a move semigroup.
It suffices to show \eqref{eq:proof-jextendsemi-closed} for
\[D_1 = \intr\bigl(\cl (C_1) \cap A \cap \gamma_1^{-1}(A) \bigr) \; \text{ and } \; D_2 = \intr\bigl(\cl (C_2) \cap A \cap \gamma_2^{-1}(A)\bigr).\]
We have on the left hand side of \eqref{eq:proof-jextendsemi-closed}
\begin{align*}
D &=  \gamma_1^{-1}(D_2) \cap D_1  \\
   &= \intr\bigl( \gamma_1^{-1}\bigl(\cl(C_2)\bigr) \cap \gamma_1^{-1}(A) \cap \gamma^{-1}(A) \bigr) \cap 
   \intr\bigl( \cl(C_1) \cap A \cap \gamma_1^{-1}(A)\bigr) \\
   &= \intr\bigl(\cl(C_1)\bigr) \cap  \gamma_1^{-1} \bigl( \intr\bigl(\cl(C_2)\bigr) \bigr) \cap A \cap \gamma_1^{-1}(A) \cap \gamma^{-1}(A),
\end{align*} 
and on the right hand side  of \eqref{eq:proof-jextendsemi-closed} $ \cl (C) \cap A \cap \gamma^{-1}(A)$.
Thus, it suffices to prove that if $x \in \intr\bigl(\cl(C_1)\bigr)$ such that $\gamma_1(x) \in \intr\bigl(\cl(C_2)\bigr)$, then $x \in \intr\bigl(\cl(C)\bigr)$. This holds since $\Gamma$ satisfies \eqref{axiom:composition}.
\end{proof}

\subsubsection{Respecting extensions}

\begin{lemma}[Extend moves by continuity]
  \label{lem:extended-moves}
  Let $\theta$ be a function that respects a move ensemble $\Omega$ 
  \tblue{new}with $\dom(\Omega), \im(\Omega) \subseteq A$. 
  Then it respects the extended move ensemble $\joinextend{\Omega}$. 
\end{lemma}
\begin{proof}
  We use the characterization of $\joinextend{\Omega}$ from
  \autoref{rem:jextend-equiv-def}.  Let $\gamma\in\FullMoveGroup$ and let
  $C_\gamma \subseteq A\cap\gamma^{-1}(A)$ be as in
  \autoref{rem:jextend-equiv-def}.  The function
  $x\mapsto \theta(\gamma(x)) - \chi(\gamma)\theta(x)$ is constant on the
  connected components of $C_\gamma$ and it is continuous on
  $A\cap \gamma^{-1}(A)$. Then it is constant on the connected components
  of~$\cl(C_\gamma) \cap A\cap\gamma^{-1}(A)$.
\end{proof}

Applied to the simple case of \autoref{lem:extend-pair}, we have the following.
\begin{corollary}
  Suppose $\theta$ respects the moves $\gamma|_{(l,m)}, \gamma|_{(m,u)}$  with
$l < m < u$ and suppose $\theta$ is continuous at $m$,
  $\gamma(m)$.  Then $\theta$ respects $\gamma|_{(l,u)}$. \tred{when defining ''respect'', is $\theta$ continuous on $D$?}
\end{corollary}

\begin{remark}
  The assumption regarding continuity at both $m$ and $\gamma(m)$ cannot be
  removed, which explains why we use $A\cap \gamma^{-1}(A)$ in
  \eqref{axiom:joinextend}.
  We illustrate this by the following example.  Let
  $A = (0,2)\cup (2,3)$. Let $\gamma = \tau_1$ and
  $\Omega = \{ \gamma|_{(0,1)}, \gamma|_{(1,2)} \}$, so
  $\dom(\Omega) = (0,1)\cup(1,2) \subseteq A$ and
  $\im(\Omega) = (1,2)\cup(2,3) \subseteq A$.  Then $1 \in A$, but
  $\gamma(1) = 2 \notin A$. Define $\theta = 0$ on $A$ and $\theta(2) = 1$, so
  it is continuous at $1$ but not at $\gamma(1)=2.$  Then $\theta$ respects
  $\Omega$, but it does not respect the move $\gamma|_{(0,2)}$.
\end{remark}

\subsection{Closed move semigroups, the moves closure $\ctscl(\Omega)$}
\label{s:ctscl-definition}

Now all axioms that we have introduced above come together.

\begin{definition}
  A \emph{closed move semigroup} is a limits-closed extension-closed
  kaleidoscopic joined move
  semigroup, i.e., a move ensemble that satisfies all the following axioms:
   \eqref{axiom:composition}, \eqref{axiom:inverse}, \eqref{axiom:join}, \eqref{axiom:restrict}, \eqref{axiom:joinextend},    \eqref{axiom:limits}, and \eqref{axiom:translation-reflection}.
\end{definition}

\begin{definition}
  Let $\Omega$ be a move ensemble with $\dom(\Omega), \im(\Omega) \subseteq A$.  
  We define the \emph{closed move semigroup}
  $\ctscl(\Omega)$ generated by~$\Omega$ (or just \emph{moves closure} of
  $\Omega$) to be the smallest (by set inclusion) closed move semigroup
  containing~$\Omega$. 
\end{definition}

\begin{lemma}\label{lemma:smallest-set-is-intersection}
  Let $\posetL$ be the family of closed move semigroups containing $\Omega$.  Then
  $\ctscl(\Omega) = \bigcap \posetL = \bigcap_{\Omega' \in \posetL} \Omega'$.
\end{lemma}
\tred{We actually use this argument (smallest=intersection) earlier already.}

\begin{proof}

  First of all, $\bigcap \posetL$ contains $\Omega$. 
  We show that $\bigcap \posetL$  is a closed move semigroup. 
  Note that each axiom is a closure property of a set~$\Omega'$ of the
  form: For all $(\Omega_1, \Omega_2) \in \mathbb X$, if $\Omega_1 \subseteq \Omega'$,
  then $\Omega_2 \subseteq \Omega'$.  
  Now if $\Omega_1 \subseteq \bigcap \posetL$, then
  $\Omega_1 \subseteq \Omega'$ for all $\Omega'\in \posetL$, and thus
  $\Omega_2 \subseteq \Omega'$ for all $\Omega'\in \posetL$.  This implies
  $\Omega_2 \subseteq \bigcap\posetL$.

  On the other hand, $\bigcap \posetL$ is contained in each of the
  ensembles $\Omega' \in \posetL$ and is therefore the smallest closed move
  semigroup containing~$\Omega$. 
\end{proof}

\begin{remark}
  In contrast to \autoref{lemma:joinsemi-is-semi} (regarding
  \eqref{axiom:join} and \eqref{axiom:restrict} and the axioms of an inverse semigroup), 
  we do not know whether $\ctscl(\Omega)$ can be obtained by
  applying a finite sequence of closures with respect to the individual
  axioms. 
\end{remark}

\begin{theorem}[Main theorem on the moves closure]
  \label{thm:cts-moves-closure}
  Suppose $\theta$ is bounded and continuous on $A$.  If $\theta$ respects a
  move ensemble $\Omega$ with $\dom(\Omega), \im(\Omega) \subseteq A$, then
  $\theta$ respects the moves closure $\ctscl(\Omega)$.
\end{theorem}
\begin{proof}[Proof of \autoref{thm:cts-moves-closure}]
  Let $\theta|_A$ denote the restriction of $\theta$ to $A$.  We consider the
  ensemble $\Gamma = \Omegaresp{\theta|_A}$ of moves that $\theta|_A$ respects,
  introduced in \autoref{s:Omegaresp}.  By definition,
  $\dom(\Gamma), \im(\Gamma) \subseteq A$.
  Since, by assumption, $\theta$ respects~$\Omega$, we have
  $\Gamma \supseteq \Omega$.
  By \autoref{thm:additivity-equation-for-move}, $\Gamma$ is a
  join-closed move semigroup. By \autoref{lem:IL-Moves}, because $\theta|_A$
  is bounded, $\Gamma$ satisfies the axiom \eqref{axiom:translation-reflection}.
  Because $\theta|_A$ is continuous, we can apply 
  \autoref{lem:limits-of-moves} to all convergent sequences
  $\{\gamma^i|_D\}_{i\in\N}
  \subseteq \Gamma$, and thus $\Gamma$ satisfies the axiom
  \eqref{axiom:limits}. Finally, by 
  \autoref{lem:extended-moves}, it satisfies the axiom
  \eqref{axiom:joinextend}.
  Hence, $\Omegaresp{\theta}$ is a closed move semigroup. 
  By \autoref{lemma:smallest-set-is-intersection}, we conclude that $\theta$ respects $\ctscl(\Omega)$.
\end{proof}

\section{The initial additive move ensemble $\Omegainit$ of a subadditive
  function
}
\label{sec:construct_initial_moves}

\tblue{This is the first place in the paper (after introduction) where
  subadditivity and the additivity set comes into play.}

We will now apply the theory of the previous sections to compute the effective
perturbation spaces of minimal valid functions.  Let $\pi\colon\R\to\R$ be a
minimal valid function.  Recall from the introduction that $\pi$ is
nonnegative, $\Z$-periodic, and satisfies $\pi(0)=0$, $\pi(f)=1$.  Its key
property is subadditivity, which we express using the \emph{subadditivity
  slack} function
\begin{equation}
  \Delta\pi(x,y) = \pi(x)+\pi(y) -\pi(x+y)
\end{equation}
as $\Delta\pi(x,y)\geq 0$. Moreover, the symmetry condition
$\Delta\pi(x,f-x)=0$ holds for all $x$.  This is the characterization that
appeared in the introduction as~\eqref{eq:minimal}.

Since $\pi$ is $\Z$-periodic, we will work with its fundamental domain
$[0,1]$.  Let $A = A(\pi)$ be the maximal open subset of $(0,1)$ on which
$\pi$ is continuous.

\subsection{The initial move ensemble $\Omegainit$}
\label{s:definition_initial_moves_general}

We begin by defining an ensemble of \emph{initial moves} $\Omegainit =
\Omegainit(\pi)$ that consists of \emph{additive moves} and \emph{limit additive
  moves}, together with their inverses and the empty moves.  We define these
moves $\gamma|_D$ on 
domains $D$ that are open
intervals such that the domain $D$ and the image $\gamma(D)$ are subsets
of~$A$
.
\begin{definition}
\label{def:initial-moves}
  \begin{enumerate}[\rm(i)]
  \item 
  An \emph{additive move} is any translation $\tau_t|_D$, where $t \in (-1, 1)$
  and $D\subseteq A$ is an open interval such that $\tau_t(D) \subseteq A$ and
  \begin{subequations}\label{eq:additive-move-equation}
    \begin{align*}
      \Delta\pi(x, t) &= \pi(x) + \pi(t) - \pi(x+ t) = 0 && \forall x \in D \\
      \intertext{or any reflection $\rho_r|_D$, where $r \in (0,2)$, and
      $D\subseteq A$ is an open interval such that $\rho_r(D) \subseteq A$
      such that} 
      \Delta\pi(x, r-x) &= \pi(x) + \pi(r-x) - \pi(r) = 0 && \forall x \in D.
    \end{align*}
  \end{subequations}
\item A \emph{limit-additive move} is any translation $\tau_{\bar t}|_D$,
  where $\bar t \in (-1,1)$ and  $D \subseteq A$ is an open interval such that
  $\tau_{\bar t}(D) \subseteq A$ and
  \begin{subequations}\label{eq:limit-additive-move-equation}
    \begin{align*}
      \lim_{t \to \bar t^+} \Delta\pi(x, t) = 0 
      &\quad\text{or}\quad \lim_{t \to \bar t^-} \Delta\pi(x, t) = 0  && \forall x \in D\\
      \intertext{or any reflection $\rho_{\bar r}|_D$, where $\bar r \in
      (0,2)$, and $\rho_{\bar r}(D) \subseteq A$  such
      that}
      \lim_{r \to \bar r^+} \Delta\pi(x, r-x) = 0 
      &\quad\text{or}\quad \lim_{r \to \bar r^-} \Delta\pi(x, r-x) = 0 && \forall x \in D.
    \end{align*}
  \end{subequations}
\item An \emph{initial move} in $\Omegainit(\pi)$ is a move that is either additive or
  limit-additive, or an inverse of such a move, or an empty move.
\end{enumerate}
\end{definition}
\begin{remark}
  The property of the moves $\gamma|_D \in \Omegainit$ that the function~$\pi$
  is continuous on the domain~$D$ and image~$\gamma(D)$ will be preserved
  throughout. 
\end{remark}

\begin{remark}\label{rem:initial-moves-join-closed}
  The initial move ensemble $\Omegainit$ is join-closed.
  Therefore, by
  \autoref{lemma:joined-is-presented-by-max}, it is equal to the restriction
  closure of its maximal elements.
  Moreover, by definition, $\Omegainit$ satisfies~\eqref{axiom:inverse}.
  However, $\Omegainit$ in general is not a semigroup.
\end{remark}

The function~$\pi$ is affinely $\Omegainit$-equivariant (\autoref{s:space-equivariant}), i.e.,
it respects all moves in $\Omegainit$.

\subsection{Moves from connected open sets of additivities}

We now specialize our results from \autoref{sec:semigroup:open} regarding
connected open ensembles to the initial moves. 
We have the following corollary.
Recall from \autoref{s:cauchy-pexider} the projections $p_1(x,y) = x$,
$p_2(x,y) = y$, and $p_3(x,y) = x+y$  as functions  from $\R^2$ to $\R$. 
\begin{corollary}
\label{lem:squares}
Suppose $E \subseteq \R^2$ be a connected open set on which $\pi$ is additive,
i.e., $\Delta\pi(x,y) = 0$ for $(x,y) \in E$.  Let
$C = p_1(E) \cup p_2(E) \cup p_3(E)$ and assume that $C \subseteq A$.  Then
$$
\MovesOfGraph(C\times C) \subseteq \joinsemi{\Omegainit},
$$
so $C$ is a connected covered component of $\joinsemi{\Omegainit}$.
\end{corollary}
See \autoref{fig:triangles-and-squares-nonsampled} for an illustration.
\begin{remark}
  In \cite{hong-koeppe-zhou:software-paper}, the intervals
  $p_1(E), p_2(E), p_3(E)$ are referred to as \emph{directly covered
    intervals}.
\end{remark}

\begin{proof}[Proof of \autoref{lem:squares}]
Denote $\Gamma^{\join/} =  \joinsemi{\Omegainit}$. By \autoref{lemma:joinsemi-is-semi}, $ \Gamma^{\join/}= \semi{\Gamma^{\join/}}$.
We first show that $\graph_\pm(\Gamma^{\join/})$ contains $E$.  Let $(x,y) \in
E$. Since $E$ is open, there exists an open interval $D \ni x$ such that the
diagonal segment $\{\,(x', r-x') \st x' \in D\,\} \subseteq E$, where
$r=x+y$. By \autoref{def:initial-moves}, we have $\rho_r|_D \in \Omegainit$,
with $\rho_r|_D(x)=y$.
Thus, $(x,y) \in  \graph_-(\Gamma^{\join/})$. There exist open intervals $D_y \ni y$ and $D_x \ni x$ such that the vertical segment $\{x\} \times D_y$ and the horizontal segment $D_x \times \{y\}$ are contained in $E$. Again by \autoref{def:initial-moves}, we have $\tau_y|_{D_x}, \tau_x|_{D_y} \in \Omegainit$. Notice that \[ x \xmapsto{\tau_y|_{D_x}} (x+y) \xmapsto{(\tau_x|_{D_y})^{-1}} y.\] Thus, $(x,y) \in \graph_+(\Gamma^{\join/})$. We showed that $\graph_\pm(\Gamma^{\join/})$ contains $E$.
By \autoref{thm:open-sets-of-moves}-(3), $\MovesOfGraph((p_1(E) \cup p_2(E)) \times (p_1(E) \cup p_2(E))) \subseteq \Gamma^{\join/}$. 

For any point $x+y \in p_3(E)$, where $x \in p_1(E)$ and $y \in p_2(E)$, the
above translation move $\tau_y|_{D_x}$ satisfies that $\tau_y|_{D_x} \in
\Omegainit$ and $\tau_y|_{D_x}(x)= x+y$. By applying
\autoref{lem:indirectly_covered_from_move} to $\C = \{p_1(E)\cup p_2(E)\}$ and
all such moves  $\tau_y|_{D_x}$, we obtain that $\MovesOfGraph((p_1(E) \cup p_2(E)
\cup p_3(E))\times (p_1(E) \cup p_2(E) \cup p_3(E))) \subseteq \Gamma^{\join/}$.
\end{proof}

\section{Piecewise linear functions, polyhedral complexes, effective perturbations}
\label{s:piecewise}

We now specialize our theory to the important case of piecewise linear
functions.  We begin with the basic definitions and review some tools that
were developed in the previous papers \ifserialtitle of the present
series\else on this topic\fi.

\subsection{Continuous and discontinuous piecewise linear
  functions~$\pi$, complex $\P_{\texorpdfstring{B}{}}$}
\label{s:pwl-functions}

We begin by giving a definition of $\Z$-periodic piecewise linear
functions~$\pi\colon \R\to\R$ that are allowed to be discontinuous, following
\cite{koeppe-zhou:crazy-perturbation}.  \cite{hong-koeppe-zhou:software-paper}
discusses how these functions are represented in the software
\cite{cutgeneratingfunctionology:lastest-release}.

Let $0 =x_0 < x_1< \dots < x_{n-1} < x_n = 1$.
Denote by
\begin{equation}\label{eq:def-B}
    \B = \{\, x_0 + t,\, x_1 + t,\, \dots,\, x_{n-1}+t\st
    t\in\Z\,\}
  \end{equation}
the set of all breakpoints.
The 0-dimensional faces are defined to be the 
singletons, $\{ x \}$, $x\in B$,
and the 1-dimensional faces are the closed intervals,
$ [x_i+t, x_{i+1}+t]$, $i=0, \dots, {n-1}$, $t\in\Z$. 
The empty face, the 0-dimensional and the 1-dimensional faces form $\P =
\P_{\B}$, a locally finite
polyhedral
complex,
periodic modulo~$\Z$.
\begin{definition}
  We call a function $\pi\colon \R \to \R$
  \emph{piecewise linear} over $\P_B$ if for each face
  $I \in \P_B$, there is an affine linear function
  $\pi_I \colon \R \to \R$, $\pi_I(x) = c_I x + d_I$ such that
  $\pi(x) = \pi_I(x)$ for all $x \in\relint(I)$.
\end{definition}
 Under this definition, piecewise linear functions can be discontinuous
 .
Let $I = [a, b] \in \P_B$ be a 1-dimensional face. The function $\pi$ can be determined on
$\intr(I) = (a, b)$ by linear
  interpolation of the limits $\pi(a^+)=\lim_{x\to a, x>a} \pi(x)
  = \pi_I(a)$ and $\pi(b^-)=\lim_{x\to b, x<b} \pi(x) = \pi_I(b)$. 

\subsection{Two-dimensional polyhedral complex $\Delta\P$ and additive faces}
\label{s:polyhedral-complex}

\begin{figure}[tp]
\centering
\begin{minipage}{.47\textwidth}
\centering\hspace*{-1em}
\includegraphics[width=\linewidth]{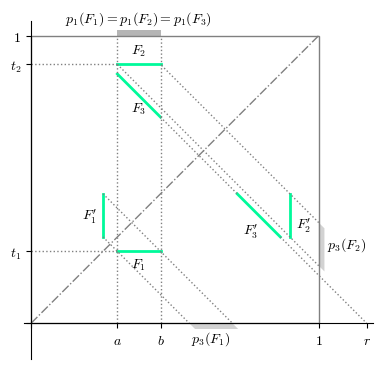}
\end{minipage}%
\begin{minipage}{.53\textwidth}
\centering
\vspace{1em}
\includegraphics[width=\linewidth]{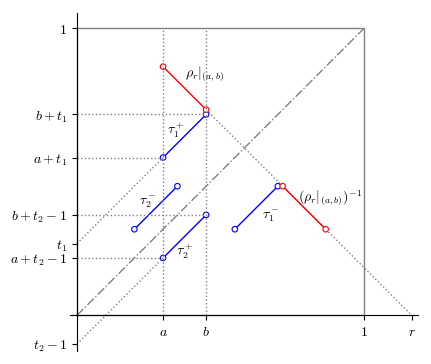}
\end{minipage}
\caption{Additive edges in $\Delta\P_B$ and the corresponding initial moves
  (additive moves and their inverses) in $\Omegainit$
}
\label{fig:moves-diagram}
\end{figure}

\newcommand{\FACETOMOVEROWBIG}[1]{%
  \mbox{$
    \vcenter{\hbox{\includegraphics[height=.45\linewidth]{graphics-for-algo-paper/#1-2d_diagram}}}
  \qquad
  \vcenter{\hbox{\includegraphics[height=.45\linewidth]{graphics-for-algo-paper/#1-completion-final}}}
  $}
}

\newcommand{\FACETOMOVEROW}[2][0]{
  \mbox{$
    \vcenter{\hbox{\includegraphics[height=.25\linewidth]{graphics-for-algo-paper/#2-2d_diagram}}}
  \rightsquigarrow
  \ifx#1\relax
  \else
  \vcenter{\hbox{\includegraphics[height=.2504\linewidth]{graphics-for-algo-paper/#2-completion-#1}}}
  \xmapstojoinsemi
  \fi
  \vcenter{\hbox{\includegraphics[height=.2506\linewidth]{graphics-for-algo-paper/#2-completion-final}}}
  $}
}
\begin{figure}
  \Huge \centering

  \mbox{$
    \vcenter{\hbox{\tikz[spy using outlines={circle}]{
          \node[anchor=south west,inner sep=0,outer sep=0] (g) at (0,0) {\includegraphics[height=.45\linewidth]{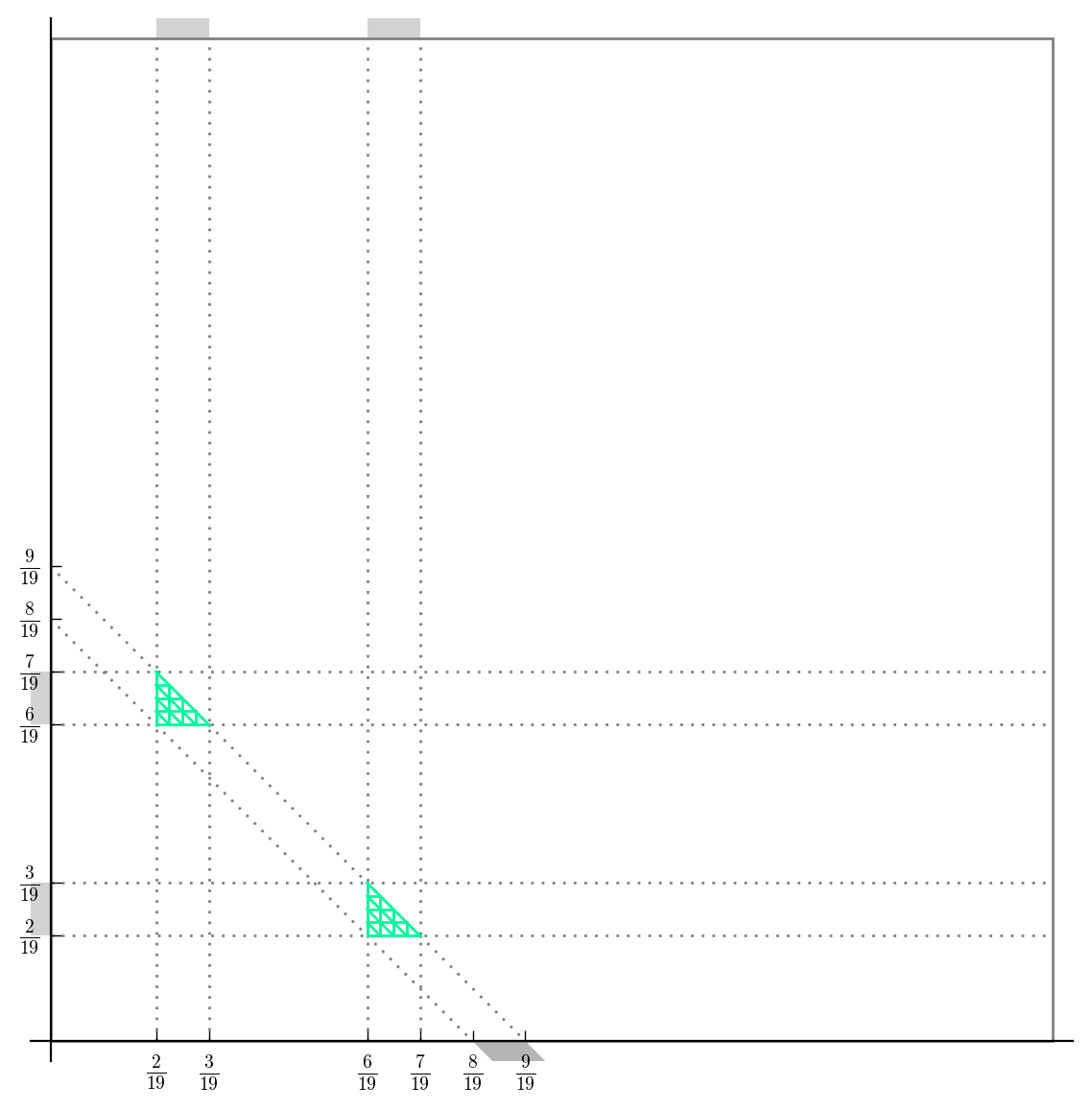}};
          \begin{scope}[x={(g.south east)},y={(g.north west)}]
            \coordinate (spytarget) at (0.18, 0.38);
            \coordinate (spyglass) at (0.7, 0.7);
            \spy [blue, draw, size=.22\linewidth, magnification = 6,
            connect spies] on (spytarget) in node at (spyglass);
          \end{scope}
        }%
      }%
    }%
    \qquad
    \vcenter{\hbox{\tikz[spy using outlines={circle}]{
          \node[anchor=south west,inner sep=0,outer sep=0] (g) at (0,0)
          {\includegraphics[height=.45\linewidth]{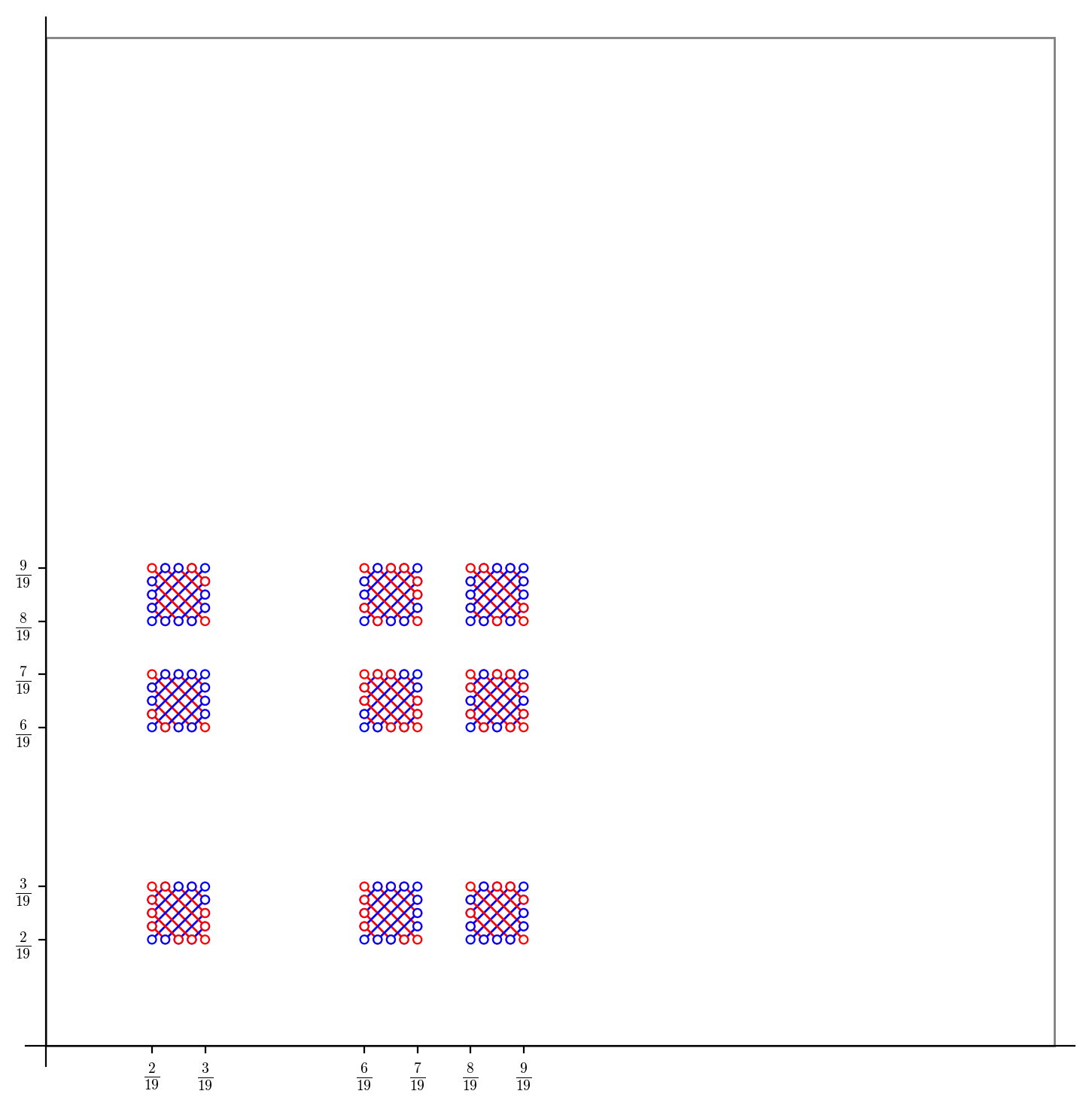}};
          \begin{scope}[x={(g.south east)},y={(g.north west)}]
            \coordinate (spytarget) at (0.455, 0.465);
            \coordinate (spyglass) at (0.7, 0.7);
            \spy [blue, draw, size=.22\linewidth, magnification = 6,
            connect spies] on (spytarget) in node at (spyglass);
          \end{scope}
        }
      }
    }
  $}
\caption[Additivities and initial moves.
  \emph{Left,} additivities sampled from two-dimensional additive faces
  of~$\Delta\P$.
  \emph{Right,} the move semigroup $\joinsemi{\Omega^0}$ generated by the initial
  moves]
  {Additivities and initial moves.
  \emph{Left,} additivities sampled from two-dimensional additive faces
  of~$\Delta\P$.
  \emph{Right,} the move semigroup $\joinsemi{\Omega^0}$ generated by the initial
  moves. The graphs
  $\graph_+(\Omega)$ (\emph{blue}) and $\graph_-(\Omega)$ (\emph{red}) are
  plotted on top of each other.
  For illustration purposes, only a finite set of additive moves is
  considered.
}
\label{fig:triangles-and-squares}
\end{figure}


\begin{figure}
  \Huge\centering
  \FACETOMOVEROWBIG{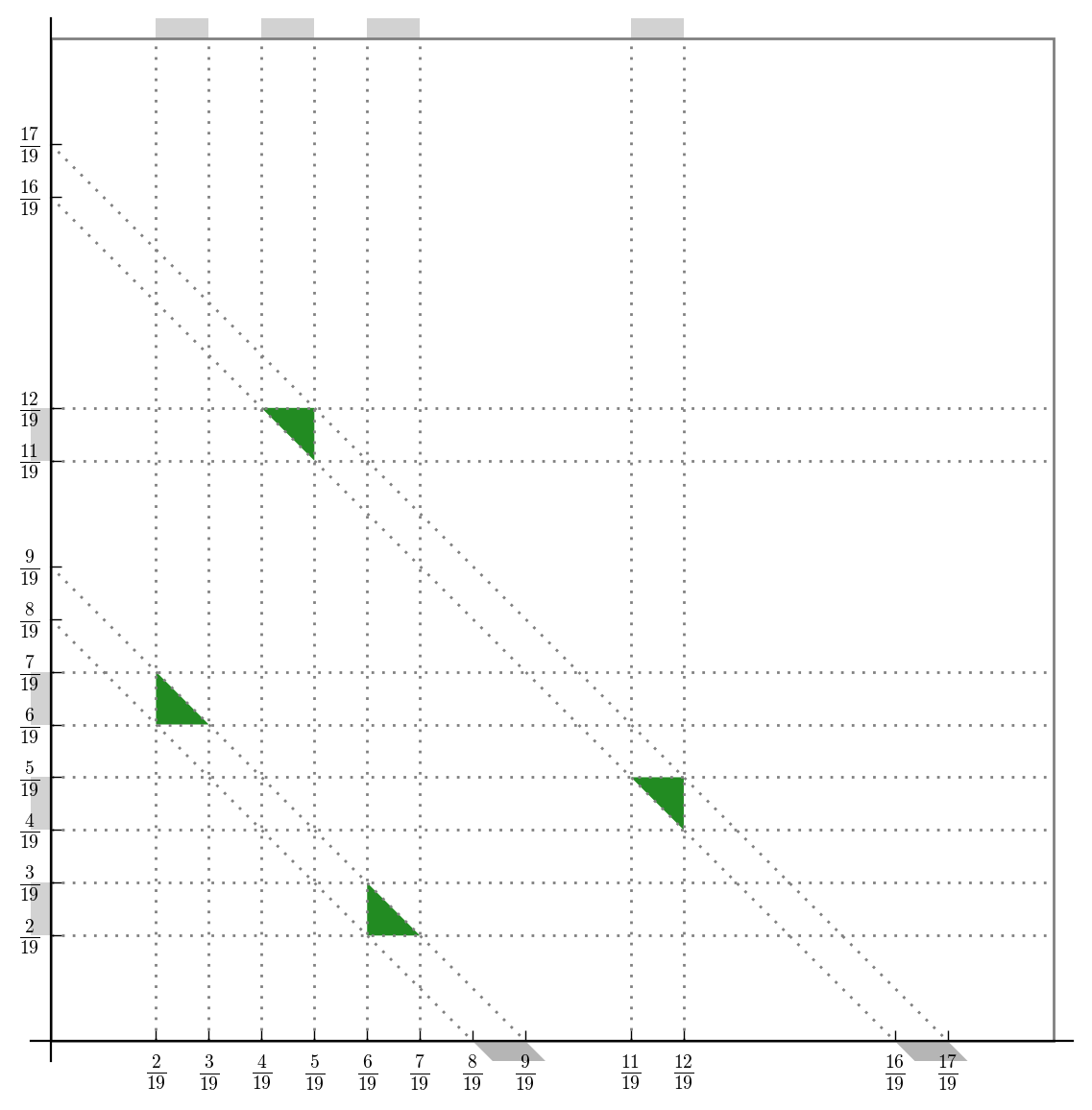} \\
\caption{
  Additivities in $E(\pi)$ and the corresponding
  connected covered components of moves}
\label{fig:triangles-and-squares-nonsampled}
\end{figure}

For a piecewise linear function (see \autoref{s:pwl-functions} for our
notation), we now explain the structure of the initial moves.  We
will use the notion of the polyhedral complex $\Delta\P$ and its
\emph{additive faces} from \cite[section 4]{hong-koeppe-zhou:software-paper}.
$\Delta\P$ is a two-dimensional polyhedral complex, which expresses the
domains of linearity of the subadditivity slack $\Delta\pi(x,y)$ introduced in
\autoref{s:intro:subadditive}.
\begin{definition}
  The polyhedral complex $\Delta\P$ of $\R\times\R$ consists of the faces
  $$F(I,J,K) = \setcond{(x,y) \in \R \times \R}{x \in I,\, y \in J,\, x + y \in
    K},$$
  where $I, J, K \in \P$, so each of $I, J, K$ is either empty, a breakpoint
  of $\pi$, or a closed interval delimited by two consecutive breakpoints.
\end{definition}

In the continuous case, since the function $\pi$ is piecewise linear over
$\P$, we have that $\Delta \pi$ is affine linear over each (closed) face
$F \in \Delta P$.  We say that a face $F \in \Delta\P$ is \emph{additive} if
$\Delta\pi =0$ over all $F$.  If $\pi$ is subadditive, then the \emph{set of additivities}
\begin{equation}
  E(\pi) = \{\, (x,y) \st \Delta\pi(x,y) = 0 \,\} \label{eq:additivity-set}
\end{equation}
 is the union of all
additive faces $F \in \Delta\P$; see \cite[section~3.4]{bhk-IPCOext}.

For a discontinuous function $\pi$, the subadditivity slack $\Delta\pi$ is
affine linear only over the relative interior of each face~$F$.
For additivity, beside the subadditivity slack $\Delta\pi(x,y)$ at a point
$(x, y)$, we also consider its limits. 
\begin{definition}
  The limit value of $\Delta\pi$ at the point $(x,y)$ approaching from the relative interior of a face
  $F \in \Delta\P$ containing $(x,y)$ 
  is denoted by
  \[\Delta\pi_{F}(x,y) = \lim_{\substack{(u,v) \to (x,y)\\ (u,v) \in
        \relint(F)}} \Delta\pi(u,v).\]
\end{definition}
\begin{definition}\label{def:set-of-additivities-and-limit-additivities-of-a-face-F}
  Let $F \in \Delta\P$.  Define the \emph{set of additivities and limit-additivities
    approaching from the relative interior of~$F$} as
  \begin{equation}
    E_{F}(\pi) = \bigl\{\,(x,y)\in F \bigst \Delta\pi_{F}(x,y) \text{ exists, and }
    \Delta\pi_{F}(x,y) = 0\,\bigr\}.\label{eq:additivity-family}
  \end{equation}
\end{definition}

\begin{remark}\label{rem:limit-additivites-stratified} 
  The points $(x,y) \in E_F(\pi)$ that lie in $\relint(F)$ capture all additivities
  of~$\pi$, whereas those that lie on the relative boundary capture
  all limit-additivities. 
  The set $E(\pi)$ that we introduced in the continuous case can be
  partitioned as
  $$E(\pi) = \bigcup_{F\in\Delta\P} (E_F(\pi) \cap \relint(F)).$$
\end{remark}

\begin{lemma}\label{lemma:additive-face-discontinuous}
  Let $\pi$ be a subadditive function that is piecewise linear over~$\P$.  
  Let $F \in \Delta\P$.
  Let $(x_0, y_0) \in E_F(\pi) \subseteq F$ and let $E$ be the
  unique face of~$F$ containing $(x_0,y_0)$ in its relative interior. 
  Then $E \subseteq E_{F}(\pi)$. 
\end{lemma}
We make the following general definition, which is equivalent to the
one found
in~\cite{hong-koeppe-zhou:software-paper,koeppe-zhou:crazy-perturbation}. 
\begin{definition}
  In the situation of \autoref{lemma:additive-face-discontinuous},
  we say that the face $E$ is \emph{additive}.%
\end{definition}
Now the following lemma is clear from the definition.
\cite{hong-koeppe-zhou:software-paper} only states this fact for the case of
continuous~$\pi$.
\begin{lemma}\label{lemma:additive-is-subcomplex}
  Let $\pi$ be a subadditive function that is piecewise linear over~$\P$.
  Then the set of additive faces of~$\pi$ is a polyhedral subcomplex
  of~$\Delta\P$, i.e., it is closed under taking subfaces.
  In particular, each additive face is the convex hull of some additive
  vertices (zero-dimensional additive faces). 
\end{lemma}

\medbreak

For a piecewise linear function $\pi$, a finite presentation of the initial
moves is easy to compute using the additive faces of the complex~$\Delta \P$.
For a detailed explanation of diagrams visualizing the additivities and
limit-additivities, we refer to \cite[sections
4.2--4.3]{hong-koeppe-zhou:software-paper}.  See \autoref{fig:moves-diagram}
for the moves from one-dimensional additive faces (edges) and
\autoref{fig:triangles-and-squares} and
\autoref{fig:triangles-and-squares-nonsampled} for the moves from two-dimensional additive
faces.  In the forthcoming paper~\cite{hildebrand-koeppe-zhou:algocomp-paper},
we will give a more detailed description how to compute the finite
presentation of the initial
moves.
\begin{remark}
  The zero-dimensional additive faces (i.e., additive vertices) of
  $\Delta\P_B$ do not give rise to moves
  (cf.~\autoref{rem:no-singleton-domains}).
  Instead they will be considered in \autoref{sec:perturbation_space} to
  determine a refinement of $\P_B$ for the decomposition of perturbations.
\end{remark}

\subsection{Effective perturbations $\tildepi$}

We recall the notion of effective perturbations from
\autoref{s:intro:effective-perturbations}.  An effective perturbation is a
function $\tilde{\pi} \colon \R \to \R$ for which there exists an $\epsilon>0$
such that $\pi^{\pm} = \pi \pm \epsilon\tilde{\pi}$ are minimal valid
functions.

\begin{remark}
\label{rk:effective-perturbation-bounded}
  Let $\pi$ be a minimal valid function for $R_f(\R,\Z)$.
  From \eqref{eq:minimal:nonneg}, \eqref{eq:minimal:01}, and
  \eqref{eq:minimal:symm} it follows that $0 \leq \pi \leq 1$, so $\pi$ is
  a bounded function.  Now if $\tilde{\pi}$ is an effective perturbation, 
  then $\pi^{\pm} = \pi \pm \epsilon\tilde{\pi}$ for some $\epsilon>0$, where
  also $0 \leq \pi^{\pm} \leq 1$, and so $\tilde{\pi}$ is a bounded function
  as well.
\end{remark}

We note that the space $\tildePi^{\pi}\RZ$ of effective perturbations,
introduced in \autoref{s:intro:effective-perturbations}, is a vector space.

\begin{lemma}\label{lemma:effective-perturbation-vector-space}
Let $\pi$ be a minimal valid function. The space $\tildePi^{\pi}\RZ$ of
effective perturbation functions
is a vector space, a
subspace of the space $\Bounded{\R}$ of bounded functions.
\end{lemma} 

For the case of piecewise linear functions~$\pi$ that are continuous from at
least one side of the origin, we have the following regularity theorem for
effective perturbations. \tred{OK this is piecewise; should it be here or later?}

\begin{lemma}[{\cite[Lemma 6.4]{hong-koeppe-zhou:software-paper}}]
  \label{lemma:perturbation-lipschitz-continuous}
  Let $\pi$ be a piecewise linear minimal valid function that is continuous from
  the right at $0$ or continuous from the left at $1$. If $\pi$ is continuous on
  a proper interval $I \subseteq [0,1]$, then for any
  $\tilde{\pi} \in \tildePi^{\pi}\RZ$ we have that $\tilde{\pi}$ is
  Lipschitz continuous on the interval $I$.
\end{lemma} 
(This is a strengthening of \cite[Theorem 2]{dey1}.)

The purpose of the additive move ensemble is to infer properties of the
effective perturbation functions.  For additive moves $\gamma|_D$, it follows
from convexity that every effective perturbation $\tilde\pi$ respects
$\gamma|_D$. In the case of piecewise linear functions, this extends to
limit-additive moves.  The following lemma is shown by the proof of
\cite[Theorem
6.3]{hong-koeppe-zhou:software-paper}
, along with \cite[Footnote
13]{hong-koeppe-zhou:software-paper}
and also by \cite[Theorem
3.3]{koeppe-zhou:crazy-perturbation}
in the case where $\pi$ is two-sided discontinuous at the origin.
\begin{lemma}
\label{lemma:additivity-equation-for-initial-move}
Let $\pi$ be a piecewise linear minimal valid function for $R_f(\R,\Z)$. Let $\gamma|_D \in \Omegainit$ be an initial move, where $D \subseteq (0,1)$ is an open interval. 
Then $\pi$ respects $\gamma|_D$,
and every effective perturbation function $\tilde{\pi} \in \tildePi^{\pi}\RZ$ respects~$\gamma|_D$. 
\end{lemma}

\begin{corollary}
  \label{cor:pwl-moves-closure-respected}
  Let $\pi$ be a piecewise linear minimal valid function for $R_f(\R,\Z)$. 
  Then $\pi$  respects the moves closure $\clsemi{\Omegainit}$.
  If $\pi$ is is continuous from at least one side of the origin,
  then every effective perturbation function $\tilde{\pi} \in
  \tildePi^{\pi}\RZ$ also respects the moves closure $\clsemi{\Omegainit}$.
\end{corollary}

\begin{proof}
This follows from \autoref{lemma:additivity-equation-for-initial-move}, \autoref{rk:effective-perturbation-bounded}, \autoref{lemma:perturbation-lipschitz-continuous} and \autoref{thm:cts-moves-closure}.
\end{proof}

\subsection{Closed move semigroup generated by $\Omegainit$, rational case}

We have the following theorem.
\begin{theorem}[Finite presentation of the moves closure, rational case]
  Let $\pi$ be a piecewise linear function whose breakpoints are rational,
  i.e., $B \subseteq G = \frac1q\Z$ for some $q\in \N$.  Then the moves
  closure $\clsemi{\Omegainit}$ has a finite presentation
  $(\Omega^{\red/}, \CC)$ in reduced form, where 
  (i) the endpoints of all domains and the values $t$ and
  $r$ of moves $\tau_t, \rho_r|_D \in \Omega^{\red/}$ lie in $G \cap [0,1]$,
  (ii) the endpoints of all maximal intervals of all $\Comp[i] \in \CC$ lie in $G \cap [0,1]$.
  \label{th:finite-presentation-closure-rational}
\end{theorem}
\begin{proof}[Proof sketch]
  We can compute $\clsemi{\Omegainit}$ in finitely many steps using a
  completion-type algorithm that manipulates finite presentations, maintaining
  properties (i) and (ii), using only the algebraic and order-theoretic axioms
  and \eqref{axiom:joinextend}.  The initialization is provided by
  \autoref{cor:open-sets-of-moves}, noting that vertices of additive faces
  of~$\Delta\P$ lie in $G\times G$.  There are only finitely many finite
  presentations satisfying (i) and (ii); this implies the finiteness of the
  algorithm.
\end{proof}
We defer all details about such an
algorithm, as well as its generalization to non-rational input, to the
forthcoming paper~\cite{hildebrand-koeppe-zhou:algocomp-paper}.

Instead, in the next section, we assume that a finite presentation
$(\Omega^{\fin/}, \CC)$ of the moves closure $\clsemi{\Omegainit}$ is
given.  Using the finite presentation, we can give a description of the
space of effective perturbations.

\section{Perturbation space}
\label{sec:perturbation_space}

Let $\pi\colon\R\to\R$ be a minimal valid function.
In this section, we work with the following assumptions.
(We will mention them explicitly only in statements of main theorems.)

\subsection{Assumptions: Piecewise linear~$\pi$, one-sided continuous at~$0$, finitely presented
  moves closure~$\clsemi{\Omegainit}$}
\label{s:assumptions}

\begin{assumption}
  \label{assumption:one-sided-continuous}
  \label{assumption:piecewise}
  The minimal valid function $\pi$ is piecewise linear
  (\autoref{s:pwl-functions}) and continuous from at least one side of the
  origin. 
\end{assumption}

\begin{assumption}
  \label{assumption:minimal-B}
  The set $B$ is minimal, i.e., $\P_B$ is the coarsest polyhedral
  complex over which $\pi$ is piecewise linear.
\end{assumption}

Let $\Omegainit = \Omegainit(\pi)$ be the
initial additive move ensemble (\autoref{sec:construct_initial_moves})
of~$\pi$.

\begin{assumption}
  \label{assumption:moves-completion}
  The moves closure $\ctscl(\Omegainit)$ has a finite presentation
  $(\Omega, \C)$ in reduced form
  (\autoref{s:finite-presentation-normal-form}).  
\end{assumption}
Thus $\Omega$ has finitely many moves and $\C$ has finitely many
connected covered components
$\Comp[1], \Comp[2], \dots, \Comp[k]$, each of which is a
finite union of open intervals.  Each $\gamma|_D\in\Omega$ is maximal in the
restriction partial order of~$\ctscl(\Omegainit)$ and is not contained in
$\joinmoves(\C)$.  Figures~\ref{fig:one-sided-discontinuous-pi1}
(right), \ref{fig:equiv7_example_xyz_2-completion},  and \ref{fig:equiv7_minimal_2_covered_2_uncovered}
show examples of $\ctscl(\Omegainit)$ satisfying
\autoref{assumption:moves-completion}.  \tred{I am not assuming \emph{normal
    form}, which makes more assumptions regarding $\C$ (see
  \autoref{s:finite-presentation-normal-form}). Check if those are used
  implicitly anywhere.}

\tred{Is a more general statement possible?  That does not need ALL properties
  of 'moves closure'?  Let's isolate the assumptions that are really used in
  the proof.}

\subsection{Properties of the finitely presented moves closure}

Let $C := \Comp[1] \cup \Comp[2] \cup \dots \cup \Comp[k]$ denote the open set of
points in $(0,1)$ that are covered.  We will refer to the open set $U := (0,1) \setminus
\cl(C)$ as the set of points in $(0,1)$ that are \emph{uncovered}.
Let
\begin{equation}\label{eq:def-X}
X := \{0\} \cup \partial C \cup \{1\} = \{0\} \cup \partial U \cup \{1\}
\end{equation}
be the set of endpoints of all covered and uncovered intervals.
Thus we have the partition $[0,1] = C \cup X \cup U$.

\newcommand\FACETOCOMPLETIONROW[1]{\begin{minipage}{.49\textwidth}
\centering
\includegraphics[width=\linewidth]{graphics-for-algo-paper/#1-2d_diagram}
\end{minipage}
\begin{minipage}{.49\textwidth}
\centering
\includegraphics[width=\linewidth]{graphics-for-algo-paper/#1-completion-final}
\end{minipage}}

\begin{example}\label{ex:equiv7_example_1}
\begin{figure}[tp]
\centering
\
\FACETOCOMPLETIONROW{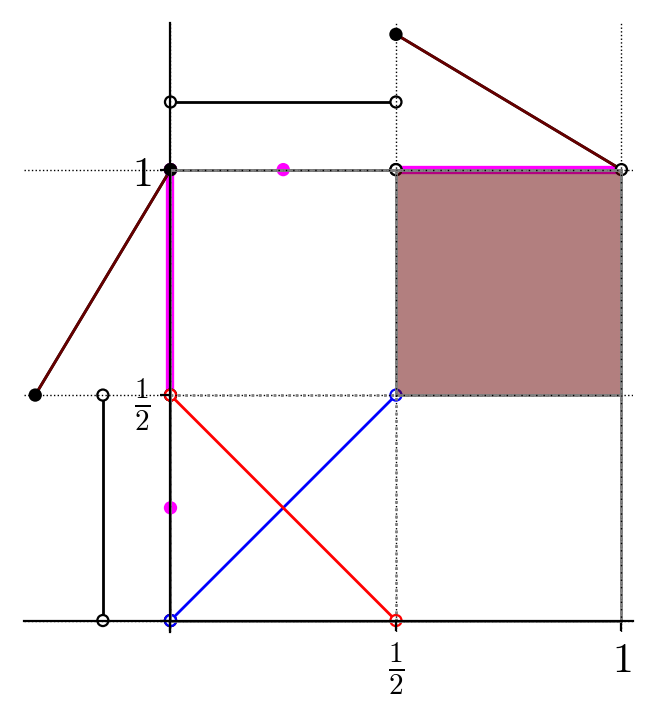}
\caption[(\textit{Left}) Two-dimensional polyhedral complex $\Delta\P$ of the
  one-sided discontinuous minimal valid function $\pi =
  \sage{equiv7\underscore{}example\underscore{}1()}$ from
  \autoref{ex:equiv7_example_1}. (\textit{Right}) The graph of the moves closure $\ctscl(\Omegainit)$
  of~$\pi$]
  {(\textit{Left}) Two-dimensional polyhedral complex $\Delta\P$ of the
  one-sided discontinuous minimal valid function $\pi =
  \sage{equiv7\underscore{}example\underscore{}1()}$ from
  \autoref{ex:equiv7_example_1} (\emph{blue graph at the 
    left and top borders}), where the additive faces are \emph{colored in
    green}. (\textit{Right}) The graph of the moves closure $\ctscl(\Omegainit)$
  of~$\pi$.  It has a finite presentation by $\Omega = 
  \{ \tau_0|_{(0,1/2)}, \rho_{1/2}|_{(0, 1/2)} \}$ (\emph{blue and red
    line segments of slopes $\pm1$})
  and one (maximal) connected covered component $C = (\frac12,1)$ (the
  \emph{brown square} shows $C\times C$). 
  The set $C \cup X \cup Y \cup Z$ of covered
  points and refined breakpoints is marked in \emph{magenta on the left and top borders}.}
\label{fig:one-sided-discontinuous-pi1}
\end{figure}
  Consider the discontinuous minimal valid function for $f=\tfrac{1}{2}$, defined by
  \[
    \pi(x) =
    \begin{cases}
      0 & \text{if } x=0 \\
      \tfrac{1}{2} &\text{if } 0< x< \tfrac{1}{2}\\
      2(1-x) & \text{if } \tfrac{1}{2} \leq x < 1.
    \end{cases}
  \]
  It is provided by the software~\cite{cutgeneratingfunctionology:lastest-release} as $\pi=
  \sage{equiv7\underscore{}example\underscore{}1()}$.
  \autoref{fig:one-sided-discontinuous-pi1} shows the two-dimensional
  polyhedral complex $\Delta\P$ and the moves closure $\ctscl(\Omegainit)$.
  The interval $C = (\tfrac12, 1)$ is covered, $U = (0, \tfrac12)$ is
  uncovered. 
  We have $X = \{0, \tfrac12, 1\}$.
\end{example}

\begin{example}\label{ex:equiv7_example_xyz_2}
\begin{figure}[t]
  \centering
  \hspace*{-1cm}%
  \includegraphics[width=.8\textwidth]{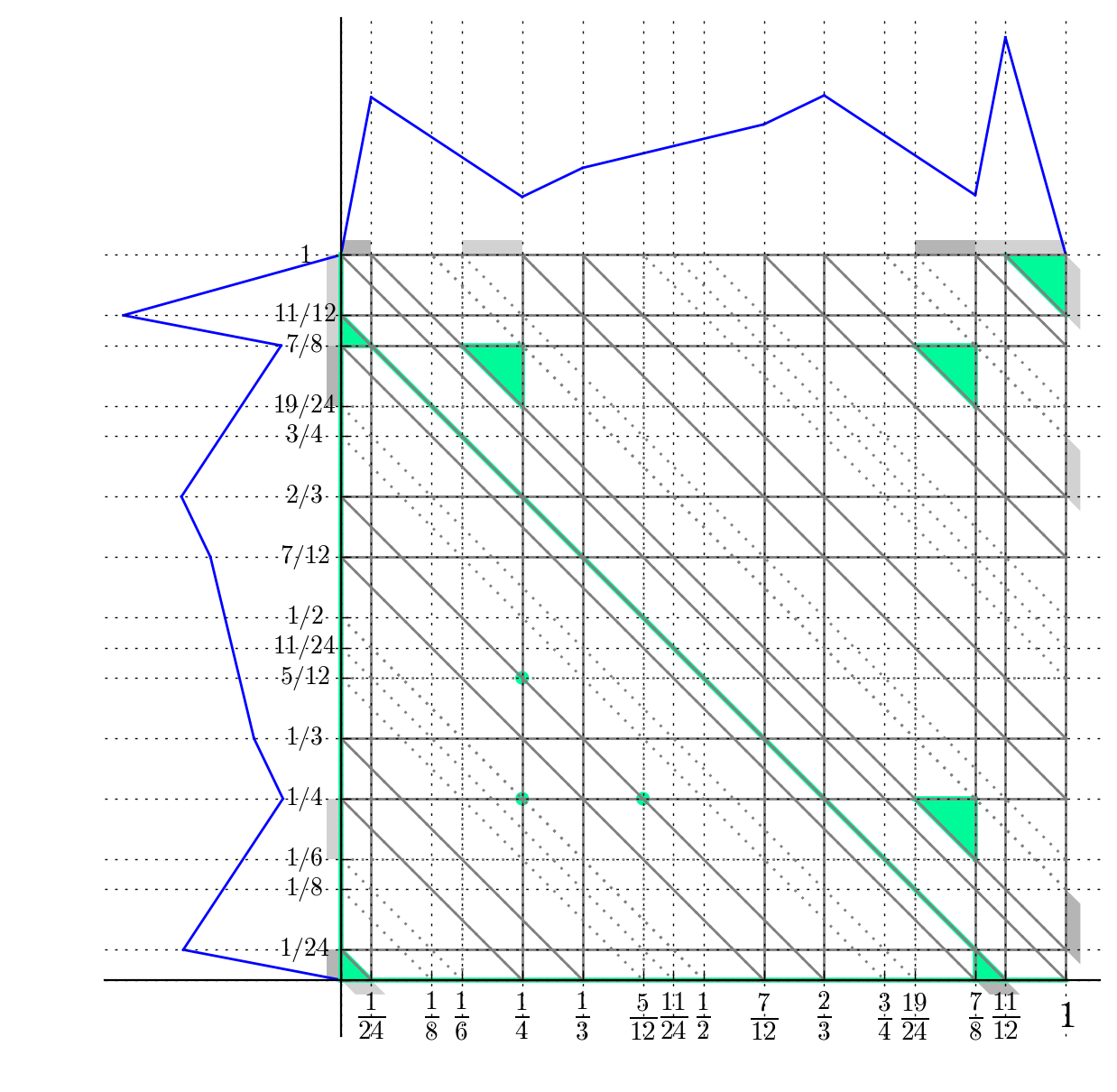}
  \caption[The function $\pi = \sage{equiv7\_example\_xyz\_2()}$ from
    \autoref{ex:equiv7_example_xyz_2}, its two-dimensional
    polyhedral complex $\Delta\P$, and the refined complex~$\Delta\T$]
    {The function $\pi = \sage{equiv7\_example\_xyz\_2()}$ from
    \autoref{ex:equiv7_example_xyz_2} (\emph{blue graph at the left and 
      top borders}) and its two-dimensional
    polyhedral complex $\Delta\P$ (\emph{solid gray lines}), where
    the additive faces are \emph{colored in green}.  The refined complex
    $\Delta\T$ is shown with \emph{dotted gray lines}.
  }
  \label{fig:equiv7_example_xyz_2}
\end{figure}
\begin{figure}[t]
  \centering
  \hspace*{-1cm}%
  \includegraphics[width=.8\linewidth]{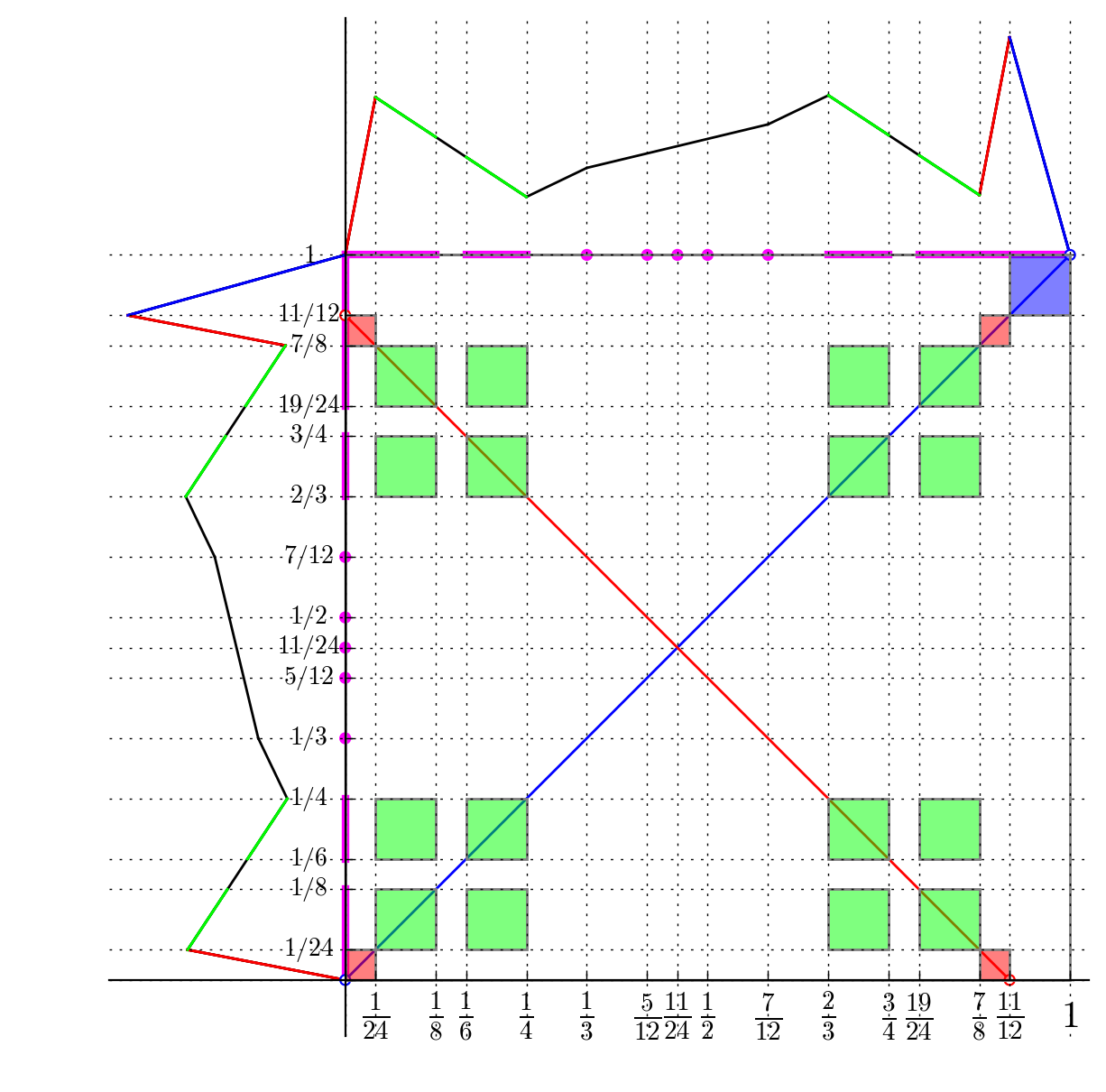}
  \caption[The function $\pi = \sage{equiv7\_example\_xyz\_2()}$ from
    \autoref{ex:equiv7_example_xyz_2} and the graph of the moves closure
    $\ctscl(\Omegainit)$ of~$\pi$] 
    {The function $\pi = \sage{equiv7\_example\_xyz\_2()}$ from
    \autoref{ex:equiv7_example_xyz_2} (\emph{colored graph at the left and 
      top borders}) and 
    the graph of the moves closure
    $\ctscl(\Omegainit)$ of $\pi$, as computed by the command \sage{igp.equiv7\_mode = True;
      igp.extremality\_test(igp.equiv7\_example\_xyz\_2(), True, show\_all\_perturbations=True)}.  It has a
    finite presentation by $\Omega = 
    \{ \tau_0|_{(0,1)}, \rho_{11/12}|_{(0, 11/12)} \}$ (\emph{blue and red
      line segments of slopes $\pm1$})
    and a set $ \C = \{\Comp[1], \Comp[2], \Comp[3]\}$ of (maximal) connected covered components
    $\Comp[1] = (\frac{11}{12}, 1)$ (the \emph{lavender square} shows $\Comp[1]\times\Comp[1]$),
    $\Comp[2] = (0, \frac1{24})\cup(\frac78,\frac{11}{12})$ (\emph{coral}),
    and $\Comp[3] = (\frac1{24}, \frac18) \cup (\frac16, \frac14) \cup
    (\frac23,\frac34) \cup (\frac{19}{24}, \frac78)$ (\emph{lime}).     
    The set $C\cup B' = C \cup X \cup Y \cup Z$ of covered
    points and refined breakpoints is marked in \emph{magenta on
      the left and top borders}. 
    }
    \label{fig:equiv7_example_xyz_2-completion}
\end{figure}
Consider the continuous minimal valid function $\pi$ that is provided as $\sage{equiv7\_\allowbreak{}example\_xyz\_2()}$ by
the software~\cite{cutgeneratingfunctionology:lastest-release}, shown in 
\autoref{fig:equiv7_example_xyz_2}. 
Figures~\ref{fig:equiv7_example_xyz_2}
and~\ref{fig:equiv7_example_xyz_2-completion} show the additive faces and the
moves closure.
We have $X = \{0, \frac1{24},
\frac18, \frac16, \frac14, \frac23, \frac34, \frac{19}{24}, \frac78, \frac{11}{12}, 1\}$.
\end{example}

\begin{example}\label{ex:equiv7_minimal_2_covered_2_uncovered}
\begin{figure}[t]
  \centering
  \includegraphics[width=\linewidth]{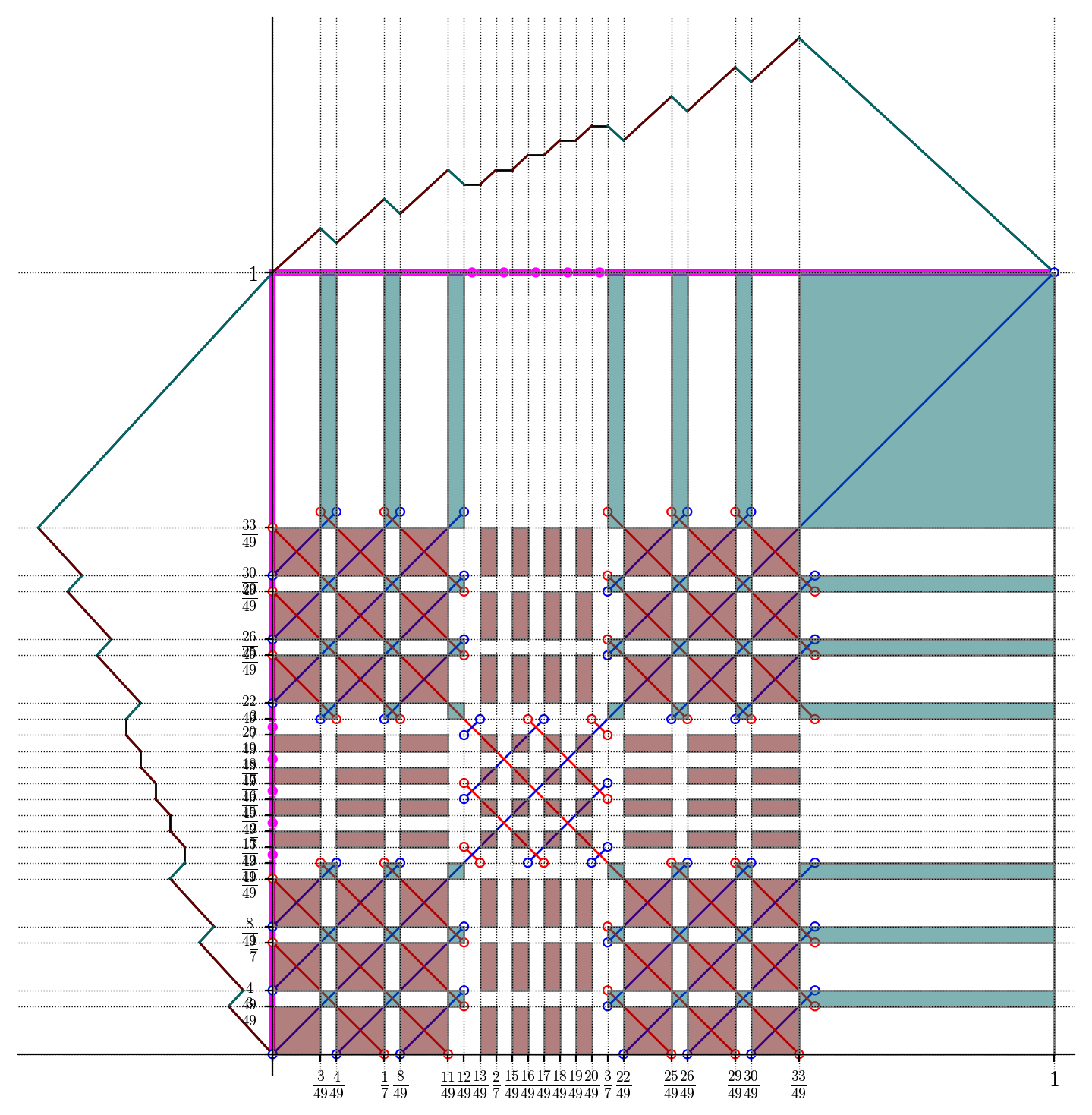}
  \caption[Moves closure $\clsemi{\Omegainit}$ for the function from
    \autoref{ex:equiv7_minimal_2_covered_2_uncovered}, $\pi =
    \sage{equiv7\_minimal\_2\_covered\_2\_uncovered()}$]
    {Moves closure $\clsemi{\Omegainit}$ for the function from
    \autoref{ex:equiv7_minimal_2_covered_2_uncovered}, $\pi =
    \sage{equiv7\_}\allowbreak\sage{minimal\_}\allowbreak\sage{2\_covered\_}\allowbreak\sage{2\_uncovered()}$.
    It has two connected covered components (\emph{cyan, brown rectangles}).}
  \label{fig:equiv7_minimal_2_covered_2_uncovered}
\end{figure}
Consider the minimal valid function $\pi$ that is provided as $
\sage{equiv7\_minimal\_2\_covered\_2\_uncovered()}$ by the
software~\cite{cutgeneratingfunctionology:lastest-release}; see
\autoref{fig:equiv7_minimal_2_covered_2_uncovered}. It has two
connected covered components.  The set of uncovered points is
$U = (\frac{12}{49}, \frac{13}{49}) \cup (\frac{14}{49}, \frac{15}{49}) \cup
\dots \cup (\frac{20}{49}, \frac{21}{49})$.
Thus we have
$X = \{0, \frac{12}{49}, \frac{13}{49}, \dots, \frac{20}{49}, \frac{21}{49},
1\}$.
\end{example}

Recall the two-dimensional polyhedral complex $\Delta\P_B$ and its additive
faces, introduced in \autoref{s:polyhedral-complex}.  Let
\begin{equation}\label{eq:def-V}
  V := \bigl\{\,p_i(x,y) \bigst (x,y) \text{
    additive vertex of }
  \Delta\P_B, \; i=1,2,3 \,\bigr\} \cap [0,1]
\end{equation}
be the set of $p_1, p_2$ and $p_3$ projections (within the fundamental domain)
of the zero-dimensional additive faces (i.e., additive vertices). 
Because $(x, 0)$ is an additive vertex of $\Delta\P_B$ for every $x \in B$,
the set~$V$ contains $B \cap [0,1]$.
By \autoref{rem:initial-moves-join-closed}, 
the initial move ensemble $\Omegainit$ is join-closed
.  \tblue{NEW:::}
We consider the ensemble $\Omegainit|_{U}$ of moves restricted
to $U$, as defined in \autoref{s:restrictions}. By
\autoref{lemma:restriction-of-restrict/join-is-restrict/join} it is also
join-closed and therefore, by \autoref{lemma:joined-is-presented-by-max}, has
a presentation by its maximal elements.
It follows from
\autoref{lemma:additive-is-subcomplex} that its
maximal elements have the following relation to the set~$V$.
\begin{lemma}\label{lemma:endpoints_max_initial}
  If $\gamma|_{(a,b)} \in \maxdom(\Omegainit|_{U})$, then the endpoints $a, b$ 
  lie in~$V \cap U$ or $\partial U$.
\end{lemma}

Next we define the set
\begin{equation}\label{eq:def-Y}
  Y := \Omega(V\cap U) = \{\,\gamma|_D(x) \st x \in V \cap U,\; x \in D \text{
    and }\gamma|_D \in \Omega\,\},
\end{equation}
the orbit of $V \cap U$ under $\Omega$, which is a finite set by \autoref{assumption:moves-completion}.  
\begin{remark}
  In terms of graphs of ensembles, we have
  \begin{equation}\label{eq:Y-via-graph}
    Y = \{\,y \mid \exists \, x \in V \cap U \text{ such that } (x, y) \in
    \graph(\Omega)\,\}.
  \end{equation}
\end{remark}

\begin{lemma}
 \label{lemma:subset-relation-Y-and-U} 
  We have (a) $V\cap U \subseteq Y$, and (b) $Y \subseteq U$.  
 \end{lemma}
\begin{proof}
(a) Let $x \in V\cap U$.  Since $\Delta\pi(x,y)=0$ for any $x \in \R$ when
$y=0$, we know that there is open interval $D$ with $x \in D \subseteq U$ such
that the idempotent $\tau_0|_D$ is in $\Omegainit$ and hence in $\restrict{\Omega}$ as well. Since $x=\tau_0|_D(x)$, we obtain that $x \in Y$.

(b) Suppose for the sake of contradiction that there is $y \in Y$ but $y \in \cl(C)$. 
We can write $y$ as $y=\gamma|_D(x)$ where $x \in V \cap U$, $x\in D$  and
$\gamma|_D \in \Omega$. Under  \autoref{assumption:moves-completion}, by
\autoref{lem:indirectly_covered_from_move} applied to $\C$ and
$\ctscl(\Omegainit)$, we have that $C$
is invariant under the action of moves from $ \ctscl(\Omegainit)$. Since the inverse move $(\gamma|_D)^{-1} \in \Omega \subseteq \ctscl(\Omegainit)$,
we obtain that  $x = (\gamma|_D)^{-1} (y) \in \cl(C)$. This contradicts $x \in U$.
 \end{proof}

\begin{example}[\autoref{ex:equiv7_minimal_2_covered_2_uncovered}, continued]
  In the example shown in \autoref{fig:equiv7_example_xyz_2-completion}, we
  have $V\cap U = \{\frac13, \frac5{12}, \frac12, \frac7{12}\}$.  This set is
  already closed under the action of $\Omega$, as
  $\rho_{11/12}(\frac13) = \frac7{12}$ and
  $\rho_{11/12}(\frac5{12}) = \frac12$. Thus $Y = V\cap U$ in the example.
\end{example}

We consider the ensembles $\Omega |_{U}$ and ${}_{U}|\Omega|_{U}$ of moves restricted and double-restricted to $U$, as defined in \autoref{s:restrictions}.
We have the following results.

\begin{lemma}
\label{lemma:omegau-uu-finite}
The move ensemble $\Omega |_{U}$
satisfies:
\begin{enumerate}[\rm(a)]
\item $\Omega |_{U} = {}_{U}|\Omega|_{U}$.
\item $\Omega |_{U}$ is a finite move ensemble.
\end{enumerate}
\end{lemma}
\begin{proof}
 It follows directly from
 \autoref{assumption:moves-completion}. 
 \tred{For (a), cite extend-component-by-move?}
\end{proof}

\begin{lemma}[\unboldmath Filtration of $\semi{\Omegainit|_U}$ by word length; maximal moves]
  \label{lemma:endpoints_maxsemi_omega0u}
  For $k\in \N$, 
  let 
  \begin{align*}
    \Omegainit|_U{}^k = \bigl\{\, \gamma^k|_{D^k} \circ \dots \circ \gamma^1|_{D^1} \bigst \,
    &\gamma^i|_{D^i} \in \Omegainit|_U \text{ for $1\leq i  \leq k$} 
           \,\bigr\}.
  \end{align*}  
  Then $\Omegainit|_U{}^1 \subseteq \Omegainit|_U{}^2 \subseteq \dots$ and
  $\semi{\Omegainit|_U} = \bigcup_{k\in\N} \Omegainit|_U{}^k$.  
  For each $k\in\N$, the ensemble $\Omegainit|_U{}^k$ satisfies
  \eqref{axiom:restrict} and has a presentation by the set
  $\maxdom(\Omegainit|_U{}^k)$ of its maximal elements, which is a finite set. For 
  $\gamma|_{(a,b)} \in \maxdom( \Omegainit|_U{}^k)$, we have $a, b,
  \gamma(a), \gamma(b) \in  X \cup Y$.
\end{lemma}
\begin{proof}
Because $\Omegainit$ satisfies \eqref{axiom:inverse},
\eqref{axiom:join}, and \eqref{axiom:restrict} by
\ref{rem:initial-moves-join-closed}, so does its double restriction ${}_U|\Omegainit|_U$ to the
uncovered set~$U$.  
By \autoref{lemma:omegau-uu-finite}-(a), we have  $\im(\Omegainit|_U) \subseteq \im(\Omega|_U) = \im({}_U|\Omega|_U) \subseteq U$, hence $\Omegainit|_U =  {}_U|\Omegainit|_U$. 
\tblue{This is the new version:}%
Recall that $\UOOU$ is join-closed and therefore has a presentation
  by its maximal elements.  It follows from the 
  definition of $\Omegainit$, \autoref{assumption:piecewise} with
  \autoref{rem:limit-additivites-stratified}, and \autoref{lem:squares} that
  $\maxdom(\UOOU)$ is finite.

  Let
  $\maxdom(\UOOU)^k = \{\, \gamma^k|_{\hat D^k} \circ \dots \circ
  \gamma^1|_{\hat D^1} \st \gamma^i|_{\hat D^i} \in \maxdom(\UOOU) \,\}$,
  a finite set.  Then $ \maxdom(\UOOUk) \subseteq \maxdom(\UOOU)^k$ is finite,
  and every element of $\UOOUk$ is the restriction of an element
  of~$\maxdom(\UOOUk)$.
  \begin{doublestaruseless}
    Indeed, let $\gamma|_D \in \UOOUk$, so
    $\gamma|_D = \gamma^k|_{D^k} \circ \dots \circ \gamma^1|_{D^1}$, where
    $\gamma^i|_{D^i} \in \UOOU$.  Replacing $\gamma^1|_{D^1}$ by a restriction
    of it, we can obtain any restriction of~$\gamma|_D$, so $\UOOUk$
    satisfies~\eqref{axiom:restrict}.  On the other hand, replacing the
    factors by maximal elements,
    $\gamma^i|_{D^i} \subseteq \gamma^i|_{\hat D^i} \in \maxdom(\UOOU)$, we
    obtain
    $\gamma|_D \subseteq \gamma^k|_{\hat D^k} \circ \dots \circ
    \gamma^1|_{\hat D^1} \in \maxdom(\UOOU)^k$.%
\end{doublestaruseless}
  The chain of inclusions
  $\Omegainit|_U{}^1 \subseteq \Omegainit|_U{}^2 \subseteq \dots$ holds because
  the idempotents $\tau_0|_D$ for intervals $D \subseteq U$ are elements of
  $\Omegainit|_U$.

  Last, we prove the claim regarding the endpoints; we actually
  prove the slightly stronger claim  $a, b, \gamma(a), \gamma(b) \in  \partial
  U \cup Y$ by induction by word length~$k$. Since each $\UOOUk$ satisfies
  \eqref{axiom:inverse}, it suffices to prove $a,b\in \partial U\cup Y$.
  For $k=1$,
  let $\gamma|_{(a,b)} \in \maxdom(\UOOU^1) = \maxdom(\UOOU)$.
  \tblue{This is the new version:}%
  Then, by
  \autoref{lemma:endpoints_max_initial}, each of the endpoints $a$, $b$ lies
  in~$V\cap U \subseteq Y$, or it lies in~$\partial U$. 
  Now we proceed by induction. 
  Take $\gamma^1|_{(a,b)} \in \maxdom(\UOOU)$ and $\gamma^2|_{(c,d)} \in
  \maxdom(\UOOU^{k-1})$, so $a,b,c,d \in \partial U\cup Y$.  
  Then, by \eqref{eq:composition}, $\gamma^2|_{(c,d)} \circ
  \gamma^1|_{(a,b)}$ has domain $(a,b) \cap (\gamma^1)^{-1}((c,d))$. 
  If the domain is nonempty, let $x$ be an endpoint of it.
  If $x = a, b$, nothing is to show, so assume $x \in (a,b)$ and $x =
  (\gamma^1)^{-1}(y)$, where $y = c$ or $y = d$, so $y \in \partial U\cup Y$. 
  But $y = \gamma^1|_{(a,b)}(x) \in U$, so $y \in Y$.  Then it follows 
  that
  also $x \in Y$. 
\end{proof}

By \autoref{assumption:moves-completion}, all elements of $\Omega$ are maximal moves
of the moves closure $\clsemi{\Omegainit}$.  Therefore, by
\autoref{lemma:restriction-of-maxdom-of-join-is-max}, all elements of $\Omega
|_{U}$ are maximal moves of $\clsemi{\Omegainit}|_{U}$. \smallbreak

After these preliminaries, we are able to state the main theorem.   
\begin{theorem}[Structure and generation theorem for finitely presented moves
  closures]
\label{thm:endpoints-of-maximal-domain-move}
Under \autoref{assumption:moves-completion}, we have
\begin{enumerate}[\rm(a)]
\item $\clsemi{\Omegainit} = \joinextend{\,\clsemi{\Omegainit|_U} \cup\clsemi{\Omegainit|_C}\,}$.
\item $\Omega|_U =  \maxdom(\joinextend{\semi{\Omegainit|_U}})$.
\item $a, b, \gamma(a), \gamma(b) \in X \cup Y$ for any $\gamma|_{(a,b)} \in \Omega |_{U}$.
\end{enumerate}
\end{theorem}
We emphasize that the theorem does \emph{not} depend on an algorithm to
compute the moves closure.
\begin{proof}
\emph{Part (a).}
Let $\Omega' $ denote the right hand side of the equation in part~(a).
Clearly, 
$\Omegainit \subseteq \Omega' \subseteq \clsemi{\Omegainit}$.
We now show that $\Omega'$ is a closed move semigroup.
By \autoref{lemma:omegau-uu-finite}-(a), we have that
\begin{align*}
\clsemi{\Omegainit|_U} &\subseteq \restrict{\Omega|_U} \subseteq \MovesOfGraph(U \times U); \\
\clsemi{\Omegainit|_C} &= \moves (\CC) \subseteq \MovesOfGraph(C \times C),
\end{align*}
where the open sets $U$ and $C$ are disjoint.
Thus, we have that  $\clsemi{\Omegainit|_U} \cup \clsemi{\Omegainit|_C}$ is a 
move semigroup, under \autoref{assumption:moves-completion}.
It follows from  \autoref{lemma:jextendsemi-closed} that $\Omega'$ is a move semigroup that satisfies \eqref{axiom:joinextend}.
Note that for any open intervals $D$ and $I$ such that 
$\MovesOfGraph(D\times I) \subseteq \clsemi{\Omegainit}$, we have $\MovesOfGraph(D\times I) \subseteq \clsemi{\Omegainit|_C}$. 
Therefore, $\Omega'$ also satisfies
\eqref{axiom:translation-reflection}. Moreover, \eqref{axiom:limits} holds by \autoref{thm:dense-boxes-in-joinlim}.
We conclude that $\Omega'$ is a closed move semigroup. Hence, part (a) holds.\medbreak

\emph{Part (b).}
 By restricting the moves ensembles on both sides of the equation in part (a) to domain $U$, we obtain that
\begin{equation}
\label{eq:omega_u_clsemi}
\restrict{\Omega|_U}  = \clsemi{\Omegainit}|_U = \clsemi{\Omegainit|_U}
\end{equation}
Next, we show that 
\begin{equation}
\label{eq:omega_u_no_lim}
\clsemi{\Omegainit|_U} = \joinextend{\semi{\Omegainit|_U}}.
\end{equation}
It follows from \autoref{lemma:jextendsemi-closed} that $\joinextend{\semi{\Omegainit|_U}}$ is a move semigroup that satisfies \eqref{axiom:joinextend} (and also \eqref{axiom:join} and \eqref{axiom:restrict}).
Since
\begin{equation}
\label{eq:jextendsemi-subset-clsemi}
\joinextend{\semi{\Omegainit|_U}} \subseteq \clsemi{\Omegainit |_{U}} = \restrict{\Omega|_U},
\end{equation}
where the equality follows from \eqref{eq:omega_u_clsemi}, and $\Omega |_{U}$ is a finite move ensemble by \autoref{lemma:omegau-uu-finite}-(b), 
we obtain that the move semigroup $\joinextend{\semi{\Omegainit|_U}}$ also satisfies \eqref{axiom:translation-reflection} and \eqref{axiom:limits}.  
Therefore, $\joinextend{\semi{\Omegainit|_U}}$ is a closed move semigroup which contains $\Omegainit|_U$. 
Since $\clsemi{\Omegainit |_{U}}$ is the smallest closed move semigroup containing $\Omegainit|_U$, we have
$$\clsemi{\Omegainit |_{U}}  \subseteq \joinextend{\semi{\Omegainit|_U}}.$$
Together with \eqref{eq:jextendsemi-subset-clsemi}, we conclude that \eqref{eq:omega_u_no_lim} holds.
Since $\Omega$ has only maximal moves, \eqref{eq:omega_u_clsemi} and \eqref{eq:omega_u_no_lim} imply the equation in part (b).
\medbreak

\emph{Part (c).}  Let $\gamma|_{(a,b)} \in \Omega
|_{U}$. By symmetry, it suffices to show that $a, b \in X \cup Y$. Consider $x=a$ or $x=b$.
Part (b) implies that \[\Omega|_U =  \maxdom(\joinextend{\joinsemi{\Omegainit|_U}}).\]
Together with \eqref{eq:jextend-equiv-def}, we know
that $x$ is the limit of a sequence $\{x^j\}_{j \in \N}$, where $x^j$ is an endpoint of the
domain $D^j$ of a move $\gamma|_{D^j} \in \maxdom(\joinsemi{\Omegainit|_U})$.
By \autoref{lemma:endpoints_maxsemi_omega0u} and \autoref{lemma:joined-ensemble-explicitly}, 
for any $j \in \N$,
we have that $D^j$ is a maximal subinterval of 
$\bigcup \{\,D \st \gamma|_D \in \bigcup_{k\in\N} \maxdom(\UOOUk)\,\}$.
Thus for every $j\in \N$, there exists a sequence $\{x^j_k\}_{k \in \N}$ such that 
each $x^j_k$ is an endpoint of the domain of a move  $\gamma|_{D^j_k} \in \maxdom(\UOOUk)$,
and $x^j_k \to x^j$ as $k \to \infty$.
We obtain that $x^k_k \to x$ as $k \to \infty$,
where each $x^k_k \in X\cup Y$ by  \autoref{lemma:endpoints_maxsemi_omega0u}.
Since $X \cup Y$ is a finite discrete set under \autoref{assumption:moves-completion}, we obtain that $x \in X \cup Y$.
\end{proof}

\subsection{Refined breakpoints $B'$, complex $\T$}

In addition to the finite sets $X$ and $Y$, we define 
\begin{equation} \label{eq:def-Z}
  Z  := \{\, x \st  x \in U,\; x = \rho|_D(x) \text{ for some reflection move
  } \rho|_D \in \Omega \,\},
\end{equation}
the set of \emph{uncovered character conflicts}.
\begin{remark}
In terms of $\graph_+$ and $\graph_-$ notations, the set 
$Z$ is the set of projections of the intersection of the translation and
reflection moves graphs restricted to the uncovered intervals, 
$Z = \{\,x \mid  x \in U, \; (x, x) \in \graph_\pm(\Omega)\,\}$.  
\end{remark}

\begin{example}[\autoref{ex:equiv7_minimal_2_covered_2_uncovered}, continued]
  In the example shown in \autoref{fig:equiv7_example_xyz_2-completion}, we
  have $Z = \{ \frac{11}{24} \}$.
\end{example}

\begin{theorem}
\label{thm:xyz-invariant-under-moves}
Under \autoref{assumption:moves-completion}, the sets $X$, $Y$, and $Z$ are
closed under the action of all moves from $\ctscl(\Omegainit)$.
\end{theorem}
\begin{proof}
Let $\gamma|_D$ be a move in $\ctscl(\Omegainit)$. 

Let $x \in X$ such that $x \in D$. Since $C$ is invariant under the action of all moves from $\ctscl(\Omegainit)$, we have that $\gamma|_D(x) \in X$. 

Let $y \in Y$ such that $y \in D$. There exist $x \in V\cap U$ and $\gamma'|_{D'} \in \Omega$ such that $\gamma'|_{D'}(x) = y$. We have $\gamma|_D(y) =\gamma|_D\circ \gamma'|_{D'}(x)$, where $\gamma|_D\circ \gamma'|_{D'} \in \restrict{\Omega}$. Therefore, $\gamma|_D(y) \in Y$. 

Let $z \in Z$ such that $z \in D$. By definition, $z \in U$ and $z = \rho|_{D'}(z)$ for some reflection move $\rho|_{D'} \in \Omega$. Let $z' = \gamma|_D(z)$. We have that $z' \in U$ and $z' = \gamma|_D \circ \rho|_{D'} \circ (\gamma|_D)^{-1} (z')$, where $\gamma|_D \circ \rho|_{D'} \circ (\gamma|_D)^{-1} \in \restrict{\Omega}$. Therefore,  $z'=\gamma|_D(z) \in Z$.
\end{proof}

Under \autoref{assumption:moves-completion}, the sets $X$, $Y$, $Z$ are
finite. 
We then define~$B'$, which is a finite set of points under \autoref{assumption:moves-completion}, a
\emph{refined set of breakpoints},
\begin{equation}
  B' := (X \cup Y \cup Z) + \Z.\label{eq:def-Bprime}
\end{equation}
By \autoref{assumption:minimal-B}, $B \cap C = \emptyset$ and
thus $B \subseteq B'$.  Hence, the polyhedral complex 
\begin{equation}
  \T := \P_{B'},\label{eq:def-T}
\end{equation}
is a refinement of $\P_B$, so our function $\pi$ is piecewise linear over $\T$. The
following result shows that each of the $p_1, p_2$ and $p_3$ projections of any
additive vertex of the two-dimensional polyhedral complex $\Delta \T$ is
either in~$B'$ or covered by $C$. 

\begin{theorem}[Breakpoint stabilization theorem]
\label{lemma:projections-of-vertDeltaT}
Let $(x,y)$ be an additive vertex of  $\Delta\T$. 
Let $z = x+y$. Then, $x, y, z \in B' \cup (C+\Z)$.
\end{theorem}
\begin{proof}
Let $F$ be the unique face of $\Delta\P_B$ such that $(x,y) \in \relint(F)$. 
Since $(x,y)$ is an additive vertex of  $\Delta\T$,
and $\Delta\pi$ is non-negative and affine linear over $F$, we have that $F$ is an additive face of $\Delta\P_B$. 
Consider $t = x, y$ or $z$. By $\Z$-periodicity, we can assume $t \in [0,1]$.
To show that $t\in (B'\cap[0,1]) \cup C$, we distinguish three cases, as follows.
We recall that $B'\cap[0,1] = X \cup Y \cup Z$ and $U = (0,1) \setminus \cl(C)$. 

Assume that $F$ is a zero-dimensional additive face of $\Delta\P_B$. Then, $(x,y)$ is an additive vertex of $\Delta\P_B$, and thus $t \in V$. If $t \in \cl(C)$, then $t \in X \cup C \subseteq B' \cup C$. Otherwise, $t \in V \cap U$. Since $V\cap U \subseteq Y$ by \autoref{lemma:subset-relation-Y-and-U}, we obtain that $ t \in Y \subseteq B'$.

Assume that $F$ is a one-dimensional additive face (say, a horizontal additive edge) of $\Delta\P_B$.
Then, $y \in B  \subseteq B'$ and the move $\tau_y|_D$ with $x \in D :=
\intr(p_1(F))$ is in $\Omegainit$. Since $(x,y)$ is a vertex of $\Delta\T$, at
least two of $x, y, z$ are in $B'$, and hence at least one of $x$ and $z$ is
in $B'$. Without loss of generality, we assume that $x \in B'$. By
\autoref{thm:xyz-invariant-under-moves}, $z = \tau_y|_D(x) \in B'$ as
well. We showed that $x, y, z \in B'$ in this case. We omit the proof of the
cases where $F$ is a vertical or diagonal additive edge of $\Delta\P_B$, which
are similar to the above proof. 

Assume that $F$ is a two-dimensional additive face of $\Delta\P_B$. Then, by
\autoref{lem:squares}, we have $t \in C$.
\end{proof}

\begin{remark}
  \autoref{lemma:projections-of-vertDeltaT} is key to our grid-free theory.
  In the grid case of \cite{basu-hildebrand-koeppe:equivariant}, where
  $B = \tfrac{1}{q} \Z$, the projections $p_1\colon (x,y)\mapsto x$,
  $p_2\colon (x,y)\mapsto y$, and $p_3\colon (x,y)\mapsto x+y$ map all
  vertices of $\Delta \P_B$ back to the set~$B$. We have stabilization of
  breakpoints due to unimodularity.  Going to higher dimension (minimal valid
  functions of several variables), the piecewise linear functions defined on a
  standard triangulation of $\R^2$ studied in
  \cite{bhk-IPCOext,basu-hildebrand-koeppe:equivariant-general-2dim} also
  stabilize.  However, the non-existence of triangulations with stabilization
  for~$\R^k$, $k\geq 3$ \cite{hildebrand:group-subdivisions} blocks the path
  for further generalizations of the approach of
  \cite{basu-hildebrand-koeppe:equivariant,bhk-IPCOext,basu-hildebrand-koeppe:equivariant-general-2dim}.
  Our \autoref{lemma:projections-of-vertDeltaT} depends on more detailed data
  of the function than the group~$G$ generated by~$B$.  This ``dynamic''
  stabilization result could pave the way to generalizations to higher dimension.
\end{remark}

\subsection{Connected uncovered components $\texorpdfstring{\Uncomp[i]}{U\unichar{"1D62}}$}
\label{s:connected-uncovered-components}

Define
\begin{equation}
U' := U \setminus B'.\label{eq:def-Uprime}
\end{equation}
The interval $[0,1]$ is partitioned into the set $C$ of covered points, the
set $U'$ of uncovered points and the set $B'\cap[0,1]$ of breakpoints of $\T$.  
We consider the ensemble $\Omega |_{U'}$ of maximal moves restricted to $U'$ as defined in \autoref{s:restrictions}.
\autoref{lemma:omegau-uu-finite} and Theorems \ref{thm:endpoints-of-maximal-domain-move} and \ref{thm:xyz-invariant-under-moves} imply the following corollary.
\begin{corollary}
\label{cor:maximal-moves-set}
Under Assumptions~\ref{assumption:minimal-B} and~\ref{assumption:moves-completion}, 
the move ensemble $\Omega |_{U'}$
satisfies that:
\begin{enumerate}[\rm(a)]
\item $\Omega |_{U'} = {}_{U'}|\Omega|_{U'}$.
\item $\Omega |_{U'}$ is a finite move ensemble.
\item For any $\gamma|_D \in \Omega |_{U'}$, $\cl(D)$ and $\cl(\gamma(D))$ are faces of $\T$.
\end{enumerate}
\end{corollary}

We partition the set of uncovered points $U'$ into the (maximal)
\emph{connected uncovered components} $\{\Uncomp[1], \dots, \Uncomp[l]\}$, as
follows.\footnote{This extends the terminology of
  \cite{basu-hildebrand-koeppe:equivariant} where connected components are
  grid-based.} A connected uncovered component $\Uncomp[i]$ ($1 \leq i \leq
l$) is a maximal subset of~$U'$ that is the disjoint union of all the uncovered intervals $I_1, I_2, \dots,  I_p\subseteq U'$ such that any pair of intervals $I_j$ and $I_k$ ($1 \leq j, k \leq p$)  are connected by a maximal move $\gamma|_{I_k} \in \Omega |_{U'}$ with domain $I_k$ and image $I_j = \gamma(I_k)$. 
\begin{remark}
\label{rk:maximal-moves-union}
The set $\Omega |_{U'}$ only has moves $\gamma|_D$ whose domain~$D$ and image~$\gamma(D)$ are both contained in the same $\Uncomp[i]$, for $i =1, 2,\dots, l$.
\end{remark}
Since the function $\pi$ is piecewise linear over $\T$ and it respects $\Omega |_{U'}$, we have that $\pi$ is affine linear with the same slope on the maximal intervals $I_1, I_2, \dots,  I_p$ of the same connected uncovered component $\Uncomp[i]$.
Since an effective perturbation  $\tilde{\pi} \in \tildePi^{\pi}\RZ$ also respects $\Omega |_{U'}$, it takes the same shape on the uncovered intervals $I_1, I_2, \dots,  I_p \subseteq  \Uncomp[i]$.
We pick $D \in \{I_1, I_2, \dots,  I_p\}$ arbitrarily as the \emph{fundamental domain}, and write $I_j = \gamma_j(D)$ where $\gamma_j|_{D} \in \Omega |_{U'}$ for $j=1, 2, \dots, p$. Then, the connected uncovered component $\Uncomp[i] \subseteq U'$ can be written as $\Uncomp[i]  = \bigcup\gamma_j(D)$.

\subsection{Finite-dimensional and equivariant perturbation subspaces}

Under \autoref{assumption:one-sided-continuous}, we define the following
spaces.

\begin{definition}
  Define the \emph{space of finite-dimensional perturbations} that are piecewise linear over $\T$: 
  \begin{equation}
  \label{eq:finite-dim-perturbations}
  \tildePi^{\pi}_{\T}\RZ := \bigl\{\, \tilde\pi \in \tildePi^{\pi}\RZ \bigst
  \tilde\pi \text{ is piecewise linear over } {\T}\,\bigr\}.
   \end{equation}
\end{definition}
Thus, functions in $\tildePi^{\pi}_{\T}\RZ$ are allowed to be discontinuous. 
\begin{definition}
  Define the \emph{space of equivariant perturbations} that vanish on the vertices of $\T$:
\[ \tildePi^{\pi}_{\mathrm{zero}(\T)}\RZ := \biggl\{\,\tilde\pi \in \tildePi^{\pi}\RZ \biggst  \tilde\pi(t)=\lim_{\substack{x \to t\\ x < t}} \tilde\pi(t) = \lim_{\substack{x \to t\\ x > t}} \tilde\pi(t) = 0,\;
  \forall  t \in \verts(\T) \,\biggr\}.\]
\end{definition}
We will show in \autoref{thm:equiv-perturbation-characterization} that all
functions in~$\tildePi^{\pi}_{\mathrm{zero}(\T)}\RZ$ are Lipschitz
continuous.
We will also show that the space is equivariant under the action
of~$\clsemi{\Omegainit}$, in the sense of \autoref{s:space-equivariant}.  This
will justify the name.

\begin{remark}
   In \autoref{lemma:effective-perturbation-vector-space} we showed that 
   the space $\tildePi^{\pi}\RZ$ of effective perturbations is a vector
   space. 
   The space $\tildePi^{\pi}_{\T}\RZ$ of finite-dimensional
   perturbations and the space  $\tildePi^{\pi}_{\mathrm{zero}(\T)}\RZ$
   of equivariant perturbations are vector subspaces of it. 
 \end{remark}
 
\begin{remark}
The vector spaces $\tildePi^{\pi}_{\T}\RZ$ and $\tildePi^{\pi}_{\mathrm{zero}(\T)}\RZ$ should not be confounded with the vector spaces $\bar\Pi^{E}_{\T}\RZ$ and $\bar\Pi^{E}_{\mathrm{zero}(\T)}\RZ$ with prescribed additivities $E  = \{\,(x,y)\st \Delta\pi(x,y)=0\,\}$, used in \cite[Lemma 3.14]{igp_survey}, where the function $\pi$ is assumed to be continuous piecewise linear over $\T$ with $\verts(\T)=\frac{1}{q}\Z$, $q \in \N$. 
\end{remark}

\subsection{Finite-dimensional linear algebra \texorpdfstring{for $\tilde \Pi^{\pi}_\T\RZ$}{}}
Let $\tilde\pi_\T \in \tildePi^{\pi}_{\T}\RZ$ be a finite-dimensional
perturbation.  Note that $\tilde\pi_\T$ is a piecewise linear function, and it
is uniquely determined by its values $\tilde\pi_\T(x)$ and limits
$\tilde\pi_\T(x^-) :=  \lim_{t \to x, t < x} \tilde\pi_\T(t)$,
$\tilde\pi_\T(x^+) :=  \lim_{t \to x, t >x} \tilde\pi_\T(t)$ at the
breakpoints $x \in B' + \Z = \verts(\T)$. 
\begin{lemma}
\label{lemma:finite-dim-system}
A function $\tilde\pi_\T \colon \R \to \R$ is a finite-dimensional perturbation, $\tilde\pi_\T \in \tildePi^{\pi}_{\T}\RZ$,  if and only if $\tilde\pi_\T$ is piecewise linear over $\T$ and satisfies the following conditions. 
\begin{enumerate}[\rm(i)]
\item $\tilde \pi_\T(0) = 0$ and $\tilde \pi_\T(f) = 0$;
\item $\tilde \pi_\T(x) = \tilde \pi_\T(x+t)$ for all $x \in \R,\, t \in \Z$;
\item For any additive vertex $(x,y)$ of $\Delta\T$ and any face $F \in \Delta\T$ such that $(x, y) \in F$, $\Delta\pi_F(x,y)=0$ implies $\Delta(\tilde\pi_\T )_F(x,y)=0$.
\end{enumerate}
\end{lemma}

Before we give the proof, we define another space
$\bar{\Pi}^{E_{\bullet}(\pi,\T)}\RZ$, following
\cite{koeppe-zhou:discontinuous-facets}. 
Recall from
\autoref{s:polyhedral-complex} the
family of
sets $E_F(\pi)$, indexed by faces $F$ of a polyhedral complex, 
which capture the set of additivities and limit-additivities of~$\pi$. 
Here we use this family with the
refined polyhedral complex~$\Delta\T$, considering $\pi$ as a piecewise linear
function on~$\T$.
\begin{definition}
  For a family $E_{\bullet} = \{ E_F \}_{F \in \Delta\T}$,
  define the \emph{space of perturbation functions with prescribed
    additivities and limit-additivities} $E_{\bullet}$,
  \begin{displaymath}
    \bar{\Pi}^{E_{\bullet}}\RZ = \left\{\bar{\pi} \colon \R \to \R \, \Bigg| \;
      \begin{array}{@{}r@{\;}c@{\;}l@{\quad}l@{}}
        \bar{\pi}(0) =
        \bar{\pi}(f) &=& 0 \\
        \Delta\bar{\pi}_F(x, y) &=& 0 & \text{ for } (x,y) \in E_F, \ F \in \Delta\T\\
        \bar{\pi}(x+t) &=& \bar{\pi}(x) & \text{ for } x \in \R,\, t \in \Z
      \end{array} \;\right\}.
  \end{displaymath}
\end{definition}

\begin{proof}[Proof of \autoref{lemma:finite-dim-system}]
We consider $\pi$ as piecewise linear over $\T$, which is a refinement of $\P_B$. 
Let $\tilde\pi_\T \in \tildePi^{\pi}_{\T}\RZ$.  Then by definition, $\tilde\pi_\T $ is also piecewise linear over $\T$. 
Since $\tilde\pi_\T \in \tildePi^{\pi}\RZ$, we have that $\tilde\pi_\T \in \bar\Pi^{E_{\bullet}}\RZ$, where
$E_{\bullet} = E_{\bullet}(\pi, \T)$ is the family of sets $E_F(\pi)$, indexed by $F\in \Delta\T$. 
Namely, $\tilde\pi_\T$ satisfies the conditions (i), (ii) and
\begin{enumerate}
\item[(iii')] For any face $F \in \Delta\T$ and any $(x, y)\in F$, if $\Delta\pi_F(x,y)=0$ then $\Delta(\tilde\pi_\T )_F(x,y)=0$.
\end{enumerate}  
The condition (iii') clearly implies (iii). Thus, we proved the ``only if'' direction.
Now let $\tilde\pi_\T$ be a piecewise linear function over $\T$ that satisfies (i)--(iii). 
Notice that function $\pi$ is subadditive and also piecewise linear over $\T$. Hence, the condition (iii) implies (iii'). We obtain that $\tilde\pi_\T \in \bar\Pi^{E_{\bullet}}\RZ$, where $E_{\bullet} = E_{\bullet}(\pi, \T)$. It then follows from \cite[Theorem 3.1]{koeppe-zhou:discontinuous-facets} that $\tilde \pi_\T \in \tilde \Pi^{\pi}\RZ$. Therefore, $\tilde\pi_\T \in \tildePi^{\pi}_{\T}\RZ$, we proved the ``if'' direction.
\end{proof}

Assume that $B' = \{x'_0=0, x'_1, \dots, x'_{n'-1}, x'_n=1\}$
and we identify
$\tilde\pi_\T(x)$ and $\tilde\pi_\T(x+t)$ for all $t\in
 \Z$. \autoref{lemma:finite-dim-system} shows that $\bigl(\tilde\pi_\T(x'_0\mbox{}^{-}),\, \tilde\pi_\T(x'_0),\,\allowbreak \tilde\pi_\T(x'_0\mbox{}^+),\,\allowbreak \tilde\pi_\T(x'_1\mbox{}^-),\, \allowbreak\dots,\allowbreak\, \tilde\pi_\T(x'_{n'-1}\mbox{}^-),\,\allowbreak \tilde\pi_\T(x'_{n'-1}),\,\allowbreak \tilde\pi_\T(x'_{n'-1}\mbox{}^+)\bigr)$ is a solution to the finite-dimen\-sional linear system defined by (i) and (iii). 
The interpolation of such a solution gives an effective perturbation function $\tilde\pi_\T \in \tilde \Pi^{\pi}_\T\RZ$.  We know that
$(0, 0,  \dots, 0)$ is a trivial solution. 
If a nontrivial solution exists, then its interpolation $ \tilde\pi_\T \not\equiv 0$, implying that the function $\pi$ is not extreme. 

\begin{remark}\label{rem:finite-dim-system-with-slope-variables}
In fact, one can reduce the number of variables in the above linear system of
equations to solve, by considering the connected components, as
follows. 
\autoref{cor:pwl-moves-closure-respected} and \eqref{eq:finite-dim-perturbations}
imply that
$\tilde\pi_\T$ is affine linear with the same slope over all the intervals
from a connected covered component $\Comp[i]$ ($i=1, 2, \dots, k$) or from a
connected uncovered component $\Uncomp[i]$ ($i=1, 2, \dots, l$). Let $\tilde s_1^c,
\dots, \tilde s_k^c$ and $\tilde s_1^u, \dots, \tilde s_l^u$ denote the corresponding slope
variables.  

In the discontinuous case, 
by 
\autoref{lemma:perturbation-lipschitz-continuous}%
, using \autoref{assumption:one-sided-continuous}, the perturbation
$\tilde\pi_\T$ can only be discontinuous at the points where $\pi$ is
discontinuous. Let the variables $\tilde d_i$ ($i=1, 2, \dots, m)$ denote the changes
of the value of $\tilde\pi_\T$ at the $m$ discontinuity points of $\pi$. 
In other words, the variables $\tilde d_i$ denote jumps $\tilde\pi_\T(x) - \tilde\pi_\T(x^-)$ when $\pi$ is discontinuous at $x$ on the left, or $\tilde\pi_\T(x^+)-\tilde\pi_\T(x)$ when $\pi$ is discontinuous at
$x$ on the right. 

Then, for any fixed $x \in \R$, the value $\tilde\pi_\T(x)$ is uniquely determined by the slope variables $\tilde s_i^c$ ($i = 1, 2, \dots, k$), $\tilde s_i^u$ ($i = 1, 2, \dots, l$) and the jump variables $\tilde d_i$ ($i = 1, 2,\dots, m$). These $k+l+m \leq 3n'$ variables satisfy the system of linear equations given by \autoref{lemma:finite-dim-system}, where $(0,0,\dots, 0)$ is a trivial solution. See \cite[Example 7.2]{hong-koeppe-zhou:software-paper}
for a concrete example.
\end{remark}

\subsection{Equivariant perturbation space $\texorpdfstring{\tildePi^{\pi}_{\textrm{zero}(\T)}}{}$}
\label{s:equivariant-perturbation-space}

Let $\tilde \pi_{\mathrm{zero}(\T)} \in \tilde \Pi^{\pi}_{\textrm{zero}(\T)}\RZ$ be an equivariant  perturbation of $\pi$. 
By \autoref{cor:affine-on-components} (or \autoref{cor:affine-on-component-finite-presentation}) and \autoref{cor:pwl-moves-closure-respected},
$\tilde \pi_{\mathrm{zero}(\T)}$ is affine linear on all
covered intervals. By definition,
$\tilde \pi_{\mathrm{zero}(\T)}(t) = \tilde
\pi_{\mathrm{zero}(\T)}(t^-)=\tilde \pi_{\mathrm{zero}(\T)}(t^+)=0$ for every
$t \in \verts(\T)$, and $\partial C \subseteq \verts(\T)$.  Therefore, $\tilde
\pi_{\mathrm{zero}(\T)}$ is zero on $\cl(C)$. If the set of uncovered points
$U' = \emptyset$, then $\tilde \pi_{\mathrm{zero}(\T)} \equiv 0$. Otherwise,
recall from \autoref{s:connected-uncovered-components} that $U'$ is partitioned into connected uncovered components $\Uncomp[1], \Uncomp[2],\dots, \Uncomp[l]$. 
The following theorem gives the characterization of the projection of a perturbation $\tilde \pi_{\mathrm{zero}(\T)}$ onto the space of functions with support contained in a connected uncovered component $\Uncomp[i]$.
\tred{Introduce name for the spaces.  Makes it easier to talk about them in examples.}
\begin{theorem}[Characterization of the equivariant perturbations 
  supported on an uncovered component]
\label{thm:equiv-perturbation-characterization}
Suppose that Assumptions~\ref{assumption:one-sided-continuous},\ref{assumption:minimal-B} and~\ref{assumption:moves-completion} hold.
Let $\Uncomp[i]  = \bigcup\gamma_j(D)$ be a connected uncovered component, where $D$ is the fundamental domain for $\Uncomp[i]$ and $\gamma_j|_{D} \in \Omega |_{U'}$ ($j=1, \dots, p$). Let $\tilde\pi_i \colon \R \to \R$  be a $\Z$-periodic function such that $\tilde\pi_i(x) =0$ for $x \not\in \Uncomp[i]$. Then $\tilde\pi_i \in \tilde \Pi^{\pi}_{\mathrm{zero}(\T)}\RZ$ if and only if
\begin{enumerate}[\rm(i)]
\item $\tilde\pi_i$ is Lipschitz continuous on $D$; \tred{Easier to understand
  perhaps: say ... on $\cl(D)$.}
\item $\tilde\pi_i(x)=\tilde\pi_i(x^-)=\tilde\pi_i(x^+)=0$ for $x \in \partial D$; 
\item $\tilde\pi_i(x) = \chi(\gamma_j) \tilde\pi_i(\gamma_j (x))$ for $x \in D$, $j=1, \dots, p$. 
\end{enumerate}
\end{theorem}
\begin{proof}
Let $\tilde\pi_i \in \tilde \Pi^{\pi}_{\textrm{zero}(\T)}\RZ$. Since $\pi$ is continuous on $D$,  by 
\autoref{lemma:perturbation-lipschitz-continuous}, $\tilde \pi_i$ is Lipschitz continuous on $\R$. Hence, the condition (i) holds. The condition (ii) is clearly satisfied, as $\tilde\pi_i(x)=\tilde\pi_i(x^-)=\tilde\pi_i(x^+)=0$ for each $x \in \verts(\T)$. Since  $\tilde\pi_i$ respects $\Omega |_{U'}$, the condition (iii) also holds.

Conversely, let $\tilde\pi_i \colon \R \to \R$  be a $\Z$-periodic function such that $\tilde\pi_i(x) =0$ for $x \not\in \Uncomp[i]$ and the conditions (i)--(iii) hold. 
It follows from (ii) that $\tilde\pi_i(x)=\tilde\pi_i(x^-)=\tilde\pi_i(x^+)=0$ for $x \in \partial \Uncomp[i]$. Since $\tilde\pi_i(x) =0$ for $x \not\in \Uncomp[i]$, we have 
\begin{equation}
\label{eq:equiv-perturb-vanish-outside-uncovered}
\tilde\pi_i(x)=\tilde\pi_i(x^-)=\tilde\pi_i(x^+)=0 \text{ for } x \in  [0,1] \setminus \Uncomp[i] \supseteq B' \cup C.
\end{equation}
We claim that  $\tilde\pi_i$ satisfies all the additivities (including the limits) that $\pi$ has. Indeed, let $F$ be a face of $\Delta\T$ and let $(x,y) \in F$ such that $\Delta\pi_F(x, y)=0$. We show that $(\Delta\tilde\pi_i)_F (x,y)=0$ by distinguishing the following three cases. 

(a) If $(x,y)$ is an additive vertex of $\Delta\T$, then by \autoref{lemma:projections-of-vertDeltaT} and \eqref{eq:equiv-perturb-vanish-outside-uncovered}, we have $(\Delta\tilde\pi_i)_F(x,y)=0$. 

(b) If $(x, y)$ is contained in the relative interior of an edge $F'$ of
$\Delta\T$, then $F' \subseteq F$ and $F'$ is an additive face of
$\Delta\T$. Consider the move $\gamma|_{D'}$ associated with $F'$. We have
either $D'$ and $\gamma(D') \subseteq (0,1) \setminus \Uncomp[i]$, or $D'$ and
$\gamma(D') \subseteq \Uncomp[i]$. In the former case, the claim holds because
of \eqref{eq:equiv-perturb-vanish-outside-uncovered}; and in the latter case,
the structure of $\Delta\T$ (\autoref{lemma:projections-of-vertDeltaT})
implies that $\gamma|_{D'} \in \Omega |_{U'}$, and thus
$(\Delta\tilde\pi_i)_F(x,y)=0$ by the condition (iii).

(c) If $(x, y)$ is contained in the relative interior of a two-dimensional face $F'$ of $\Delta\T$, then $F' = F$ is a two-dimensional additive face of $\Delta\T$. We have $x, y, (x+y) \bmod 1 \in C$, hence the claim follows from \eqref{eq:equiv-perturb-vanish-outside-uncovered}.  

We showed that $\tilde\pi_i \in \bar\Pi^{E_{\bullet}}\RZ$, where $E_{\bullet}
= E_{\bullet}(\pi, \T)$ is the family from the proof of \autoref{lemma:finite-dim-system}. Then, \cite[Theorem 3.1]{koeppe-zhou:discontinuous-facets} implies that $\tilde\pi_i \in \tilde \Pi^{\pi}\RZ$. Therefore, $\tilde\pi_i \in \tilde \Pi^{\pi}_{\textrm{zero}(\T)}\RZ$.
\end{proof}

For $i = 1,\dots,l$, denote the space of functions $\tilde\pi_i$ as in the
theorem by $\tildePi^\pi_{U_i}$.  It is independent of the choice of
fundamental domain.

\begin{theorem}[Direct sum decomposition of equivariant perturbations by
  uncovered components]
  \label{th:equiv-perturbation-decomposition}
  We have the direct sum decomposition $\tilde
  \Pi^{\pi}_{\mathrm{zero}(\T)}\RZ = \tildePi^\pi_{U_1} \oplus \dots \oplus
  \tildePi^\pi_{U_l}$, i.e., 
  if $\tilde \pi
  \in \tilde \Pi^{\pi}_{\mathrm{zero}(\T)}\RZ$, 
  then it has a unique decomposition
  \begin{math}
    \tilde{\pi}
    = \tilde{\pi}_1 +\tilde{\pi}_2 +\dots +\tilde{\pi}_l
  \end{math}
  such that $\tilde\pi_i \in \tildePi^\pi_{U_i}$ for $i = 1, \dots, l$.
\end{theorem}
\begin{proof}
Let $\tilde \pi \in \tilde \Pi^{\pi}_{\textrm{zero}(\T)}\RZ$. 
For $i = 1, 2, \dots, l$, define $\tilde\pi_i : \R \to \R$,  
$\tilde\pi_i (x) = \tilde \pi(x)$ if $x \in \Uncomp[i]$ and $\tilde\pi_i (x) =0$ otherwise.
Then 
\begin{equation}
  \tilde{\pi} = \tilde{\pi}_1 +\tilde{\pi}_2 +\dots +\tilde{\pi}_l, 
\end{equation}
where each $\tilde{\pi}_i$ ($i = 1, 2, \dots, l$) satisfies the conditions in \autoref{thm:equiv-perturbation-characterization}.
\end{proof}
Each of the component functions $\tilde\pi_i$ ($i = 1, 2, \dots, l$)
is supported on
the connected uncovered component $\Uncomp[i]$ and is obtained by 
choosing an arbitrary Lipschitz continuous template on the fundamental
domain~$D_i$, then by extending equivariantly to the other intervals through the moves in
$\Omega|_{U'}$.
 
\subsection{Decomposition theorem for effective perturbations}
The following perturbation decomposition theorem, a generalization of \cite[Lemma 3.14]{igp_survey} without assuming $\pi$ is continuous and $\verts(\T) = \frac{1}{q}\Z$, shows that the effective perturbation space $\tilde \Pi^{\pi}\RZ$ is the direct sum of the finite-dimensional perturbation space $\tildePi^{\pi}_{\T}\RZ$ and the equivariant perturbation space $\tildePi^{\pi}_{\mathrm{zero}(\T)}\RZ$.

\begin{theorem}[Perturbation decomposition theorem]
\label{thm:decomposition-perturbation}
Under Assumptions~\ref{assumption:one-sided-continuous},\ref{assumption:minimal-B} and~\ref{assumption:moves-completion},
for every effective perturbation $\tilde \pi \in \tilde \Pi^{\pi}\RZ$, there exist a unique finite-dimensional perturbation $\tilde\pi_\T \in \tilde \Pi^{\pi}_\T\RZ$ and a unique equivariant perturbation $\tilde  \pi_{\mathrm{zero}(\T)} \in \tilde \Pi^{\pi}_{\mathrm{zero}(\T)}\RZ$ such that 
\[ \tilde \pi = \tilde \pi_\T + \tilde \pi_{\mathrm{zero}(\T)}. \]
\end{theorem}

\begin{proof}
Let $\tilde \pi \in \tilde \Pi^{\pi}\RZ$ be an effective perturbation. By \cite[Corollary 6.5]{hong-koeppe-zhou:software-paper}
, the limits $\tilde\pi(t^-)$ and $\tilde\pi(t^+)$ exist for every $t \in \verts(\T)$. Let $\tilde\pi_\T$ be the unique piecewise linear function over $\T$ such that $\tilde\pi_\T(t)=\tilde\pi(t)$, $\tilde\pi_\T(t^-)=\tilde\pi(t^-)$ and $\tilde\pi_\T(t^+)=\tilde\pi(t^+)$ for every $t \in \verts(\T)$.  Define $\tilde  \pi_{\mathrm{zero}(\T)} = \tilde \pi - \tilde \pi_\T$. Note that $\tilde\pi_\T$ is the unique piecewise linear function over $\T$ such that $\tilde  \pi_{\mathrm{zero}(\T)}(t)=\tilde  \pi_{\mathrm{zero}(\T)}(t^-)=\tilde  \pi_{\mathrm{zero}(\T)}(t^+)=0$ for every $t \in \verts(\T)$. It is left to show that $\tilde \pi_\T, \tilde \pi_{\mathrm{zero}(\T)} \in \tilde \Pi^{\pi}\RZ$.   

We first show that $\tilde \pi_\T \in \tilde \Pi^{\pi}\RZ$, by applying \autoref{lemma:finite-dim-system}. It suffices to show that $\tilde\pi_\T$ satisfies condition (iii) of \autoref{lemma:finite-dim-system}. Let $(x,y)$ be an additive vertex of a face $F \in \Delta\T$ with $\Delta\pi_F(x,y)=0$. By \cite[Lemma 6.1]{hong-koeppe-zhou:software-paper}
, $\Delta\pi_F(x,y)=0$ implies that $\Delta\tilde\pi_F(x,y)=0$. Since $(x,y)$ is an additive vertex of $\Delta\T$, \autoref{lemma:projections-of-vertDeltaT} implies that $x, y, z \in B' \cup C$, where $z=(x+y)\bmod 1$. We have $\tilde\pi_\T(t)=\tilde\pi(t)$, $\tilde\pi_\T(t^-)=\tilde\pi(t^-)$ and $\tilde\pi_\T(t^+)=\tilde\pi(t^+)$ for $t = x, y$ or~$z$, and hence $\Delta(\tilde\pi_\T )_F(x,y)=\Delta\tilde\pi_F(x,y)=0$.  Therefore, $\tilde \pi_\T \in \tilde \Pi^{\pi}\RZ$.

Since the vector space $\tilde \Pi^{\pi}\RZ$ contains both $\tilde \pi$ and $\tilde \pi_\T$, we obtain that $\tilde \pi_{\mathrm{zero}(\T)} = \tilde \pi - \tilde \pi_\T \in \tilde \Pi^{\pi}\RZ$. 
\end{proof}

\begin{example}[\autoref{ex:equiv7_example_1}, continued]
\begin{figure}[tp]
\centering
\begin{minipage}{.32\textwidth}
\centering
\includegraphics[width=\linewidth]{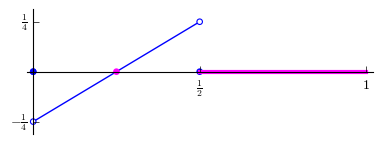}
\end{minipage}
\begin{minipage}{.32\textwidth}
\centering
\includegraphics[width=\linewidth]{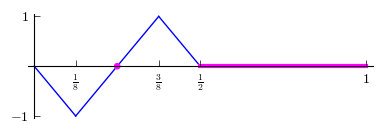}
\end{minipage}
\begin{minipage}{.32\textwidth}
\centering
\includegraphics[width=\linewidth]{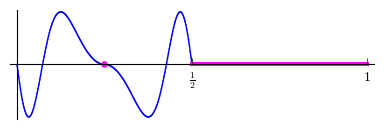}
\end{minipage}
\caption{(\textit{Left}) Finite-dimensional perturbation $\tilde\pi_\T$ of
  $\pi = \sage{equiv7\underscore{}example\underscore{}1()}$ from
  \autoref{ex:equiv7_example_1}. 
 (\textit{Middle--right}) Examples of equivariant perturbations
  $\tilde \pi_{\mathrm{zero}(\T)}$ of $\pi$. \tred{Code that generates these
  figures?}
}
\label{fig:one-sided-discontinuous-perturb}
\end{figure}
For the function in
\autoref{fig:one-sided-discontinuous-pi1}, $\pi = \sage{equiv7\underscore{}example\underscore{}1()}$,
the refined polyhedral complex $\T$ has vertices $B' = \{0, \tfrac14,
\tfrac12, 1\}$. The finite-dimensional perturbation space $\tilde
\Pi^{\pi}_\T\RZ$ has dimension $1$ and is spanned by the basic perturbation
\[
 \tilde\pi_\T(x) = 
  \begin{cases} 
   0 & \text{if } x=0 \\
   x-\tfrac{1}{4} &\text{if } 0< x< \tfrac{1}{2}\\
   0       & \text{if } \tfrac{1}{2} \leq x < 1,
  \end{cases}
\]
see \autoref{fig:one-sided-discontinuous-perturb} (left).
The two intervals $I_1 = (0, \tfrac14)$ and $I_2 = (\tfrac14,
\tfrac12)$ are uncovered, and they are connected through the move $\rho_{f}|_{(0,1/2)}$
in $\Omega$. 
Because there is only one connected uncovered component, 
the equivariant perturbation space $\tildePi^{\pi}_{\textrm{zero}(\T)}$
consists of all Lipschitz continuous functions $\tilde
\pi_{\mathrm{zero}(\T)}$ satisfying that $\tilde \pi_{\mathrm{zero}(\T)}(x) =
0$ for $x \in C \cup B'$ and that $\tilde \pi_{\mathrm{zero}(\T)}(x) = -\tilde
\pi_{\mathrm{zero}(\T)}(f-x)$ for $x \in U'$. See
\autoref{fig:one-sided-discontinuous-perturb} (middle, right) for examples of
such functions.
\end{example}

\begin{example}[\autoref{ex:equiv7_example_xyz_2}, continued]\label{ex:equiv7_example_xyz_2-continued}
  \begin{figure}[tp]
    \centering
    \includegraphics[width=0.8\linewidth]{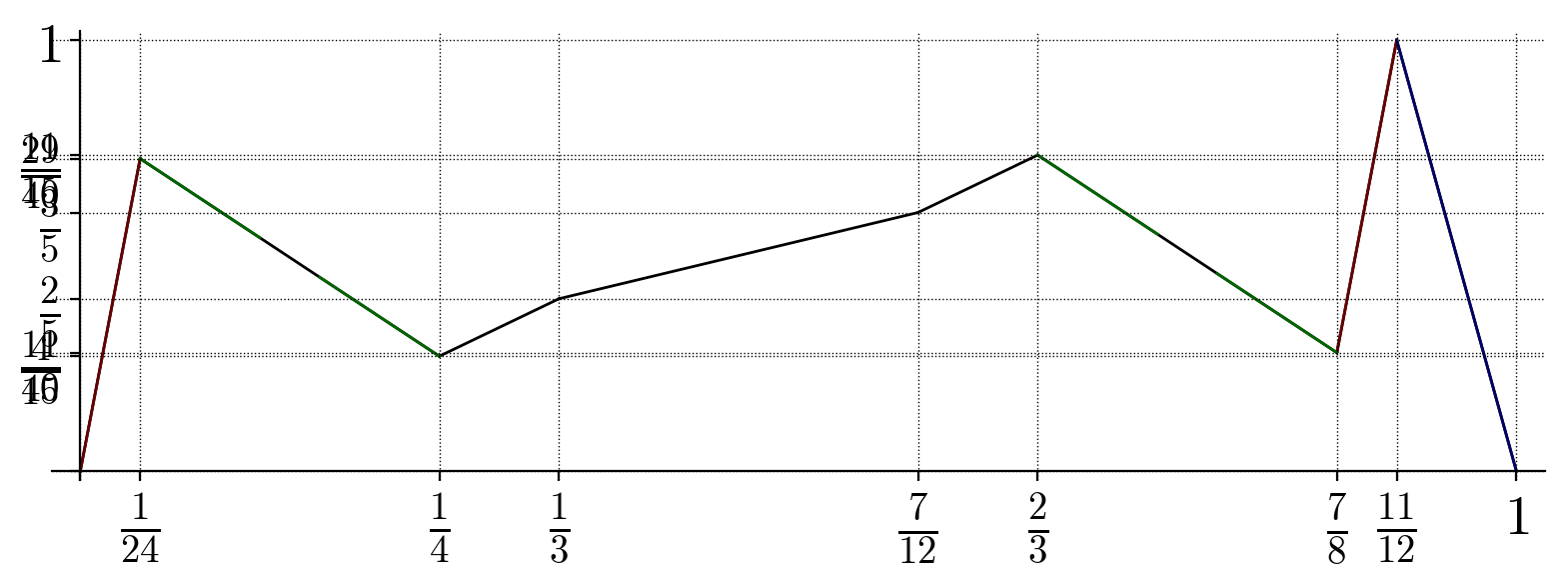}\par
    \includegraphics[width=0.8\linewidth]{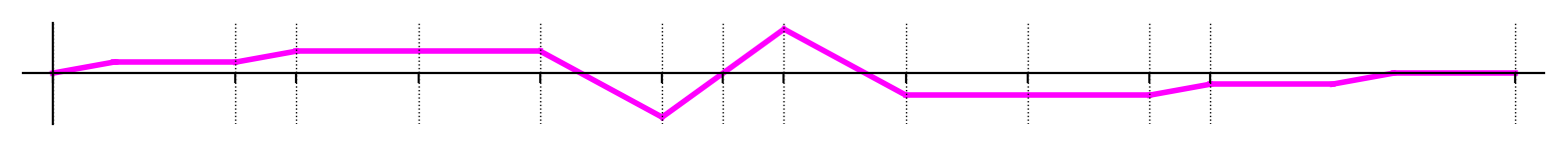}\par
    \includegraphics[width=0.8\linewidth]{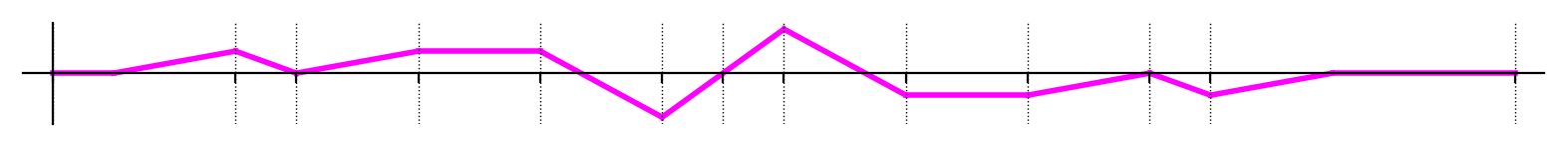}\par
    \includegraphics[width=0.8\linewidth]{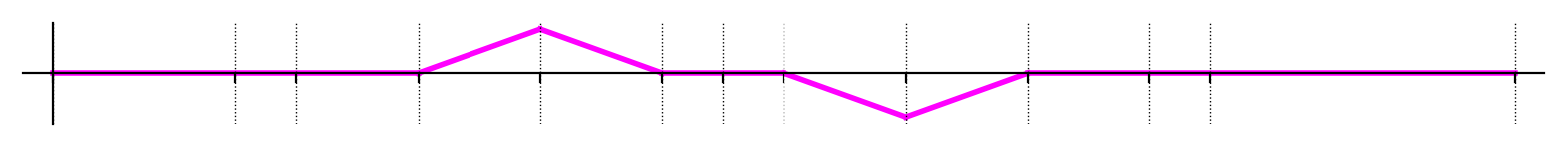}
    \includegraphics[width=0.8\linewidth]{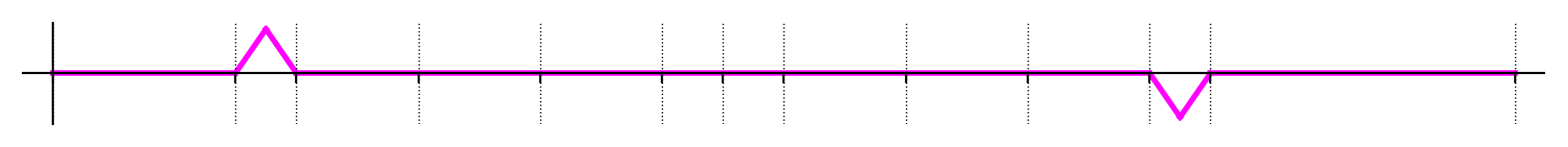}\par
    \includegraphics[width=0.8\linewidth]{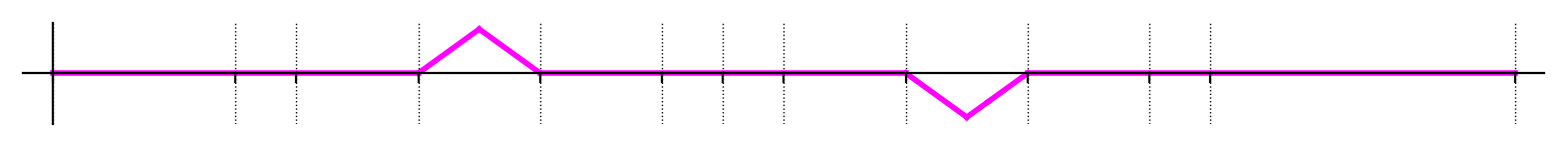}
    \includegraphics[width=0.8\linewidth]{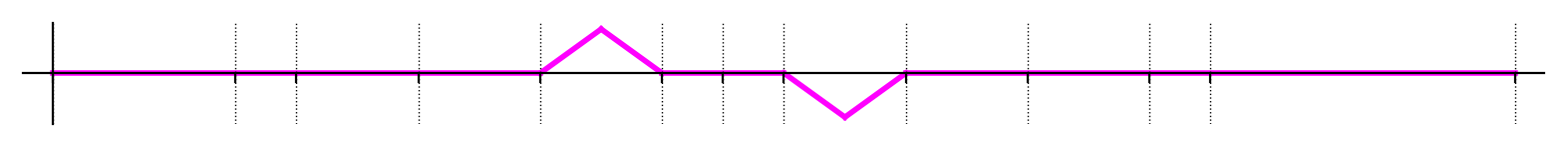}\par
    \includegraphics[width=0.8\linewidth]{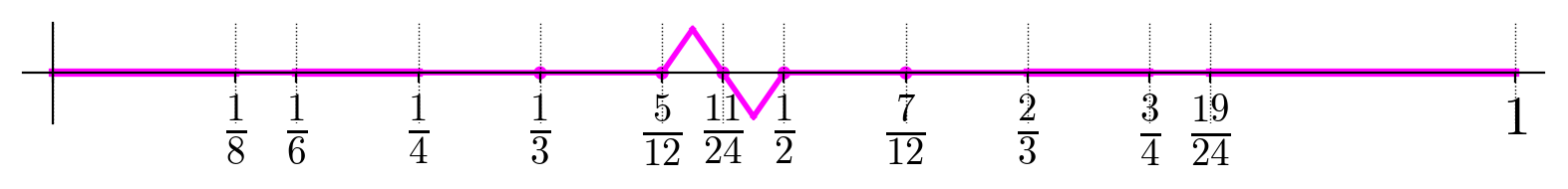}
    \caption{Decomposition of the space of effective perturbations for the function from
      \autoref{ex:equiv7_example_xyz_2}/\ref{ex:equiv7_example_xyz_2-continued}, 
      $\pi = \sage{equiv7\_example\_xyz\_2()}$. 
      (a) The function~$\pi_2$. (b--d) basis of the space~$\tildePi_\T$ of finite-dimensional
      perturbations. (e--h) representatives of the equivariant perturbation
      spaces $\tilde \Pi^\pi_{U_i}$ for the 4 connected uncovered components $U_i$.
    }
    \label{fig:equiv7_example_xyz_2-perturbations}
  \end{figure}
  \autoref{fig:equiv7_example_xyz_2-perturbations} illustrates
  the decomposition of the space of effective perturbations. 
\end{example}

\begin{example}[\autoref{ex:equiv7_minimal_2_covered_2_uncovered}, continued]\label{ex:equiv7_minimal_2_covered_2_uncovered-continued}
  \begin{figure}[tp]
    \centering
    \includegraphics[width=\linewidth]{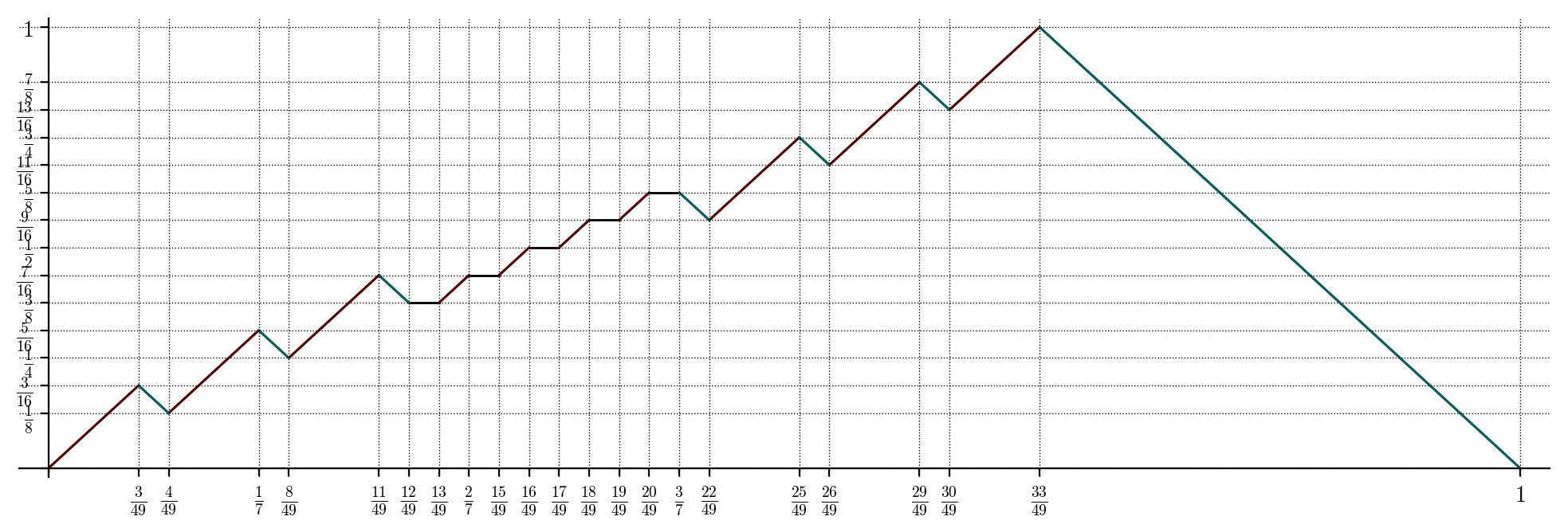}\par
    \includegraphics[width=\linewidth]{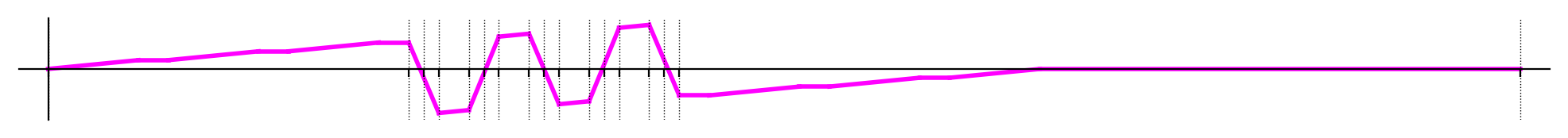}\par
    \includegraphics[width=\linewidth]{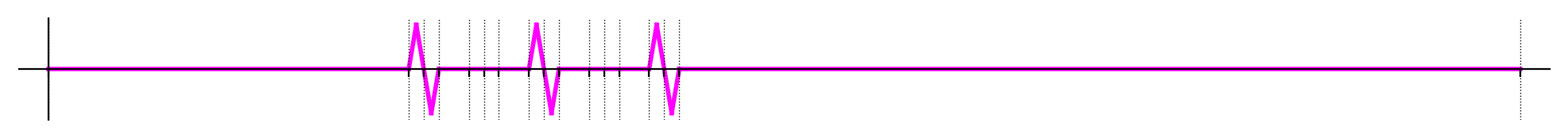}\par
    \includegraphics[width=\linewidth]{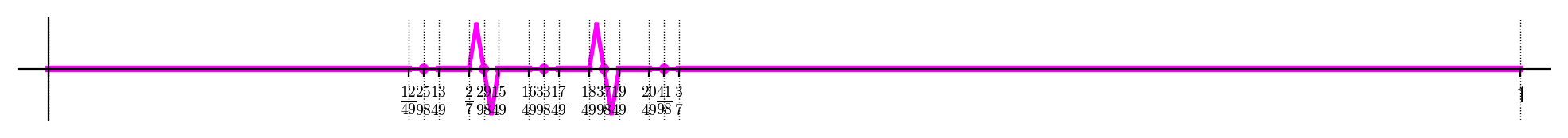}
    \caption{Decomposition of the space of effective perturbations for the function from
      \autoref{ex:equiv7_minimal_2_covered_2_uncovered}/\ref{ex:equiv7_minimal_2_covered_2_uncovered-continued}, 
      $\pi = \sage{equiv7\_minimal\_2\_covered\_2\_uncovered()}$. 
      (a) The function~$\pi$. (b) finite-dimensional perturbation
      $\tilde\pi_\T$. (c), (d) examples of equivariant perturbations
      $\tilde\pi_1, \tilde \pi_2$ from the direct summands. \tred{Use notation
        for spaces here}
    }
    \label{fig:equiv7_minimal_2_covered_2_uncovered-perturbations}
  \end{figure}
  \autoref{fig:equiv7_minimal_2_covered_2_uncovered-perturbations} illustrates
  the decomposition of the space of effective perturbations. 
\end{example}

\subsection{Relation of the moves closure to the semigroup
  $\Omegaresp{\tildePi^{\texorpdfstring{\pi}{}}}$ of respected moves}
\label{s:ctscl-almost-respected-moves}
In this section, still under the assumptions from \autoref{s:assumptions},
we establish the relation between $\ctscl(\Omegainit)$ and two other
move semigroups:
\begin{enumerate}[\rm(a)]
\item the semigroup $\Omegaresp{\tildePi^\pi}$ of moves respected by all
  effective perturbation functions~$\tilde\pi$,
\item the semigroup
  $\Omegaresp{\{\pi\} \cup \tildePi^\pi\RZ} = \Omegaresp{\pi +
    \tildePi^\pi\RZ} $ of moves respected by $\pi$ and its perturbations.
\end{enumerate}

We already know from \autoref{cor:pwl-moves-closure-respected} that
\begin{equation}\label{eq:move-semigroup-resp-inclusions}
  \ctscl(\Omegainit) \subseteq \Omegaresp{\pi +
    \tildePi^\pi\RZ} \subseteq \Omegaresp{\tildePi^\pi\RZ}.
\end{equation}
In the case of an extreme function~$\pi$, the space $\tildePi^\pi\RZ$ of
effective perturbations is trivial; and thus,
$\Omegaresp{\tildePi^\pi} = \FullMoveSemigroup$.

More generally, whenever a function~$\theta$ is affine on intervals
$D_1, D_2, \dots, D_k$ with the same slope, then
$\MovesOfGraph((D_1\cup\dots\cup D_k)\times (D_1\cup\dots\cup D_k)) \subseteq \Omegaresp{\theta}$.  Thus, we have
the following:
\begin{lemma}\label{lemma:posterior-moves-for-trivial-finitedim}
  Suppose the space $\tildePi^{\pi}_{\T}\RZ$ of finite-dimensional
  perturbations is trivial.
  \begin{enumerate}[\rm(a)]
  \item Let $C$ be the set of covered points. Then $\MovesOfGraph(C\times C)
    \subseteq \Omegaresp{\tildePi^\pi\RZ}$.  
  \item Let $D_1, \dots, D_k \subseteq C$ be covered intervals on which
    $\pi$~is affine with the same slope. Then $\MovesOfGraph((D_1\cup\dots\cup D_k)\times
    (D_1\cup\dots\cup D_k)) \subseteq \Omegaresp{\pi + \tildePi^\pi\RZ}$.
  \end{enumerate}
\end{lemma}

\begin{example}\label{ex:equiv7_example_post_3}
  \begin{figure}[tp]
    \centering
    \includegraphics[width=\linewidth]{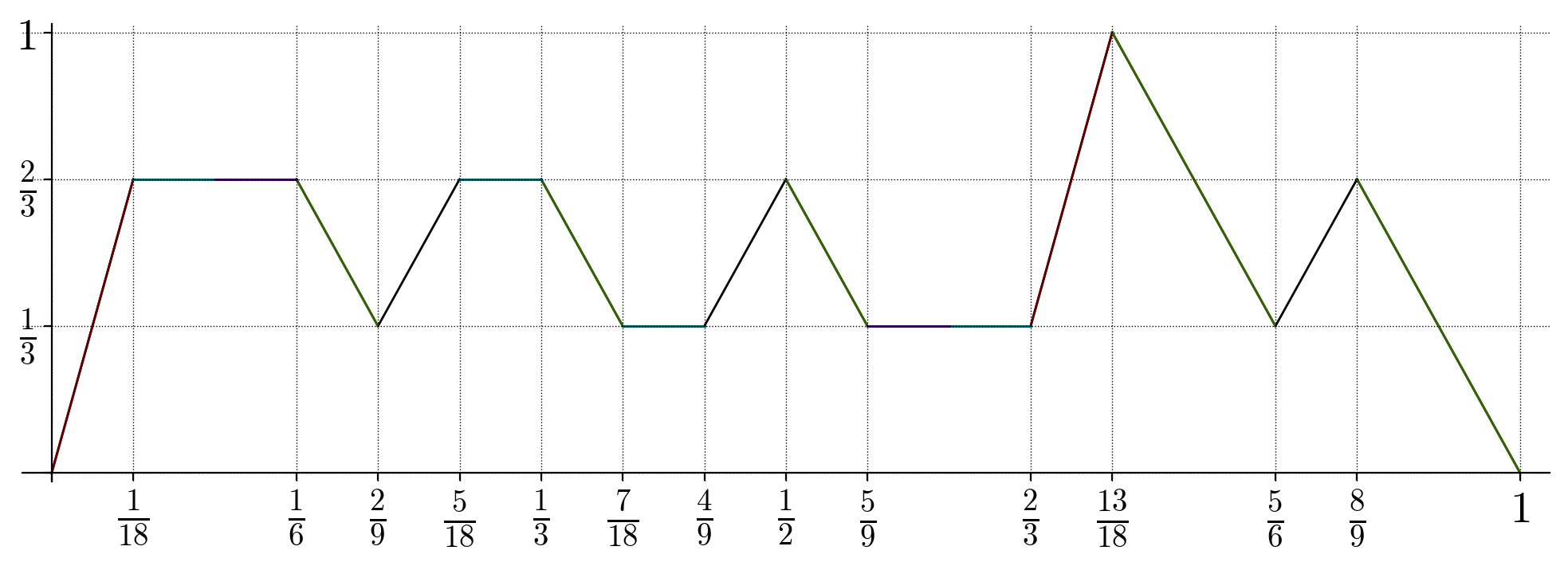}\par
    \caption{Function 
      $\pi = \sage{equiv7\_example\_post\_3()}$ from
      \autoref{ex:equiv7_example_post_3}.
    }
    \label{fig:equiv7_example_post_3-perturbations}
  \end{figure}
  Consider the function
  $\pi = \sage{equiv7\_example\_post\_3()}$, shown  in
  \autoref{fig:equiv7_example_post_3-perturbations}.  It has 4 connected
  covered components (colored slopes in the figure) and 2 connected uncovered
  components $U_1 = (\frac29, \frac5{18}) \cup (\frac49, \frac12)$ and $U_2 =
  (\frac56, \frac{31}{36}) \cup (\frac{31}{36}, \frac89)$. 
  Its finite-dimensional perturbation space is trivial.
  \begin{enumerate}[\rm(a)]
  \item From~\autoref{lemma:posterior-moves-for-trivial-finitedim}\,(a) we see
    that $\MovesOfGraph(C\times C) \subseteq \Omegaresp{\tildePi^\pi\RZ}$.  
  \item For the smaller semigroup $\Omegaresp{\pi + \tildePi^\pi\RZ}$, we
    observe that the function $\pi$ is affine with slope $0$ on the intervals
    $D_1 = (\frac1{18}, \frac2{18})$ and $D_2 = (\frac2{18}, \frac3{18})$,
    which belong to separate connected covered components (\emph{cyan} and
    \emph{lavender}).  Because the finite-dimensional perturbation space is
    trivial, all functions in $\pi + \tildePi^\pi\RZ$ take the same slope on
    $D_1$ and $D_2$, and hence from \autoref{lemma:posterior-moves-for-trivial-finitedim}\,(b) we have
    $\MovesOfGraph((D_1\cup D_2)\times (D_1\cup D_2)) \subseteq \Omegaresp{\pi +
      \tildePi^\pi\RZ}$.
    By continuity, we also have
    $\MovesOfGraph((\frac1{18}, \frac3{18})\times (\frac1{18}, \frac3{18}))
    \subseteq \Omegaresp{\pi + 
      \tildePi^\pi\RZ}$.
  \end{enumerate}
\end{example}
\begin{remark}
  Suppose the finite-dimensional perturbation space has a positive dimension.
  Recall its description using slope variables $\tilde s_i^c$ (for the
  connected covered components $C_i$) and $\tilde s_i^u$ (for the connected
  uncovered components $U_i$) from
  \autoref{rem:finite-dim-system-with-slope-variables}.  Whenever for some
  $i$, $j$, we have that $\tilde s_i^c = \tilde s_j^c$ holds for all
  solutions, then
  $\MovesOfGraph((C_i\cup C_j)\times (C_i\cup C_j)) \subseteq \Omegaresp{\tildePi^\pi\RZ}$.  
  A similar
  statement holds for $\Omegaresp{\pi + \tildePi^\pi\RZ}$.
\end{remark}

Consider these move ensembles restricted to the set $U$ of uncovered points in $(0,1)$.
We have the following theorem.

\begin{theorem}
\label{thm:ctscl-almost-respected-moves}
Under Assumptions~\ref{assumption:one-sided-continuous},
\ref{assumption:minimal-B}, and~\ref{assumption:moves-completion}, we have
that 
\begin{equation}
  \label{eq:ctscl-almost-respected-moves}
    \clsemi{\Omegainit}|_U = \Omegaresp{\pi +
    \tildePi^\pi\RZ}|_U = \Omegaresp{\tildePi^\pi\RZ}|_U.
\end{equation}
where $U$ is the set of uncovered points in $(0,1)$.
\end{theorem}
\begin{proof}
  We use the notations of the present section. By \eqref{eq:move-semigroup-resp-inclusions}, it suffices
  to show that if the domain of a move $\gamma|_D\in
  \Omegaresp{\tildePi^\pi}$ is contained in~$U$, then $\gamma|_D \in \clsemi{\Omegainit}$.

  Recall that we can write an arbitrary connected uncovered component
  $\Uncomp[i]$ in the form of $\Uncomp[i] = \bigcup_{j=1}^p \gamma_j(I)$,
  where $I$ is the fundamental domain for $\Uncomp[i]$, $\gamma_j|_{I} \in \clsemi{\Omegainit}$,
  and the open intervals $\gamma_j(I)$ are disjoint.
As $\clsemi{\Omegainit}$ is join-closed and extension-closed,
by taking sub-moves, it suffices to show that 
if a move $\gamma|_D$ satisfies that $D \subseteq I$ and the unrestricted move $\gamma \neq \gamma_j$ for all $j=1,\dots, p$, then
$\gamma|_D\not\in \Omegaresp{\tildePi^\pi}$.

Consider a move $\gamma|_D$ where  $D \subseteq I$ and $\gamma \neq \gamma_j$ for all $j=1,\dots, p$. There exists an open interval $D' \subseteq D$ such that $\gamma(D') \cap \gamma_j(D') = \emptyset$ for all  $j=1,\dots, p$. We can construct a perturbation $\tilde{\pi}$ such that 
\begin{enumerate}[(i)]
\item $\tilde\pi$ is non-zero and Lipschitz continuous on $D'$;
\item  $\tilde\pi(x)=\tilde\pi(x^-)=\tilde\pi(x^+)=0$ for $x \in \partial D'$;
\item  $\tilde\pi(x) = \chi(\gamma_j) \tilde\pi(\gamma_j (x))$ for $x \in D'$, $j=1, \dots, p$; and 
\item $\tilde\pi(x)=0$ for $x \not\in \bigcup_{j=1}^p \gamma_j(D')$. 
\end{enumerate}
Since $\tilde\pi|_{D'} \not\equiv 0$ but $\tilde\pi|_{\gamma(D')} \equiv 0$,
we have that $\gamma|_{D'} \not\in \Omegaresp{\tilde{\pi}}$, and hence $\gamma|_{D} \not\in \Omegaresp{\tilde{\pi}}$.
By \autoref{thm:equiv-perturbation-characterization}, $\tilde{\pi} \in \tildePi^{\pi}_{\mathrm{zero}(\T)} \subseteq  \tildePi^{\pi}$.
Therefore, $\gamma|_D\not\in \Omegaresp{\tildePi^\pi}$.
We conclude that \eqref{eq:ctscl-almost-respected-moves} holds.
\end{proof}

\bigbreak

\section{Conclusion}
\label{s:conclusion}

\subsection{Forthcoming computational companion paper}
In the forthcoming paper \cite{hildebrand-koeppe-zhou:algocomp-paper},
\ifserialtitle part VIII of the series, \fi
we will describe a method to compute the moves
closure~$\ctscl(\Omegainit)$ for a large class of piecewise linear minimal valid
functions, including all functions with rational breakpoints, for which the
moves closure has a finite presentation.  The decomposition of the
perturbation space in \autoref{sec:perturbation_space} is already
algorithmic. Thus we will obtain a grid-free extremality test.

\subsection{Limits of the approach of this paper}
We now discuss the limitations to the equivariant perturbation theory\ifserialtitle
developed in our series of papers\fi.

For two-sided discontinuous functions, the decomposition of the perturbation
spaces breaks down.  
\autoref{thm:equiv-perturbation-characterization} and \autoref{thm:decomposition-perturbation} do not hold when the function $\pi$ is discontinuous from both sides of the origin, as the following example shows.
\begin{example}\label{ex:minimal_no_covered_interval}
\begin{figure}[tp]
\centering
\begin{minipage}{.49\textwidth}
\centering
\includegraphics[width=.8\linewidth]{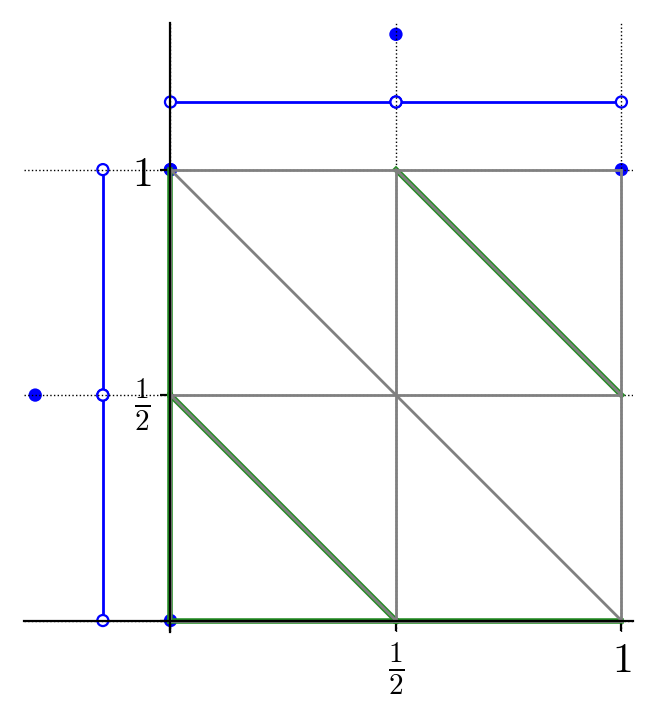}
\end{minipage}
\begin{minipage}{.49\textwidth}
\centering
\includegraphics[width=.8\linewidth]{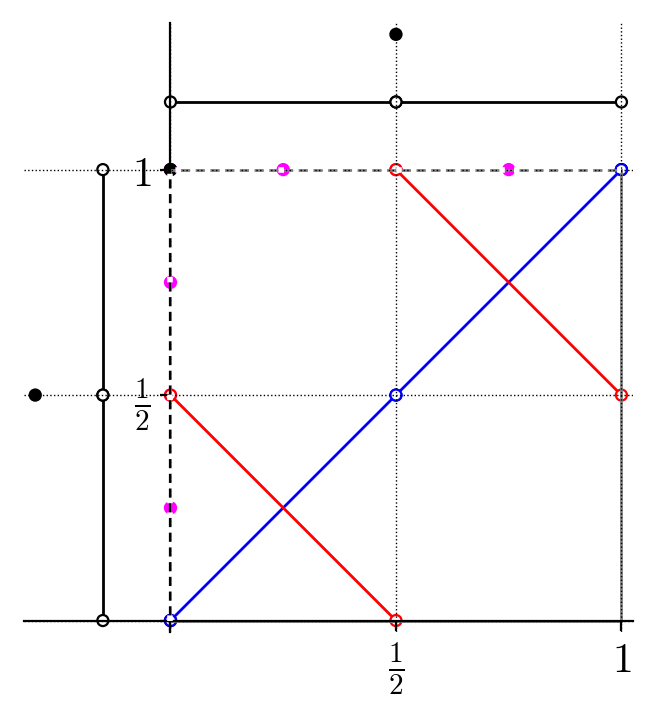}
\end{minipage}
\caption[(\textit{Left}) Two-dimensional polyhedral complex $\Delta\P$ of a
  two-sided discontinuous minimal valid function $\pi =
  \sage{minimal\underscore{}no\underscore{}covered\underscore{}interval()}$
  from \autoref{ex:minimal_no_covered_interval}. (\textit{Right}) The graph of
  the move ensemble $\clsemi{\Omegainit}$ of $\pi$]
  {(\textit{Left}) Two-dimensional polyhedral complex $\Delta\P$ of a
  two-sided discontinuous minimal valid function $\pi =
  \sage{minimal\underscore{}no\underscore{}covered\underscore{}interval()}$
  (\emph{blue graph at the left and top borders}) from \autoref{ex:minimal_no_covered_interval}, where the additive faces are colored in \emph{green}. (\textit{Right}) The graph of the move ensemble $\clsemi{\Omegainit}$ of $\pi$, where the set $C \cup B' = [0,1) \setminus U'$ of covered points and breakpoints  are marked in \emph{magenta on the left and top borders}.}
\label{fig:two-sided-discontinuous-pi2}
\end{figure}

Consider the minimal valid function $\pi=$ \sage{minimal\underscore{}no\underscore{}covered\underscore{}interval()}  with $f=\tfrac{1}{2}$, defined by
\[
 \pi(x) = 
  \begin{cases} 
   0 & \text{if } x=0 \\
   \tfrac{1}{2} &\text{if } 0< x< \tfrac{1}{2} \text{ or } \tfrac{1}{2} < x < 1\\
   1      & \text{if } x=\tfrac12,
  \end{cases}
\]
which is discontinuous from both sides of the origin.

Observe from \autoref{fig:two-sided-discontinuous-pi2} 
that $C = \emptyset$, $B' = \{0, \tfrac14, \tfrac12, \tfrac34, 1\}$ and the
connected uncovered components are $\Uncomp[1] = (0, \tfrac14) \cup
(\tfrac14, \tfrac12)$ and $\Uncomp[2] = (\tfrac12, \tfrac34) \cup  (\tfrac34,
1)$, where the two intervals in either $\Uncomp[1]$ or $\Uncomp[2]$ are
connected through the move $\rho_{f}|_{(0,f)}$ or $\rho_{f}|_{(f,1)}$ in
$\clsemi{\Omegainit}$. Any bounded $\Z$-periodic 
function $\tilde\pi$ satisfying that $\tilde\pi(x)=0$ for $x \in \B'$ and
$\tilde\pi(x) = \tilde\pi(\rho_{f}(x))$ for $x \in [0,1)$ is an effective
perturbation of $\pi$.  
For example, define a $\Z$-periodic function $\tilde\pi =$ \sage{equiv7\underscore{}example\underscore{}2\underscore{}crazy\underscore{}perturbation()} by
\[
 \tilde\pi(x)= 
      \begin{cases} 
        1& \text{if } x \in (0, \frac{1}{4}) \text{ such that } x  \in G\text{, or} \\
        & \text{if } x \in (\frac{1}{4}, \frac{1}{2}) \text{ such that } x - \frac{1}{4}\in G; \\
        -1 & \text{if } x \in (0, \frac{1}{4})\text{ such that } x - \frac{1}{4}\in G\text{, or} \\
        & \text{if } x \in (\frac{1}{4}, \frac{1}{2}) \text{ such that } x - \frac{1}{2} \in G; \\
        0 & \text{otherwise},
      \end{cases}
\] 
where the group $G =\langle 1, \sqrt{2}\rangle_{\Z}$ is dense in $\R$.
Then $\tilde{\pi}$ is an effective perturbation of $\pi$, since both $\pi
\pm \epsilon\tilde{\pi}$ are minimal valid functions for $0< \epsilon \leq
\frac{1}{6}$.\footnote{This positive $\epsilon$ is verified by calling
  \sage{find\underscore{}epsilon\underscore{}for\underscore{}crazy\underscore{}perturbation($\pi$,
    $\tilde{\pi}$)}} Observe that $\tilde{\pi}$ is a highly discontinuous function, which does not have a limit at any point in $(0, \frac{1}{2})$.
Thus, without \autoref{assumption:one-sided-continuous}, an equivariant perturbation is not necessarily Lipschitz continuous; and the limits of an effective perturbation at the breakpoints might not exist. For this reason, the decomposition of perturbations does not make sense when the function $\pi$ is discontinuous from both sides of the origin.
\end{example}

Note that in
\cite{koeppe-zhou:crazy-perturbation}, though an algorithm was presented that
checks the effectiveness of a given perturbation function $\tilde\pi$, and a
perturbation was constructed for an example function, it was left as an open
question how to construct perturbations in general.  This is still open.

We conjecture that the equivariant perturbation theory also breaks down for
the case of non--piecewise linear functions, such as the fractal functions
presented in \cite{bccz08222222} and \cite{bcdsp:arbitrary-slopes}.  In
particular we note that (1) the finite system of equations describing the
space of finite-dimensional perturbations would be replaced by a system of
functional equations, for which we have no lemmas available; (2) we expect
that the decomposition theorem no longer holds.

\clearpage
\bookmarksetup{italic}
\section*{Acknowledgments}
The authors wish to thank Chun Yu Hong, who worked on a first implementation
of a grid-free procedure as an undergraduate researcher at UC Davis in 2013;
experiments based on his code have helped the
authors to develop the full theory described in the present paper.
The authors also wish to thank Quentin Louveaux and Reuben La Haye for
valuable discussions.

\bibliographystyle{../amsabbrvurl}
\bibliography{../bib/open-optimization-bibliography/cutgeneratingfunctionology,../bib/MLFCB-minus-oo-cgf}

\providecommand\CheckAccent[1]{\accent20 #1}
\providecommand{\bysame}{\leavevmode\hbox to3em{\hrulefill}\thinspace}
\providecommand\ISBN{ISBN }
\providecommand{\MR}{\relax\ifhmode\unskip\space\fi MR }
\providecommand{\MRhref}[2]{%
  \href{http://www.ams.org/mathscinet-getitem?mr=#1}{#2}
}
\providecommand{\href}[2]{#2}
\begin{thebibliography}{10}

\bibitem{bccz08222222}
A.~Basu, M.~Conforti, G.~Cornu{\'e}jols, and G.~Zambelli, \emph{A
  counterexample to a conjecture of {G}omory and {J}ohnson}, Mathematical
  Programming Ser. A \textbf{133} (2012), no.~1--2, 25--38, \href
  {https://doi.org/10.1007/s10107-010-0407-1}
  {\path{https://doi.org/10.1007/s10107-010-0407-1}}.

\bibitem{bcdsp:arbitrary-slopes}
A.~Basu, M.~Conforti, M.~Di~Summa, and J.~Paat, \emph{Extreme functions with an
  arbitrary number of slopes}, Integer Programming and Combinatorial
  Optimization: 18th International Conference, IPCO 2016, Li{\`e}ge, Belgium,
  June 1--3, 2016, Proceedings (Q.~Louveaux and M.~Skutella, eds.), Springer
  International Publishing, Cham, 2016, pp.~190--201, \href
  {https://doi.org/10.1007/978-3-319-33461-5_16}
  {\path{https://doi.org/10.1007/978-3-319-33461-5_16}},
  \ISBN{978-3-319-33461-5}.

\bibitem{basu-hildebrand-koeppe:equivariant}
A.~Basu, R.~Hildebrand, and M.~K{\"o}ppe, \emph{Equivariant perturbation in
  {G}omory and {J}ohnson's infinite group problem. {I}. {T}he one-dimensional
  case}, Mathematics of Operations Research \textbf{40} (2014), no.~1,
  105--129, \href {https://doi.org/10.1287/moor.2014.0660}
  {\path{https://doi.org/10.1287/moor.2014.0660}}.

\bibitem{basu-hildebrand-koeppe:equivariant-general-2dim}
\bysame, \emph{Equivariant perturbation in {G}omory and {J}ohnson's infinite
  group problem. {IV}. {T}he general unimodular two-dimensional case},
  Manuscript, 2016.

\bibitem{igp_survey}
\bysame, \emph{Light on the infinite group relaxation {I}: foundations and
  taxonomy}, 4OR \textbf{14} (2016), no.~1, 1--40, \href
  {https://doi.org/10.1007/s10288-015-0292-9}
  {\path{https://doi.org/10.1007/s10288-015-0292-9}}.

\bibitem{igp_survey_part_2}
\bysame, \emph{Light on the infinite group relaxation {II}: sufficient
  conditions for extremality, sequences, and algorithms}, 4OR \textbf{14}
  (2016), no.~2, 107--131, \href {https://doi.org/10.1007/s10288-015-0293-8}
  {\path{https://doi.org/10.1007/s10288-015-0293-8}}.

\bibitem{bhk-IPCOext}
\bysame, \emph{Equivariant perturbation in {G}omory and {J}ohnson's infinite
  group problem---{III}: Foundations for the {$k$}-dimensional case with
  applications to {$k=2$}}, Mathematical Programming \textbf{163} (2017),
  no.~1, 301--358, \href {https://doi.org/10.1007/s10107-016-1064-9}
  {\path{https://doi.org/10.1007/s10107-016-1064-9}}.

\bibitem{dey1}
S.~S. Dey, J.-P.~P. Richard, Y.~Li, and L.~A. Miller, \emph{On the extreme
  inequalities of infinite group problems}, Mathematical Programming
  \textbf{121} (2010), no.~1, 145--170, \href
  {https://doi.org/10.1007/s10107-008-0229-6}
  {\path{https://doi.org/10.1007/s10107-008-0229-6}}.

\bibitem{DiSumma-2018:piecewise-smooth-piecewise-linear}
M.~Di~Summa, \emph{Piecewise smooth extreme functions are piecewise linear},
  Mathematical Programming (2018), 1--29, \href
  {https://doi.org/10.1007/s10107-018-1330-0}
  {\path{https://doi.org/10.1007/s10107-018-1330-0}}.

\bibitem{gom}
R.~E. Gomory, \emph{Some polyhedra related to combinatorial problems}, Linear
  Algebra and its Applications \textbf{2} (1969), 451--558, \href
  {https://doi.org/10.1016/0024-3795(69)90017-2}
  {\path{https://doi.org/10.1016/0024-3795(69)90017-2}}.

\bibitem{infinite}
R.~E. Gomory and E.~L. Johnson, \emph{Some continuous functions related to
  corner polyhedra, {I}}, Mathematical Programming \textbf{3} (1972), 23--85,
  \href {https://doi.org/10.1007/BF01584976}
  {\path{https://doi.org/10.1007/BF01584976}}.

\bibitem{infinite2}
\bysame, \emph{Some continuous functions related to corner polyhedra, {II}},
  Mathematical Programming \textbf{3} (1972), 359--389, \href
  {https://doi.org/10.1007/BF01585008}
  {\path{https://doi.org/10.1007/BF01585008}}.

\bibitem{hildebrand:group-subdivisions}
R.~Hildebrand, \emph{On polyhedral subdivisions closed under group operations},
  Manuscript, 2013.

\bibitem{hildebrand-koeppe-zhou:algocomp-paper}
R.~Hildebrand, M.~K{\"o}ppe, and Y.~Zhou, \emph{Equivariant perturbation in
  {G}omory and {J}ohnson's infinite group problem. {VIII}. {G}rid-free
  extremality test---general algorithm and implementation}, manuscript, 2019.

\bibitem{hildebrand-koeppe-zhou:algo-paper-abstract-ipco}
\bysame, \emph{On perturbation spaces of minimal valid functions: Inverse
  semigroup theory and equivariant decomposition theorem}, Integer Programming
  and Combinatorial Optimization. IPCO 2019 (A.~Lodi and V.~Nagarajan, eds.),
  Lecture Notes in Computer Science, vol. 11480, Springer, Cham, 2019,
  pp.~247--260, \href {https://doi.org/10.1007/978-3-030-17953-3_19}
  {\path{https://doi.org/10.1007/978-3-030-17953-3_19}},
  \ISBN{978-3-030-17952-6}.

\bibitem{hong-koeppe-zhou:software-paper}
C.~Y. Hong, M.~K{\"o}ppe, and Y.~Zhou, \emph{Equivariant perturbation in
  {G}omory and {J}ohnson's infinite group problem {(V)}. {S}oftware for the
  continuous and discontinuous 1-row case}, Optimization Methods and Software
  \textbf{33} (2018), no.~3, 475--498, \href
  {https://doi.org/10.1080/10556788.2017.1366486}
  {\path{https://doi.org/10.1080/10556788.2017.1366486}}.

\bibitem{zhou:extreme-notes}
M.~K{\"o}ppe and Y.~Zhou, \emph{An electronic compendium of extreme functions
  for the {G}omory--{J}ohnson infinite group problem}, Operations Research
  Letters \textbf{43} (2015), no.~4, 438--444, \href
  {https://doi.org/10.1016/j.orl.2015.06.004}
  {\path{https://doi.org/10.1016/j.orl.2015.06.004}}.

\bibitem{koeppe-zhou:param-abstract-isco}
\bysame, \emph{Toward computer-assisted discovery and automated proofs of
  cutting plane theorems}, Combinatorial Optimization: 4th International
  Symposium, ISCO 2016, Vietri sul Mare, Italy, May 16--18, 2016, Revised
  Selected Papers (R.~Cerulli, S.~Fujishige, and A.~R. Mahjoub, eds.), Springer
  International Publishing, Cham, 2016, pp.~332--344, \href
  {https://doi.org/10.1007/978-3-319-45587-7_29}
  {\path{https://doi.org/10.1007/978-3-319-45587-7_29}},
  \ISBN{978-3-319-45587-7}.

\bibitem{koeppe-zhou:discontinuous-facets}
\bysame, \emph{On the notions of facets, weak facets, and extreme functions of
  the {G}omory--{J}ohnson infinite group problem}, Integer Programming and
  Combinatorial Optimization: 19th International Conference, IPCO 2017,
  Waterloo, ON, Canada, June 26--28, 2017, Proceedings (F.~Eisenbrand and
  J.~Koenemann, eds.), Springer International Publishing, Cham, 2017,
  pp.~330--342, \href {https://doi.org/10.1007/978-3-319-59250-3_27}
  {\path{https://doi.org/10.1007/978-3-319-59250-3_27}},
  \ISBN{978-3-319-59250-3}.

\bibitem{koeppe-zhou:crazy-perturbation}
\bysame, \emph{Equivariant perturbation in {G}omory and {J}ohnson's infinite
  group problem. {VI}. {T}he curious case of two-sided discontinuous minimal
  valid functions}, Discrete Optimization \textbf{30} (2018), 51--72, \href
  {https://doi.org/10.1016/j.disopt.2018.05.003}
  {\path{https://doi.org/10.1016/j.disopt.2018.05.003}}.

\bibitem{cutgeneratingfunctionology:lastest-release}
M.~K{\"o}ppe, Y.~Zhou, C.~Y. Hong, and J.~Wang,
  \emph{{cutgeneratingfunctionology}: {Python} code for computation and
  experimentation with cut-generating functions},
  \url{https://github.com/mkoeppe/cutgeneratingfunctionology}, November 2019,
  (Version 1.4.1).

\bibitem{lawson-1998-inverse-semigroups}
M.~V. Lawson, \emph{Inverse semigroups: The theory of partial symmetries},
  World Scientific, 1998.

\bibitem{zhou:dissertation}
Y.~Zhou, \emph{Infinite-dimensional relaxations of mixed-integer optimization
  problems}, Ph.D. thesis, University of California, Davis, Graduate Group in
  Applied Mathematics, May 2017, available from
  \url{https://search.proquest.com/docview/1950269648}.

\end{thebibliography}

\end{document}